\pdfoutput=1
\RequirePackage{ifpdf}
\ifpdf 
\documentclass[pdftex]{sigma}
\else
\documentclass{sigma}
\fi

\numberwithin{equation}{section}

\usepackage{amscd}
\usepackage{braket, dsfont}
\usepackage{cancel}
\usepackage{mathrsfs}
\usepackage{tikz-cd}

\usepackage[all]{xy}
\usepackage{stmaryrd}
\usepackage[mathscr]{euscript}
\usepackage{tikz}

\usepackage{etoolbox} 

\newcommand{\Proj}{\operatorname{Proj}}
\newcommand{\smk}{{\fontfamily{pzc}\selectfont \text{k}}}
\newcommand{\tto}{\xrightarrow}
\newcommand{\chr}{{\mathsf{c}}}
\newcommand{\gm}{{\mathsf{g}}}
\newcommand{\ptr}{\operatorname{ptr}}

\makeatletter
\let\c@proposition\c@theorem
\let\c@lemma\c@theorem
\let\c@corollary\c@theorem
\let\c@definition\c@theorem
\let\c@remark\c@theorem
\let\c@example\c@theorem

\makeatother
\numberwithin{theorem}{section}
\numberwithin{proposition}{section}
\numberwithin{definition}{section}
\numberwithin{lemma}{section}
\numberwithin{corollary}{section}
\numberwithin{remark}{section}
\numberwithin{example}{section}

\newcommand{\rcoev}{\stackrel{\longrightarrow}{\operatorname{coev}}}
\newcommand{\rev}{\stackrel{\longrightarrow}{\operatorname{ev}}\!\!}
\newcommand{\lev}{\stackrel{\longleftarrow}{\operatorname{ev}}\!\!}
\newcommand{\lcoev}{\stackrel{\longleftarrow}{\operatorname{coev}}}

\newcommand{\bp}[1]{{\left(#1\right)}}

\newcommand{\ci}{\tilde\imath}
\newcommand{\cat}{\mathcal{C}}
\newcommand{\ttt}{\mathcal{T}}
\newcommand{\coend}{\mathscr{L}}

\newcommand{\catd}{\mathcal{D}}
\newcommand{\C}{\ensuremath{\mathbb{C}}}
\newcommand{\Z}{\ensuremath{\mathbb{Z}}}
\newcommand{\R}{\ensuremath{\mathbb{R}}}
\newcommand{\N}{\ensuremath{\mathbb{N}}}

\newcommand{\slt}{\ensuremath{\mathfrak{sl}(2)}}

\newcommand{\End}{\operatorname{End}}
\newcommand{\Hom}{\operatorname{Hom}}

\newcommand{\bbone}{\text{\usefont{U}{bbold}{m}{n}1}}
\MakeRobust{\bbone}
\newcommand{\unit}{\ensuremath{\mathds{1}}}
\newcommand{\Id}{\operatorname{Id}}
\newcommand{\qdim}{\operatorname{qdim}}

\newcommand{\FK}{{\Bbbk}}

\newcommand{\ve}{\varepsilon}
\newcommand{\wt}{\widetilde}

\newcommand{\ms}[1]{\mbox{\tiny\ensuremath{#1}}}

\newcommand{\kk}{\Bbbk}

\newcommand{\CP}{\mathbb{CP}}

\newcommand{\mt}{{\operatorname{\mathsf{t}}}}

\newcommand{\qt}{\operatorname{\mathsf{t}}}

\newcommand{\CC}{\mathcal{C}}
\newcommand{\EE}{\mathcal{E}}
\newcommand{\LL}{\mathcal{L}}

\newcommand{\TSkein}{{\mathscr{S}}}

\newcommand{\Emb}{\mathrm{Emb}}
\newcommand{\Vect}{\mathrm{Vec}}

\newcommand{\cob}{\operatorname{\textbf{Cob}}}
\newcommand{\nc}[1]{{#1^{\mathsf{nc}}}}
\newcommand{\man}{\operatorname{\textbf{Man}}}
\newcommand{\Sp}{\mathbb{S}}
\newcommand{\ThreeManInv}{\mathrm{B}}
\newcommand{\BB}{\ThreeManInv}

\newcommand{\pathtoFig}{}

\usepackage[calc]{picture}\usepackage{settobox}
\newlength{\posx}
\newlength{\posy}
\newlength{\Imagew}
\newlength{\Imageh}
\newlength{\Textw}
\newlength{\Texth}
\newsavebox{\Image}
\newsavebox{\Text}
\newcommand{\epsh}[2]{
  \savebox{\Image}{\epsfig{figure=\pathtoFig#1,height=#2}}
  \settoboxwidth{\Imagew}{\Image}
  \settoboxheight{\Imageh}{\Image}
  \begin{array}{c} \hspace{-1.3mm}
    \raisebox{-4pt}{\usebox{\Image}}
    \hspace{-1.9mm}\end{array}
}
\newcommand{\epsw}[2]{
  \savebox{\Image}{\epsfig{figure=\pathtoFig#1,width=#2}}
  \settoboxwidth{\Imagew}{\Image}
  \settoboxheight{\Imageh}{\Image}
  \begin{array}{c} \hspace{-1.3mm}
    \raisebox{-4pt}{\usebox{\Image}}
    \hspace{-1.9mm}\end{array}
}
\newcommand{\putl}[3]{
  \savebox{\Text}{#3}
  \settoboxwidth{\Textw}{\Text}
  \settoboxheight{\Texth}{\Text}
  \setlength\posx{-\Imagew+0.4pt+0.01\Imagew*\real{#1}}
  \setlength\posy{-0.5\Imageh+2.5pt+0.01\Imageh*\real{#2}}
  \put(\posx,\posy){#3}
}
\newcommand{\putlc}[3]{
  \savebox{\Text}{#3}
  \settoboxwidth{\Textw}{\Text}
  \settoboxheight{\Texth}{\Text}
  \setlength\posx{-\Imagew+0.4pt+0.01\Imagew*\real{#1}}
  \setlength\posy{-0.5\Imageh+2.5pt+0.01\Imageh*\real{#2}-0.5\Texth}
  \put(\posx,\posy){#3}
}
\newcommand{\putr}[3]{
  \savebox{\Text}{#3}
  \settoboxwidth{\Textw}{\Text}
  \settoboxheight{\Texth}{\Text}
  \setlength\posx{-\Imagew+0.4pt+0.01\Imagew*\real{#1}-\Textw}
  \setlength\posy{-0.5\Imageh+2.5pt+0.01\Imageh*\real{#2}}
  \put(\posx,\posy){#3}
}
\newcommand{\putrc}[3]{
  \savebox{\Text}{#3}
  \settoboxwidth{\Textw}{\Text}
  \settoboxheight{\Texth}{\Text}
  \setlength\posx{-\Imagew+0.4pt+0.01\Imagew*\real{#1}-\Textw}
  \setlength\posy{-0.5\Imageh+2.5pt+0.01\Imageh*\real{#2}-0.5\Texth}
  \put(\posx,\posy){#3}
}
\newcommand{\putc}[3]{
  \savebox{\Text}{#3}
  \settoboxwidth{\Textw}{\Text}
  \settoboxheight{\Texth}{\Text}
  \setlength\posx{-\Imagew+0.4pt+0.01\Imagew*\real{#1}-0.5\Textw}
  \setlength\posy{-0.5\Imageh+2.5pt+0.01\Imageh*\real{#2}-0.5\Texth}
  \put(\posx,\posy){#3}
}
\newcommand{\putcb}[3]{
  \savebox{\Text}{#3}
  \settoboxwidth{\Textw}{\Text}
  \settoboxheight{\Texth}{\Text}
  \setlength\posx{-\Imagew+0.4pt+0.01\Imagew*\real{#1}-0.5\Textw}
  \setlength\posy{-0.5\Imageh+2.5pt+0.01\Imageh*\real{#2}-\Texth}
  \put(\posx,\posy){#3}
}
\newcommand{\putct}[3]{
  \savebox{\Text}{#3}
  \settoboxwidth{\Textw}{\Text}
  \settoboxheight{\Texth}{\Text}
  \setlength\posx{-\Imagew+0.4pt+0.01\Imagew*\real{#1}-0.5\Textw}
  \setlength\posy{-0.5\Imageh+2.5pt+0.01\Imageh*\real{#2}}
  \put(\posx,\posy){#3}
}

\newcounter{exo}

\begin{document}

\allowdisplaybreaks

\newcommand{\arXivNumber}{2306.03225}

\renewcommand{\PaperNumber}{034}

\FirstPageHeading

\ShortArticleName{Skein (3+1)-TQFTs from Non-Semisimple Ribbon Categories}

\ArticleName{Skein (3+1)-TQFTs from Non-Semisimple\\ Ribbon Categories}

\Author{Francesco COSTANTINO~$^{\rm a}$, Nathan GEER~$^{\rm b}$, Benjamin HA\"IOUN~$^{\rm a}$ \newline and Bertrand PATUREAU MIRAND~$^{\rm c}$}

\AuthorNameForHeading{F.~Costantino, N.~Geer, B.~Ha\"{\i}oun and B.~Patureau Mirand}

\Address{$^{\rm a)}$~Institut de Math\'ematiques de Toulouse, 118 route de Narbonne, 31062 Toulouse, France}
\EmailD{\mail{francesco.costantino@math.univ-toulouse.fr}, \mail{benjamin.haioun@math.univ-toulouse.fr}}

\Address{$^{\rm b)}$~Mathematics and Statistics, Utah State University, Logan, Utah 84322, USA}
\EmailD{\mail{nathan.geer@gmail.com}}

\Address{$^{\rm c)}$~Univ Bretagne Sud, CNRS UMR 6205, LMBA, 56000 Vannes, France}
\EmailD{\mail{bertrand.patureau@univ-ubs.fr}}

\ArticleDates{Received July 04, 2025, in final form March 16, 2026; Published online April 14, 2026}

\Abstract{We define a (3+1)-TQFT associated with possibly non-semisimple finite unimodular ribbon tensor categories using skein theory. This gives an explicit realization of a TQFT predicted by the cobordism hypothesis, based on recent results on dualizability. State spaces are given by admissible skein modules, and we prescribe the TQFT on handle attachments. We give some explicit algebraic conditions on the input category to define this TQFT, namely to be ``chromatic non-degenerate''. As a by-product, we obtain an invariant of 4-manifolds equipped with a ribbon graph in their boundary, and in the ``twist non-degenerate'' case, an invariant of 3-manifolds. Our construction generalizes the Crane--Yetter--Kauffman TQFTs in the semi-simple case, and the Lyubashenko (hence also Hennings and WRT) invariants of 3-manifolds. The whole construction is very elementary, and we can easily characterize the invertibility of the TQFTs, study their behavior under connected sums and provide some examples.}


\Classification{57M27; 57R56}

\section{Introduction}

Since Atiyah's pioneering paper~\cite{At}, the construction of topological quantum field theories (TQFTs) of smooth compact oriented manifolds has flourished. Many of the results in this area apply to
manifolds of dimension $3$ and $(2+1)$-dimensional cobordisms.
Some of the milestones in this dimension include Turaev's~\cite{Tu} general TQFT construction associated to a semi-simple modular category (which is an extension of the Witten--Reshetikhin--Turaev 3-manifold invariant) and the construction of a TQFT associated to the Kauffman bracket skein algebra~\cite{BHMV95}.

These constructions have been generalized to non-semisimple settings first by Hennings~\cite{H96} and Lyubashenko and Kerler~\cite{KerLy2001, L94} and later in~\cite{BCGP14,CGPV2023a,DRGGPMR2022} using the theory of modified traces and the notion of \emph{admissible skein modules}. These non-semisimple TQFTs have new properties, often proving more powerful than their semi-simple analogs. For example, non-semisimple TQFTs have
been
shown to distinguish diffeomorphism types of homotopically equivalent lens spaces which were not distinguished by semi-simple TQFTs.
The Reshetikhin--Turaev 3-TQFTs contain in particular Turaev--Viro theories, and non-semisimple
generalizations
also exist for the latter using non-semisimple pivotal categories and based on admissible string nets, see~\mbox{\cite{CGPV2023a, MSWY23}}.

Less research has been done in dimension $4$. A notable exception is due to Crane and Yetter~\cite{CY}
who defined (3+1)-TQFTs associated to any semi-simple modular category. These TQFTs turn out to be very simple, and the underlying invariants of closed 4-manifolds only depend on the Euler characteristic
and
the signature.
This construction was generalized to certain non-modular semi-simple categories by Crane, Kauffman and Yetter~\cite{CYK}. In~\cite{BB18}, B\"arentz and Barrett generalize the CKY construction to pairs of finite semi-simple categories $\cat \to \catd$ (the CKY construction corresponds to the inclusion of $\cat$ in its Drinfeld center). The underlying 4-manifold invariants of these TQFTs, at least in a broad list of cases, are conjecturally related to the Euler characteristic, the signature and the fundamental group, see~\cite[Conjecture~8.1]{BB18}.
A different kind of quantum topology construction in dimension 4 has
recently been defined by Beliakova and De Renzi in
\cite{BD22, BD21}.
They consider $4$-manifolds with corners which admit a handle decomposition without
$3$- and $4$-handles, and considered up to a suitable $2$-equivalence
relation.\footnote{In some sense, the TQFT of this paper might be seen as
 an extension of their work to include higher index handle (but using
 different methods and reversing the labeling of the index of handle decomposition).}

However, there are good reasons to believe that this list is incomplete.
The main purpose of this paper is to add to this list by considering a general framework of different kinds of 4-TQFT arising from non-semisimple categories and skein theory. Let us put this into context.
All
of
the TQFTs mentioned above are expected to be fully extended and classified, by the Baez--Dolan--Hopkins--Lurie cobordism hypothesis, by some fully dualizable objects in some higher category. The fully dualizable objects that would be associated to the semisimple CKY theories have been studied in~\cite{BJS}, but recently non-semisimple fully dualizable objects were exhibited in~\cite{BJSS}. They prove that non-semisimple modular tensor categories are fully dualizable, and hence should induce 4-TQFTs. Moreover, it is conjectured in~\cite{Haioun} that these 4-TQFTs support a boundary condition that recovers the non-semisimple (2+1)-TQFTs of~\cite{DRGGPMR2022}. See~\cite{HaiounWRT} for a proof of this statement in a non-extended setting using the construction of the present paper.
In this paper, we construct 4-TQFTs which
we expect should be the 3-4 part of
the TQFTs predicted by~\cite{BJSS}
(at least when the category is a ${\rm SO}(3)$-homotopy fixed point so that the associated theory is only oriented),
and give a different look on conditions for full dualizability.
Another way to phrase the main objective of this paper is:
\begin{center}
 \it When do admissible skein modules of $3$-manifolds extend to a $(3+1)$-TQFT?
\end{center}

Note, that the usual definition of a TQFT is that of a symmetric monoidal functor from the category of oriented cobordisms to the category of finite-dimensional vector spaces. In the present article, however, we work with non-semisimple categories, which necessitates a generalization of this framework.
In particular, we consider also non-compact TQFTs (see Section~\ref{S:Top}).
Similarly, the invariant $\ThreeManInv'$ appearing in Theorem~\ref{T:Exist3ManInv} depends not only on the underlying 3-manifold, but also on auxiliary data ensuring admissibility,
but one can extract from it an~invariant of the underlying $3$-manifold denoted $\ThreeManInv_{\cat}(M)$.

\textbf{Dualizability and 4-TQFTs.}
We now contextualize the results of this paper within a larger perspective.
By the Cobordism Hypothesis, fully extended framed $n$-TQFTs should be classified by fully dualizable objects in some $n$-category $n\hskip-2pt\Vect$. As we will see, it proposes an answer to the question above through dualizability requirements. In this introduction, we consider $n=4$ and the model for $4\hskip-2pt\Vect$
will be
 the Morita 4-category of linear braided categories $\operatorname{BrTens}$ introduced in~\cite{BJS, JFS}.
The objects of this 4-category are presentable braided categories, a notion we will not recall here, and we will limit ourselves to examples of the form $\EE:= \operatorname{Fun}(\operatorname{Proj}(\CC)^{\rm op}, \Vect)$, the free cocompletion of the tensor ideal of projectives in a ribbon category $\CC$.

It is shown in~\cite{BJS} that every braided rigid tensor category $\CC$ is
3-dualizable in $\operatorname{BrTens}$
(or more precisely the associated category $\EE$ is 3-dualizable).

We need 4-dualizability for 4-TQFTs, but a 3-dualizable object will still induce a partially-defined 4-TQFT: it
associates
vector spaces to 3-manifolds, but
in general no
linear maps to 4-cobordisms (except those induced by diffeomorphisms). This TQFT is in general framed (all 3-manifolds above are assumed to be framed) but if the object $\CC$ is endowed with an ${\rm SO}(3)$-homotopy-fixed-point structure, the Cobordism Hypothesis predicts that one can obtain an~oriented TQFT. It is expected that a ribbon structure on $\CC$ induces such an ${\rm SO}(3)$-structure and that the vector spaces assigned to oriented 3-manifolds are the admissible skein modules of~\cite{CGP22}. Work of the third author and Brown~\cite{BH} shows that the categories assigned to surfaces by this partially-defined, fully-extended 4-TQFT are indeed described by admissible skein modules in the thickened surface.
Note however that at present the rigorous study of ${\rm SO}(3)$-structures is still out of reach.

Assuming the expectations above, and the cobordism hypothesis, the following two questions become equivalent:
\begin{itemize}\itemsep=0pt
\item When is $\CC$ 4-dualizable and how can one equip it with an ${\rm SO}(4)$-structure extending its ${\rm SO}(3)$-structure?
\item When and how can one extend admissible skein modules to a 4-TQFT?
\end{itemize}
There are actually a number of intermediate steps between 3- and 4-dualizability, see~\cite[Definition~2.25]{Haioun} or~\cite{Lurie},
\begin{gather*}
 \text{3-dualizability = (4,0)-dzb $\Leftarrow$ (4,1)-dzb $\Leftarrow$ (4,2)-dzb}\\
 \hphantom{\text{3-dualizability = (4,0)-dzb }}{}
 \text{$\Leftarrow$ (4,3)-dzb $\Leftarrow$ (4,4)-dzb = 4-dualizability}.
\end{gather*}
Each
level of
dualizability
induces a partially-defined
4-TQFT
which is defined on a
larger
class of 4-cobordisms~\cite[Section~3.4]{Lurie}. If $\CC$ is $(4,k)$-dualizable, it should induce a TQFT defined
on
4-cobordism obtained using only $0,\dots, k$-handles. However, these cobordisms are not considered up to diffeomorphism, but up to $k$-equivalence, which is an a priori weaker relation. There is also a framing issue which we will not discuss here.

Note that (4,2)-dualizable objects are used in~\cite{BD22, BD21} to construct invariants of 4-dimensional 2-handlebody up to 2-equivalence, with the goal of distinguishing 2-equivalent but non-dif\-feo\-mor\-phic 4-manifolds.

In this paper we construct non-compact (3+1)-TQFTs which we expect
will
realize the theory associated with the (4,3)-dualizable case in the above list. In this case, the notion of 3-equivalence is known to coincide with diffeomorphism, hence we do obtain invariants of 4-manifolds.

We will build the (3+1)-TQFT handle-by-handle, and we expect that our successive requirements strongly reflect the successive dualizability requirements displayed above (together with an orientability requirement). Note that in our construction one has to read the indices backwards: we first construct the 4-handle, then the 3-handle and so on. One can take orientation reversal to get the usual indices.

We will work under some simplifying assumptions on the category $\CC$ recalled in Section~\ref{SS:Hypothesiscat}, which we expect
to
correspond to (4,2)-dualizability together with an orientability requirement. Under these assumptions on $\CC$, we expect that the notions of \emph{chromatic non-degenerate} and \emph{chromatic compact} which we introduce in this paper coincide respectively with (4,3)- and (4,4)-dualizability. We do not show that the extension of admissible skein modules to a 4-TQFT that we construct is the only possible one, but this is shown in a once-extended setting by work of the third author~\cite[Theorem~6.1]{HaiounHandle} that
has
appeared after this article was prepublished.

\textbf{The construction.}
\begin{figure}[t]
 \centering
 \small
 \[\begin{tikzcd}[column sep=tiny]
 {\begin{array}{c} \cat\text{ is \color{blue}chromatic}\\\text{\color{blue} non-degenerate}
 \end{array} } &\ & {\exists \left\{\begin{array}{@{}l}{\color{red} \TSkein_\cat}\text{ non-compact (3+1)-TQFT}\\ {\color{red} \dot\TSkein_\cat\bigl(W^4\bigr)} \text{ closed 4-manifold invariant}\\{\color{red} \TSkein_\cat\bigl(W^4,T\bigr)}\text{ invariant of pair }T\subset -\partial W^4\end{array}\right.} \\
 {\cat \text{ is \color{blue}chromatic compact}} &\ & {\begin{array}[t]{l}\TSkein_\cat\text{ extends to a TQFT}\\\TSkein_\cat\bigl(W^4\bigr)=\zeta\dot\TSkein_\cat\bigl(W^4\bigr)\end{array}} \\
 & \begin{array}{c} \cat\text{ is \color{blue}twist}\\
 \text{\color{blue} non-degenerate} \end{array}
 &
 \begin{array}[t]{l}\exists\left\{
 \begin{array}{@{}l}{\color{red} \BB_\cat\bigl(M^3\bigr)}\text{ 3-manifold invariant} \\
 {\color{red}\BB'_\cat\bigl(M^3,T\bigr)}\text{ invariant of pair }T\subset M^3
 \end{array}\right.
 \\\text{generalizing Lyubashenko's invariants}
 \end{array}\\
 {\cat\text{ is \color{blue}modular}} &\ & {
\TSkein_\cat\text{ is invertible}
 }
 \\
 & \begin{array}{c} \cat\text{ is \color{blue}fusion}\\
 \dim\cat\neq0 \end{array} & \TSkein_\cat\text{ is the Crane--Yetter--Kauffman TQFT}\\
 {\cat\text{ is \color{blue}semi-simple modular}} &\ &
 {\begin{array}{l}\TSkein_\cat\text{ is the Crane--Yetter TQFT}\\\BB_\cat\text{ is the WRT 3-manifold invariant}\end{array}}
 \arrow[Rightarrow, from=4-1, to=2-1]
 \arrow[Rightarrow, from=5-2, to=2-1] \arrow[Rightarrow, from=2-1,
 to=1-1] \arrow[rightarrow, from=1-1, to=1-3] \arrow[rightarrow,
 from=2-1, to=2-3] \arrow[Rightarrow, from=4-1, to=3-2] \arrow[Rightarrow, from=6-1, to=5-2]
 \arrow[leftrightarrow, from=4-1, to=4-3]
 \arrow[rightarrow, from=6-1, to=6-3]
 \arrow[Rightarrow, from=6-1, to=4-1]
 \arrow[rightarrow, from=3-2, to=3-3]
 \arrow[rightarrow, from=5-2, to=5-3]
\end{tikzcd}\]
\caption{All categories
$\cat$ above are assumed to be ribbon. This figure represents different properties on a ribbon chromatic category $\cat$ and their relationships and corresponding 3-manifold invariants and TQFTs. A category at the tail of a double arrow implies the property at the head of the arrow. For example, chromatic compact implies chromatic non-degenerate. A category at the tail of a single arrow implies the existence of the invariant at the head of the arrow. For example, a chromatic non-degenerate category gives rise to a non-compact $(3+1)$-TQFT $\TSkein_\cat$.}
\label{F:Representation}
\end{figure}
In this paper, we construct 3-manifold invariants and (3+1)-TQFTs, which
have different qualities depending on the underlying category,
 see
Figure~\ref{F:Representation}.
Our main algebraic tool is a chromatic category, which loosely
speaking, is a $\FK$-linear pivotal category with a modified trace and
a chromatic morphism. For example, any finite tensor category (in the
sense of~\cite{EGNO}) has a structure of a chromatic category.
Spherical chromatic categories were used in~\cite{CGPV2023a} to
produce non-compact $(2+1)$-TQFTs.

In Section~\ref{sec:presman},
we show that a \emph{twist non-degenerate} ribbon chromatic category $\cat$ gives rise to a 3-manifold invariant $\BB_\cat$.
Moreover, we show that this construction includes and generalizes many well known invariants:
 \begin{enumerate}\itemsep=0pt
\item[(1)] WRT-invariants~\cite{RT} (when $\cat$ is semi-simple modular),
\item[(2)] the (modified) Hennings invariants~\cite{DRGPM18, H96} (when $\cat$ is the category of representations over a finite-dimensional, unimodular ribbon Hopf algebra), and
\item[(3)] the (modified) Lyubashenko invariants~\cite{DRGGPMR2022, L94} (when $\cat$ is a finite ribbon tensor category).
\end{enumerate}
The construction of $\BB_\cat$ is very similar to that of Reshetikhin--Turaev invariants with the chromatic morphism playing the role of a non-semisimple Kirby color.

 Then in Section~\ref{sec:3mfld}, we build a $(3+1)$-TQFT from
a ribbon chromatic category $\cat$ with certain properties.
The idea is to use work of Juh\'asz~\cite{Juh2018} which allows one to build a TQFT by associating a vector space to each closed $3$-manifold (all manifolds are
assumed to be
smooth from now on) and a set of operators associated to $4$-cobordisms formed by a single handle of index $k$ for each $0\leq k \leq 4$ (see Section~\ref{sec:presman} for details).
We use the vector space of admissible skeins introduced in~\cite{CGP22} as value for the TQFT on closed $3$-manifolds. In order to define the maps associated
with
 handles of index $k$ we need to progressively add some structure or restrictions on the category $\cat$, when $k$ starts from $4$ and ranges down to $0$. Schematically:
\begin{description}\itemsep=0pt
\item[The 4-handle] will be given by the \emph{modified trace}, the usual key ingredient in non-semisimple skein theory. See Figure~\ref{fig:4handle}. It is shown to exist and to be unique up to scalar in a~unimodular ribbon tensor category in~\cite[Corollary 5.6]{GKP22}. It is equivalent to a~linear form on the admissible skein module of $S^3$, see~\cite[Theorem~3.1]{CGP22}. Our TQFT will depend on a~choice of modified trace, but simply by a term depending only on the Euler characteristic, see Proposition~\ref{prop:changeOfModifiedTraceTQFT}.
\item[The 3-handle] will be given by the \emph{cutting morphism} first introduced in~\cite{CGPT20} and already used in~\cite{CGPV2023a}. See Figures~\ref{fig:3handle} and~\ref{F:cut}. It is associated with the copairing of the modified trace and its existence imposes a non-degeneracy condition on the modified trace.
\item[The 2-handle] will be given by the \emph{chromatic morphism} studied in~\cite{CGPV2023b}. See Figures~\ref{fig:2handle} and~\ref{F:surgery}. It plays the role of the Kirby color, and in the abelian case is another way to phrase the integral in the coend used in~\cite{DRGGPMR2022}. The existence of such a map is proved in~\cite{CGPV2023b} under suitable finiteness conditions on $\cat$, even for spherical categories, but we reprove it here for the reader's convenience when $\cat$ is a finite unimodular ribbon tensor category, see Section~\ref{ss:coend}. This chromatic morphism is not claimed to be unique in any way, but our construction does not depend on a choice, see Proposition~\ref{prop:redtoblue}.
\item[The 1-handle] will be given by the \emph{gluing morphism}. See Figures~\ref{fig:1handle} and~\ref{F:glue}. This is a new notion which we introduce. The category $\cat$ is called \emph{chromatic non-degenerate} when a~gluing morphism exists, see Definition~\ref{def:maindef}. It can indeed be rephrased by asking that a~certain morphism involving the chromatic morphism and the double braiding is non-zero, see Proposition~\ref{P:existsGluing}. Again, our construction will not depend on the choice of gluing morphism, see Proposition~\ref{prop:1handlewelldef}.
\item[The 0-handle] will be given by the \emph{global dimension} of $\cat$. See Figure~\ref{fig:0handle}. It is only the choice of a scalar because the state space of the 3-sphere is one-dimensional. However, for a~coherent choice to exist, it will impose conditions on the gluing morphism, namely that it is invertible, see Lemma~\ref{lem:invertibleg}. We call this condition \emph{chromatic compact} because it is a~strengthening of the above that yields to fully-defined TQFTs instead of non-compact~ones.
\end{description}

As already mentioned, the $(3+1)$-TQFT of this paper has different qualities depending on
the new notions of the category $\cat$, see Figure~\ref{F:Representation}.
The most general notion we
consider is called a \emph{chromatic non-degenerate} category.
Theorem~\ref{T:S} says such a category gives rise to a non-compact $(3+1)$-TQFT. Here non-compact means we only
consider cobordisms whose connected components all have non-empty sources and so we do not consider $0$-handles. Therefore, this construction associates morphisms to $k$-handles,
for $k=1,2,3,4$ as described above. To extend this TQFT to
all cobordisms the category must satisfy a stronger condition, namely it should be
\emph{chromatic compact}. We show that if we make further
restrictions on the category we can characterize the corresponding
TQFTs: the $(3+1)$-TQFT is invertible if and only if $\cat$ is
\emph{factorizable} which is a generalization of factorizability of
finite ribbon tensor categories, see Theorem~\ref{T:invertible}.

One advantage of the techniques of this paper is that they are based on algebraic data which are easy to formulate with low-level technology using monoidal categories, in particular the notions and properties of non-degenerate modified trace and chromatic morphisms.
These properties are easily represented graphically and most of our proofs reduce to diagrammatic ones.\looseness=-1

\textbf{Higher structure.}
As discussed above, the TQFTs constructed in this paper are expected to be an incarnation of theories that were predicted by higher-algebraic results, and we expect that our construction recovers the 3-4-part of the fully extended theories associated with the fully dualizable objects of~\cite{BJSS}.

Now that we have described our construction, let us compare it more explicitly with the predictions coming from higher algebra. First, note that the dualizability results of~\cite{BJSS, BJS} include only categories that are either semisimple or modular, in characteristic 0,
while
our construction is much more general. We naturally expect that chromatic compact categories, for which our construction works, are fully dualizable. We believe that this expectation was not known in the higher-algebraic
settings,
 and our conditions can guide future dualizability results.
It was recently proven in~\cite{DecoppetRelativeInvert}, following expectations of~\cite{BJSS} that a finite tensor category with separable semi-simple M\"uger center is fully dualizable. It would be interesting to compare this condition to our notion of chromatic compactness. We can achieve this in the extreme case where the category is symmetric monoidal, see Proposition~\ref{P:symcase}.

In particular, we expect that the TQFTs we describe can be extended down to the point using skein theory. A description of the presumed values of this theory on surfaces appeared after the publication of a preprint of this article in~\cite{BH}.

It is also natural to expect that the construction of this paper is an explicit
example
 in dimension four of the work of David Reutter and Kevin Walker announced in~\cite{ReutterWalker1,Walker2021,ReutterWalker2}.

\textbf{Boundary conditions.}
A particularity of skein TQFTs is that vectors in the TQFT space are
represented by $\cat$-colored links or ribbon graphs in 3-manifolds.
Hence the abstract TQFT comes with an invariant of
a pair $(W,T)$ where $T$ is a $\cat$-colored link
or ribbon graph in the boundary of $W$. We show in Theorem~\ref{T:4TQFT-3M} that if the category is twist
non-degenerate, this invariant recovers the invariant of 3-manifold $\BB'_\cat$ that we constructed.

This invariant can be much stronger than the TQFT itself. Indeed the invertible CY TQFTs are known to only depend on two complex numbers~\cite[Theorem~7.6]{SP}, but given the empty skein they recover the WRT invariant of the boundary of a 4-manifold. The invertible TQFTs associated with a non-semisimple modular category are similarly simple, but we show in Theorem~\ref{T:4TQFT-3M} that, given an admissible skein in their boundary, they recover the invariants of decorated 3-manifolds from~\cite{DRGGPMR2022}.

In fact, these particularly simple examples of 4-TQFTs constructed from modular input were an important motivation for this paper. Indeed, they model the \emph{anomaly} of the non-semisimple~\cite{DRGGPMR2022} (2+1)-TQFTs. The notion of anomaly, describing the projectivity of a TQFT,
has
recently received a lot of attention~\cite{FreedAnomalyInvFT, FreedAnomaly}. According to Freed, a projective $n$-TQFT is a~boundary condition to an invertible $(n+1)$-TQFT, called its anomaly.

The first and most fundamental example of projective TQFT, which motivated Freed's notion, is the Witten--Reshetikhin--Turaev (2+1)-TQFT. Walker describes how it can be seen as~a~boundary condition to its anomaly: the Crane--Yetter 4-TQFT. Using the present paper, the third author~\cite{HaiounWRT} extended this story to the non-semisimple case, and wrote \cite{DRGGPMR2022} the theory as~a~boundary condition to the 4-TQFT we construct here.

\textbf{Detecting exotic pairs.}
It is a natural question to ask how fine 4-manifold topology can 4-TQFTs detect, and in particular if they can be sensitive to exotic phenomena. There are many obstructions to this, boiling down to the fact that by
definition TQFTs behave well with respect to connected sums,
but exotic phenomena can be killed by taking connected sums with some fixed 4-manifold. Nevertheless, there are TQFT-like constructions, e.g., Heegaard--Floer homology or lasagna skein modules, which are known to be sensitive to exotic phenomena. Since we also construct partially-defined 4-TQFTs, it is natural to ask:
\begin{center}
 \it How sensitive can $4$-manifolds invariants coming from skein theory be?
\end{center}
Let us discuss some positive and negative partial answers. We begin with obstructions for our construction to detect exotic pairs of 4-manifolds.

First, extending a previous result of Wall,
Gompf shows in~\cite{Go}
that two compact orientable $4$-manifolds (possibly with boundary) which are homeomorphic become diffeomorphic after some finite sequence of connected sums with $S^2\times S^2$. Therefore, a $(3+1)$-TQFT which is multiplicative under connected sum, has a hope to distinguish exotic pairs only if the numerical invariant it associates to $S^2\times S^2$ is not invertible.
A similar restriction is that if one stabilizes a closed simply connected $4$-manifold by a sequence of connected sums with $\CP^2$ or \smash{$\overline{\CP^2}$} then one can make it diffeomorphic to a connected sum of the form \smash{$\#_m\CP^2\#_n\overline{\CP^2}$} for some $m,n\in \N$. Therefore, again, in order for a $(3+1)$-TQFT to have interesting invariants on closed simply-connected $4$-manifolds, it needs to have non-invertible values on $\CP^2$ and \smash{$\overline{\CP^2}$}.

Furthermore, given a $(3+1)$-TQFT $Z$ then both $Z\bigl(S^3\bigr)$ and $Z\bigl(S^2\times S^1\bigr)$ can be endowed with the structure of a commutative Frobenius algebra (this is done in a similar way to the standard approach in dimension $(1+1)$).
In~\cite{Reutter}, Reutter showed that a $4$-dimensional TQFT which is ``semisimple'' (i.e., for which $Z\bigl(S^3\bigr)$ and $Z\bigl(S^2\times S^1\bigr)$ are semi-simple algebras) cannot detect exotic pairs.
Finally, we mention that Schommer-Pries in~\cite[Theorem~7.6]{SP} shows that if~$Z$ is invertible (i.e.,
if
 $Z(M)$ is $1$-dimensional for each $3$-manifold $M$ and $Z(W)$ is an invertible linear map for each 4-dimensional cobordism $W$) then its $4$-dimensional invariant only depends on the Euler characteristic and the first Pontryagin number (for instance CY theories are all invertible).

Let us turn to positive partial answers, all of which come from explicit examples.

First, in~\cite{ReutterSP}, the authors show that some semisimple TQFTs generalizing Dijkgraaf--Witten theories can detect stable-diffeomorphism classes \big(i.e., 4-manifolds up to taking connected sums with $S^2\times S^2$ on both sides\big) under some $\pi$-finiteness assumptions on the 4-manifolds.

In Section~\ref{S:Examples}, we consider several examples. First, in Section~\ref{SS:ExSS}, we consider the semi-simple case and show we recover CKY theories, which conjecturally depend only on Euler characteristic, signature and fundamental group.
Then in Section~\ref{SS:Ex-sl2}, we
carry out some computations
for the TQFT associated to finite-dimensional versions of quantum $\slt$ at root of unity. In many
cases, the associated category is shown to be chromatic compact but not
factorizable. At~a~\smash{$8r$\textsuperscript{th}} root of unity, the category is twist
degenerate and $\TSkein_\cat\bigl(\CP^2\bigr)=0$. We can show that our invariant does not depend only on Euler characteristic, signature and fundamental group, but even then, ${\TSkein_\cat\bigl(S^2\times S^2\bigr)\neq 0}$, so this example cannot detect
exotic pairs.
In Section~\ref{SS:Ex-chp}, we conclude with a toy example
in characteristic $p$ showing the existence of chromatic
non-degenerate categories which fail to be chromatic compact.

In order to detect exotic pairs with our construction, we expect that the input category should be chromatic non-degenerate but not chromatic compact.
By the above discussion, it certainly should not be semisimple or modular.
In fact, using the recent results of~\cite{HaiounWRT}, the construction of the present paper yields once-extended TQFTs, and the results of~\cite{Reutter} show that TQFTs arising from chromatic compact categories cannot detect exotic pairs.
Our expectation is that, looking at less trivial examples of chromatic non-degenerate categories over a field with finite characteristics (expanding the example of Section~\ref{SS:Ex-chp}) could produce non trivial invariants, but currently we have no explicit examples yet.
This is the first time that this kind of structure is shown to have topological applications, and therefore it has never been looked for before. However, since the publication of a preprint of this article, a very interesting family of examples, which first appeared in~\cite{STWZ23}, have been put forward in~\cite{Decoppet}. The third author and D\'ecoppet studied invariants of 4-dimensional 2-handlebodies associated to these categories using the constructions of the present paper in~\cite{DH2HB}. We are hopeful that our construction will motivate and open up new directions of research of which these papers are the first example. See in particular the recent studies~\cite{FaesManko, Manko}.

\section{Algebraic setting}
\subsection{Hypothesis on the category}\label{SS:Hypothesiscat}
Throughout this paper, we fix an algebraically closed field $\FK$ and
an essentially small strict ribbon $\FK$-linear
category $\cat$. We do not need $\cat$ to be finite abelian but we
assume that it has finite-dimensional hom-spaces, is additive,
idempotent complete, has a simple unit $\unit$ such that any non-zero
morphism to $\unit$ is an epimorphism, and admits a non-zero projective
generator.

We will also require that $\cat$ is a chromatic category, i.e., it has a
non-degenerate m-trace and that there exists a chromatic morphism for a projective generator (see the
next section for the definitions).

All of these conditions are automatically satisfied for finite unimodular ribbon tensor categories. Depending on the constructions, we will sometimes further require that $\cat$ is twist non-degenerate, chromatic non-degenerate or chromatic compact, see Figure~\ref{F:Representation}.

 \subsection{Cutting, chromatic and gluing morphisms}\label{SS:Def-gluingmorph}
 We use the notation and terminology of~\cite{CGP22}. In particular,
$\cat$ is a braided category with braiding
$c_{V,W}\colon V\otimes W\to W\otimes V$ and a compatible pivotal structure
which associates to each $V \in \cat$ an~object $V^*\in \cat$ and the
(co)evaluation morphisms $\lev_V\colon V^*\otimes V \to \unit, \lcoev_V, \rev_V,\rcoev_V$ which are
related by a canonical isomorphism $\phi_V\colon V\to V^{**}$. We consider
$\cat$-colored ribbon graphs in oriented manifolds. These ribbon
graphs may have coupons which we denote by a box filled with a morphism.

\begin{definition}
 Let $\Proj$ be the ideal of projective objects of $\cat$. A
\emph{non-degenerate m-trace on $\Proj$} is a family of linear maps
$\mt=\{\mt_P\colon\End_\cat(P)\to\FK\}_{P\in\Proj}$ satisfying the
following properties:
\begin{enumerate}\itemsep=0pt
\item \textit{Cyclicity:} For all $U,V\in \Proj$ and morphisms
 \smash{$ V\tto f U$}, \smash{$U\tto g V$}, we have $ \qt_V(g f)=\qt_U(f g)$.
\item \textit{Right partial trace:} 
 If
 $U\in \Proj$ and $W\in \cat$, then for any
 $f\in \End_\cat(U\otimes W)$,
\begin{equation}\label{E:RightPartialTraceProp}
 \qt_{U\otimes W}\bp{f}
 =\qt_U\bigl(\bigl(\Id_U \otimes \rev_W\bigr)(f\otimes \Id_{W^*})\bigl(\Id_U \otimes \lcoev_W\bigr)\bigr).
\end{equation}
\item \textit{Non-degeneracy}: For any $P\in\Proj$, the pairing
 $\Hom_\cat(\unit,P)\otimes_\FK\Hom_\cat(P,\unit)\to\FK$ given by
 $(x,y)\mapsto\mt_P(x\circ y)$ is non-degenerate.
\end{enumerate}
Note since $\cat$ is ribbon the m-trace also satisfies the left
partial trace property similar to equation~\eqref{E:RightPartialTraceProp}.
\end{definition}

Let $F$ be the Reshetikhin--Turaev functor (R-T functor from now on) from the category of $\cat$-ribbon graphs
in $\R^2\times [0,1]$ to $\cat$, see~\cite{Tu}.
Let $\LL_{\rm adm}$ be the class of all $\cat$-colored ribbon graphs in
$S^3$ obtained as the braid closure of a (1,1)-ribbon graph $T_V$
whose open edge is colored with an object $V\in\Proj$. Then the
assignment
\begin{equation*}\label{E:DefF'}
 F'\colon \ \LL_{\rm adm}\to \kk \qquad\text{given by } F'(L)=\mt_V(F(T_V))
\end{equation*} is an isotopy invariant of $L$ in $S^3$ (see~\cite{GP18, GPT09}).

\begin{definition}
 For any $P \in \Proj$, we set
\begin{gather}
 \Omega_P=\sum_i x^i\otimes_\FK x_i \in\Hom_\cat(P,\unit)\otimes_\FK\Hom_\cat(\unit,P)
 \qquad \text{and}\nonumber\\
 \Lambda_P =\sum_ix_i\circ x^i \in\End_\cat(P),\label{eq:cop}
\end{gather}
where \smash{$\bigl\{x^i\bigr\}_i$} and $\{x_i\}_i$ are basis of $\Hom_\cat(P,\unit)$ and
$\Hom_\cat(\unit,P)$ which are dual with respect to the m-trace, that
is, such that \smash{$\mt_P\bigl(x_i\circ x^j\bigr)=\delta_{i,j}$}. Clearly, $\Omega_P$
and $\Lambda_P$ are independent of the choice of such dual basis.
\end{definition}

The properties of the $m$-trace translate to the copairings $\Omega_P$
as follows.
\begin{lemma} \label{P:Omega-nat} \samepage \
 \begin{enumerate}\itemsep=0pt
 \item[$1.$] Duality: If \smash{$\Omega_P=\sum_ix^i\otimes x_i$}, then
 \smash{$\Omega_{P^*}=\sum_ix_i^*\otimes
 \bigl(x^i\bigr)^*\in\Hom_\cat(P^*,\unit)\otimes_\FK\Hom_\cat(\unit,P^*).$}
 \item[$2.$] Naturality: If $f\colon P\to Q$ is a morphism in $\Proj$, \smash{$\Omega_P=\sum_ix^i\otimes x_i$}, and
 \smash{$\Omega_Q=\sum_iy^i\otimes y_i$}, then
 \begin{gather*}
 \sum_ix^i\otimes (f\circ x_i)=\sum_i\bigl(y^i\circ f\bigr)\otimes
 y_i\in\Hom_\cat(P,\unit)\otimes_\FK\Hom_\cat(\unit,Q).
 \end{gather*}
 \item[$3.$] Rotation: If $V\in\cat$ and \smash{$\Omega_{P\otimes V}=\sum_iz^i\otimes z_i$},
 then \smash{$\Omega_{V\otimes P}=\sum_i\wt z^i\otimes \wt z_i$}, where
 \begin{gather*}
 \wt z_i=\bigl(\Id_{V\otimes P}\otimes\rev_V\bigr)\circ(\Id_V\otimes
 z_i\otimes\Id_{V^*})\circ\bigl(\lcoev_V\bigr) \qquad \text{and} \\
 \wt z^i=\bigl(\lev_V\bigr)\circ\bigl(\Id_V\otimes
 z^i\otimes\Id_{V^*}\bigr)\circ\bigl(\Id_{V\otimes P}\otimes\rcoev_V\bigr).
 \end{gather*}
 \end{enumerate}
\end{lemma}
\begin{proof}
 The duality and rotation properties follow since we apply
 transformations that send dual bases to dual bases. The naturality
 can be checked by applying
 $\mt_P(x_k\circ\_)\otimes\mt_Q\bigl(\_\circ y^\ell\bigr)$ to both side then
 the equality reduces to the cyclic property of the m-trace:
 $\mt_Q\bigl(f\circ x_k\circ y^\ell\bigr)=\mt_P\bigl(x_k\circ y^\ell\circ f\bigr)$.
\end{proof}

We fix a projective cover of the unit $P_\unit$ (which can be chosen
as any indecomposable summand of $P^*\otimes P$ on which \smash{$\lev_P$} is
non-zero for any object $P\in\Proj$) and a non-zero morphism
$\ve\colon P_\unit\to\unit$. Then since the
m-trace
is non-degenerate, there
exist $\eta\colon \unit\to P_\unit$
such that
$\mt_{P_\unit}(\eta\circ\ve)=1$.
(If $\cat$ is semi-simple, we choose $P_\unit=\unit$ and
$\ve=\eta=\Id_\unit$). Let $\Gamma_0$ be the closed ribbon graph
\begin{equation}\label{GammaZero}
\Gamma_0=\begin{tikzpicture}[baseline = 8pt]
 \node[draw, rectangle, minimum height = 0.4cm, minimum width =
 0.7cm] (eta) at
 (0,0){$\eta$}; \node[draw, rectangle, minimum height = 0.4cm,
 minimum width = 0.7cm] (epsilon) at
 (0,1){$\varepsilon$}; \draw[blue] (eta) -- (epsilon)
 node[midway,sloped]{$>$} node[black, midway, right]{$P_\unit$};
\end{tikzpicture}.
\end{equation}
Then $F'(\Gamma_0)=\mt_{P_\unit}(\eta\circ\ve)=1$.

Recall that a \emph{projective generator} of $\cat$ is a projective object $G$ such that any projective object is a retract of $G^{\oplus n}$ for some $n\in\N$.
\begin{definition}
A \emph{chromatic morphism} for a projective generator $G$ is a map $\chr\in\End_\cat(G\otimes G)$
satisfying
\begin{equation*}
\epsh{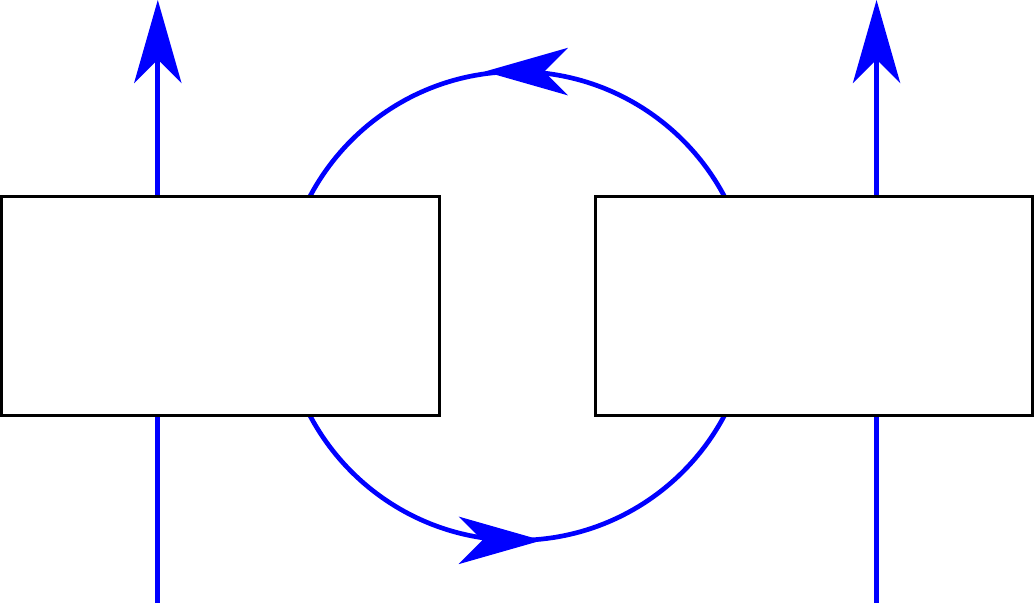}{10ex}
\putc{20}{48}{$\ms{\Lambda_{G\otimes G^*}}$}
\putc{79}{49}{$\chr$}=
 \epsh{fig20}{10ex} \,.\end{equation*}
More generally, a \emph{chromatic morphism based on $P \in \Proj$} for
a projective generator $G$ is a map $\chr_P\in\End_\cat(G\otimes P)$
such that for all $V\in\cat$, we have
\begin{equation}
 \label{eq:chrP}
 \epsh{fig19.pdf}{10ex}
\putc{20}{48}{$\ms{\Lambda_{V\otimes G^*}}$}
\putc{79}{49}{$\chr_P$}=
 \epsh{fig20}{10ex} \;
 \qquad\text{that is}\qquad
 \sum_i\epsh{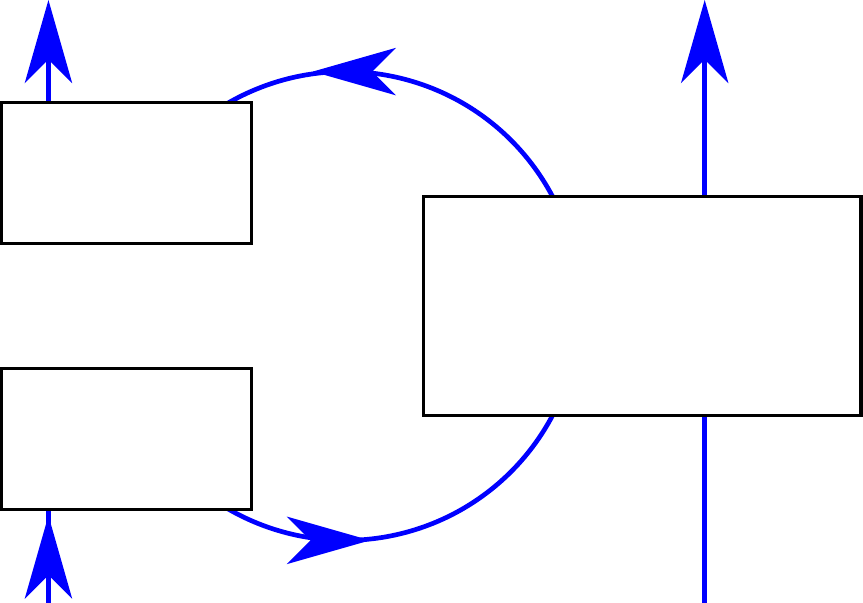}{10ex}
\putc{14}{72}{$\ms{x_i}$}
\putc{14}{26}{$\ms{x^i}$}
\putc{72}{50}{$\ms{\chr_P}$}
\putrc{1}{93}{$\ms{V}$}
\putrc{1}{6}{$\ms{V}$}
\putcb{38}{4}{$\ms{G}$}
\putct{42}{95}{$\ms{G}$}
\putlc{86}{93}{$\ms{P}$}
=\epsh{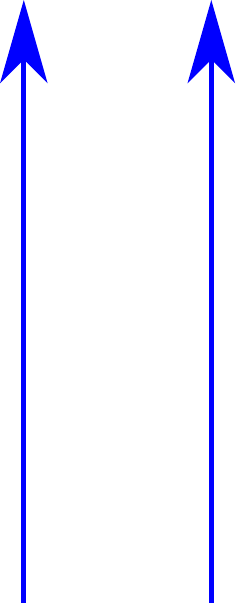}{10ex}
\putrc{-1}{92}{$\ms{V}$}
\putlc{100}{93}{$\ms{P}$}
 \,,
\end{equation}
where $\{x_i\}_i$ and \smash{$\bigl\{x^i\bigr\}_i$} are any dual bases.
\end{definition}

\begin{remark}
 If $c\colon G\otimes P \to G\otimes P$ is a morphism satisfying
 equation \eqref{eq:chrP} where~$G$ and~$P$ are projective objects
 then
 one can show that
 $G$ is a projective generator.
 Indeed, if~${V\in\Proj}$, there exists
 $f\in\End_\cat(V\otimes P)$ such that $\ptr_R(f)=\Id_V$. Then if
 $\Omega_{V\otimes G^*}=\sum_ix^i\otimes x_i$, we have that
 $\Id_V=\sum_i\beta_i\circ\alpha_i$ where
 \begin{gather*}
 {\alpha_i=\bigl(x^i\otimes\Id_G\bigr)\circ\bigl(\Id_V\otimes\rcoev_G\bigr)\colon \ V\to G} \qquad \text{and}\\
 {\beta_i=\ptr_R\bigl(f\circ\bigl(\Id_V\otimes\lev_G\otimes\Id_P\bigr)\circ(x_i\otimes c)\bigr)\colon \ G\to V}.
 \end{gather*}
\end{remark}

 Clearly, a chromatic morphism based on $G$ is a chromatic morphism. Conversely,
any chromatic morphism gives rise to chromatic morphisms based on
projective objects.
\begin{lemma}[{\cite[Lemma 1.2]{CGPV2023a}}]
 Let $\chr\in\End_\cat(G\otimes G)$ be a chromatic morphism and
 ${P\in\Proj}$. Pick any non-zero morphism $\ve_G\colon G\to\unit$ and a
 morphism ${{e}}_{P,G}\colon P\to G\otimes P$ such that
 $\Id_P=(\ve_G\otimes\Id_P)\circ{{e}}_{P,G}$ $($such morphisms always
 exist$)$. Then the map
 \begin{equation*}
 {{\chr}}_P=
 (\Id_{G}\otimes\ve_G\otimes\Id_P)\circ(\chr\otimes\Id_P)\circ(\Id_{G}\otimes{{e}}_{P,G}) \in \End_\cat(G\otimes P)
 \end{equation*}
 is a chromatic morphism based on $P$.
\end{lemma}

\begin{definition}
 A \emph{chromatic
 category} is a $\FK$-linear pivotal category $\cat$ endowed with a
 non-degenerate m-trace on $\Proj$, in which there exist a non-zero
 projective generator and a chromatic morphism.
\end{definition}
\begin{remark}
In the case where $\cat$ is a tensor category in the sense of~\cite{EGNO}, in particular is abelian, the notion of chromatic morphism is a reformulation of that of the integral in the canonical coend developed by Lyubashenko~\cite{L94} to produce invariants of 3-manifolds, see \eqref{eq:chrcoend}. However, it has two key advantages.

First, it allows for non-abelian categories satisfying our hypotheses above.
An example of $\cat$ of this kind is the full subcategory of $U_q(\mathfrak{sl}_2)$-mod (for $q$ a root of unity) whose objects are all the simple and projective modules and their direct sums. It is easy to check that this category is chromatic but it is not abelian as, for instance, the kernel of the projective cover of $\unit$ is not an object in $\cat$ (it is an indecomposable, non projective, non simple object).
More generally,
we~expect that the requirements we fix on chromatic categories are
more flexible and should allow one to find more examples than
restricting to abelian categories. For instance, we expect that working
over a ring (instead of a field) could yield examples of interesting
chromatic categories which are not abelian; another potential source
of examples is coming from the construction outlined in~\cite[Theorem~4.2]{CGP22} where a suitable quotient of a monoidal category is
considered: in general, we do not expect such a quotient to be abelian,
but it could be chromatic (for instance it does, by design, admit a
modified trace).

At this point, one may also ask whether it is possible to suitably
embed a chromatic category in an abelian one which has the same ideal of projective objects, and yield the same topological invariants. A natural approach for such an abelianization, studied for example in~\cite{BEO23}, is to consider the category of finite-dimensional
$End(G)$-modules, where $G$ is a projective generator of $\cat$. This category inherits a monoidal structure and a braiding from that of $\cat$. However,
we~were not able to verify that it is rigid in full
generality.

The second point which makes us prefer chromatic morphisms is practical computational purposes. Even when the abelianization process suggested above does yield a chromatic abelian category, it is often much harder to work with than $\cat$ itself. This is, for example, the case in the mixed higher Verlinde categories studied in~\cite{Decoppet, DH2HB} where $\cat$ has a good diagrammatic description and a complete classification of all irreducible projective objects is known, but its abelianization is still largely unexplored.
\end{remark}

We will prove the following lemma in Section~\ref{ss:gluing}.

\begin{lemma}\label{L:Delta}
 There exist scalars $\Delta_+,\Delta_-\in\FK$ and a family of
 \smash{$\bigl\{\Delta_0^P\in \Hom_{\cat}(P,P)\bigr\}_{P\in\Proj}$}, such that for any
 chromatic morphisms $\chr_{P_\unit}$, $\chr_P$ based on $P_\unit$
 and $P$ respectively, one has
 \begin{gather*}
 F\bp{\epsh{fig22-1}{12ex}
 \putr{43}{53}{$\ms{G}$}
 \putl{85}{8}{$\ms{P_\unit}$}
 \putc{75}{35}{$\ms{\chr_{P_\unit}}$}
 \putc{81}{91}{$\ms{\ve}$}
 \ }=\Delta_+\ve,
 \quad F\bp{\epsh{fig22-2}{12ex}
 \putr{43}{53}{$\ms{G}$}
 \putl{85}{8}{$\ms{P_\unit}$}
 \putc{75}{35}{$\ms{\chr_{P_\unit}}$}
 \putc{81}{91}{$\ms{\ve}$}
 \ }=\Delta_-\ve, \quad\text{and}\quad
 F\bp{\ \epsh{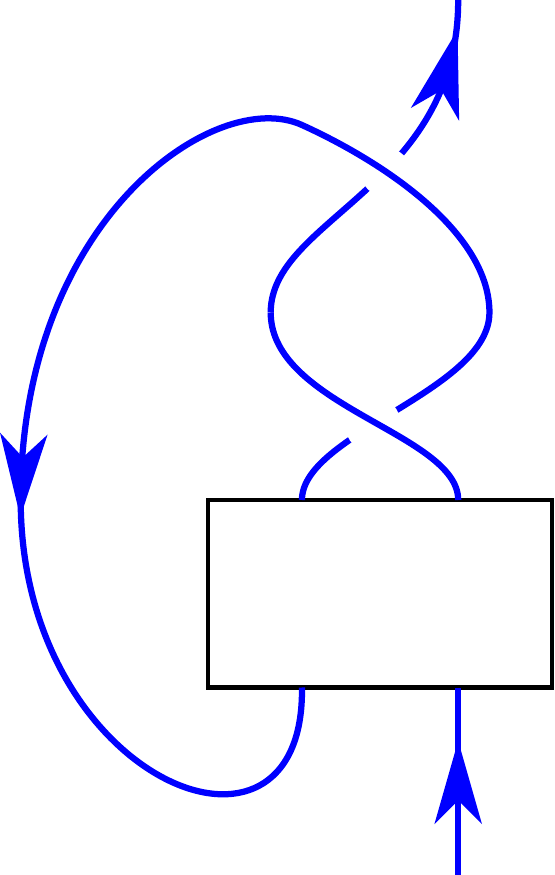}{12ex}
 \putlc{88}{91}{$\ms{P}$}
 \putrc{-2}{45}{$\ms{G}$}
 \putc{68}{32}{$\ms{\chr_{P}}$}
 \putlc{91}{10}{$\ms{P}$}\
 }
 =\Delta_0^P.
 \end{gather*}
 In particular, the evaluation of these diagrams does not depend on the choice of chromatic morphisms.
\end{lemma}

\begin{definition}
 A
 \emph{gluing morphism}
 is an endomorphism
 \[\gm\in\End_\cat(P_\unit)\text{ such that }
 \gm\circ\Delta_0^{P_\unit}=\Lambda_{P_\unit},\text{ i.e., }\quad
 \epsh{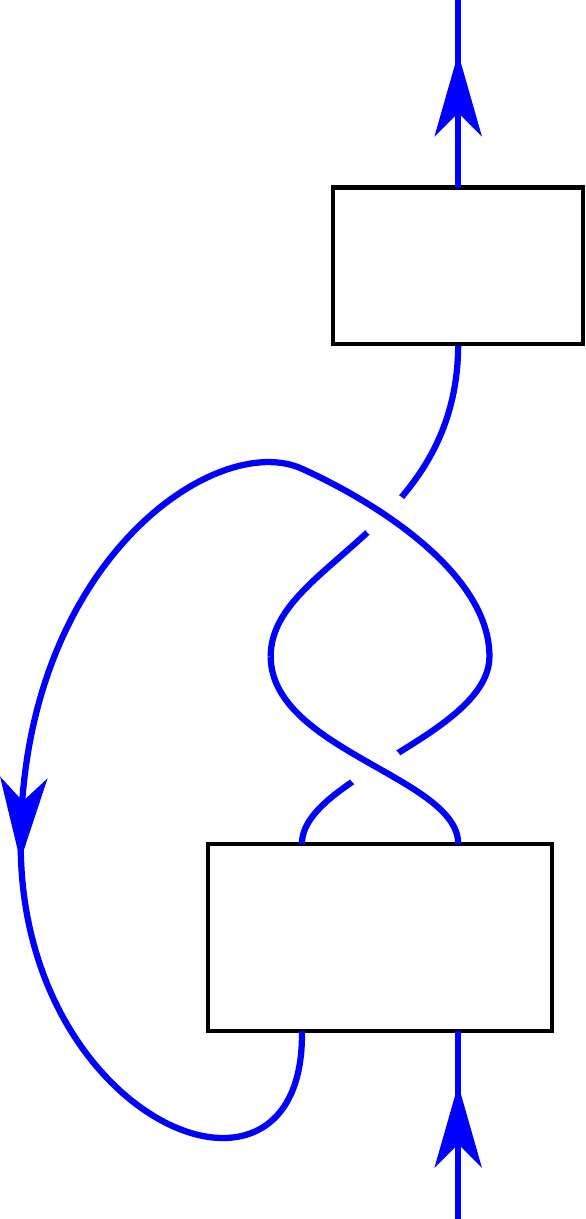}{16ex}
 \putlc{88}{92}{$\ms{P_\unit}$}
 \putlc{86}{7}{$\ms{P_\unit}$}
 \putrc{-2}{32}{$\ms{G}$}
 \putc{78}{78}{$\ms{\gm}$}
 \putc{65}{25}{$\ms{\chr_{P_\unit}}$}=
 \epsh{fig21-0}{9ex}
 \putc{50}{50}{$\ms{\Lambda_{P_\unit}}$}
 \putlc{60}{90}{$\ms{P_\unit}$} .
 \]
 \vspace*{-3ex}
 \end{definition}
 We will prove in Section~\ref{ss:gluing} the following:
\begin{proposition}\label{P:existsGluing}
 The category $\cat$ admits a gluing morphism
 $\gm\in\End_\cat(P_\unit)$ if and only if
 $\Delta_0^{P_\unit} \neq 0$.
\end{proposition}

\subsection{Statement of main results}
\begin{definition}\label{def:maindef}
 We say that
 \begin{enumerate}\itemsep=0pt
 \item $\cat$ is \emph{twist non-degenerate} if
 $\Delta_+\Delta_-\neq0$,
 \item $\cat$ is \emph{chromatic non-degenerate} if
 \smash{$\Delta_0^{P_\unit}\neq0$}, i.e., if $\cat$ admits a gluing
 morphism,
 \item $\cat$ is \emph{chromatic compact} if there exists a scalar
 $\zeta\in\FK^*$, called the \emph{global dimension} of $\cat$, such that
 \smash{$\Delta_0^{P_\unit}=\zeta\Lambda_{P_\unit}$},
 \item $\cat$ is \emph{factorizable} if there exists a scalar
 $\zeta\in\FK^*$ such that for any projective $P$,
 \smash{$\Delta_0^P=\zeta\Lambda_P$}.
 \end{enumerate}
\end{definition}
We will see in Section~\ref{ss:coend} that the notion of factorizable (sometimes called modular)
and twist non-degenerate coincide with the usual ones in finite tensor
categories. The global dimension $\zeta$ is called modularity parameter in~\cite{DRGGPMR2022}. Note that it does depend on the choice of the modified trace, see Proposition~\ref{prop:changeOfModifiedTraceTQFT}.

Clearly, factorizable
$\implies$ chromatic compact (with the same
scalar $\zeta$) $\implies$ chromatic non-degenerate. We will also see in Lemma~\ref{L:fact>twistnd}
that factorizable
$\implies$ twist non-degenerate with
$\Delta_+\Delta_-=\zeta$.

The main constructions of this paper are the following:
\begin{enumerate}\itemsep=0pt
\item If $\cat$ is twist non-degenerate, there exists 3-manifold
 invariants $\ThreeManInv_\cat$ and $\ThreeManInv'_\cat$ (see Theorem~\ref{T:Exist3ManInv})
 that generalize many quantum invariants defined
 through link surgery.
\item If $\cat$ is chromatic non-degenerate, there exists a non-compact (3+1)-TQFT $\TSkein_\cat$ (see Theorem~\ref{T:S}) whose vector
 spaces are the admissible skein modules of 3-manifolds.
\item If
$\cat$ is chromatic compact, then $\TSkein_\cat$ extends to a full (3+1)-TQFT (see Theorem~\ref{T:S}).
\item If $\cat$ is factorizable, then $\TSkein_\cat$ is an invertible
 (3+1)-TQFT (see Theorem~\ref{T:invertible}).
\end{enumerate}
\begin{remark}
The name \emph{chromatic compact} refers to the fact that $\TSkein_\cat$ extends from a non-compact TQFT (i.e., one defined only
on
 cobordisms whose components have non-empty source) to a
full TQFT, i.e., one defined
on all cobordisms.
\end{remark}

\subsection{Existence of gluing morphisms} \label{ss:gluing}
The proof of the existence of a gluing morphism as stated in Proposition
\ref{P:existsGluing} is a direct consequence of the last statement in
the following lemma.

\begin{lemma}\label{L:prop-proj}\samepage
 With our
 hypotheses
 on $\cat$, we have
 \begin{enumerate}\itemsep=0pt
 \item[$1.$] For any non-zero morphism \smash{$P\tto f\unit$} where $P$ is
 projective, there exists
 an
 epimorphism \smash{$\wt f\colon P\to P_\unit$}
 such that
 \smash{$f=\ve\circ\wt f$}.
 \item[$2.$]
 $P_\unit\simeq P_\unit^*$, i.e., $\cat$ is unimodular.
 \item[$3.$] $\Hom_\cat(P_\unit,\unit)=\FK\ve$ and $\Hom_\cat(\unit,P_\unit)=\FK\eta$.
 \item[$4.$] $\Lambda_{P_\unit}=\eta\circ\ve$.
 \item[$5.$] For any $f\in\End(P_\unit)$, $f$ is nilpotent if and only if
 $\Lambda_{P_\unit}\circ f=0$.
 \item[$6.$] For any non-zero endomorphism $f\in \End_\cat(P_\unit)$, there exists
 $g\in \End_\cat(P_\unit)$ with $g\circ f=\Lambda_{P_\unit}$.
 \end{enumerate}
\end{lemma}
\begin{proof}
 First remark that since $\FK$ is algebraically closed (so
 $\FK$ is the unique finite-dimensional division $\FK$-algebra) and
 since $P_\unit$ is indecomposable
 and $\cat$ is idempotent complete then the $\FK$-algebra
 $\End_\cat(P_\unit)$ is indecomposable. By~\cite[Theorem~3.2.2]{DK94},
 this implies that $\End_\cat(P_\unit)$ is local, i.e.,
 $\End_\cat(P_\unit)/J$ is a division algebra where $J$ is the
 Jacobson radical. Hence $\End_\cat(P_\unit)=\FK\Id\oplus J$. Then,
 by~\cite[Proposition 3.1.9]{DK94}, $J$ is a nilpotent ideal which,
 in this case, is formed by the nilpotent endomorphisms of $P_\unit$.
 \begin{enumerate}\itemsep=0pt
 \item For any nilpotent $n\in\End_\cat(P_\unit)$, if
 $\ve\circ n\neq 0$, then $\ve\circ n$ is an epimorphism and since
 $P_\unit$ is projective, $\ve$ factors through it:
 $\ve=\ve\circ n\circ g$ for some $g\in\End_\cat(P_\unit)$. But
 $n\circ g$ belongs to~$J$ which is nilpotent thus
 \smash{$\ve=\ve\circ(n\circ g)^{\dim_\FK(J)}=0$}
 and we have a contradiction. Then the kernel of
 $\ve\circ\_
 \colon \End_\cat(P_\unit)\to\Hom_\cat(P_\unit,\unit)$
 contains $J$ which is a maximal left
 ideal so it is equal to $J$. Now if $\ve'\colon P\to\unit$ is non-zero, then it is an epimorphism and since $P$, $P_\unit$ are
 projective, there exist \smash{$P\tto fP_\unit$}, \smash{$P_\unit\tto gP$} such that
 $\ve'=\ve f$, $\ve=\ve'g=\ve fg$. In particular, since $fg$ is not
 nilpotent, it is an isomorphism and so $f$ is an epimorphism which
 is split since $P_\unit$ is projective.

 \item Apply the previous to \smash{$P_\unit^*\tto{\eta^*}\unit$} implies
 that $P_\unit$ is a direct summand of $P_\unit^*$ which is
 indecomposable so $P_\unit\simeq P_\unit^*$.

 \item If $\ve'\colon P_\unit\to\unit$ is non-zero, then there is a split
 isomorphism \smash{$P_\unit\tto f P_\unit$} such that ${\ve'=\ve f}$. Write
 $f=\lambda\Id+n$ with $n$ nilpotent and $\lambda\in\FK$, then
 $\ve'=\lambda\ve+\ve n=\lambda\ve$. So,
 ${\Hom(P_\unit,\ve)\hspace{-0.9pt}=\hspace{-0.9pt}\FK\ve}$. Then
 $\Hom_\cat(\unit,P_\unit)\simeq\Hom_\cat(P_\unit^*,\unit)\simeq\Hom_\cat(P_\unit,\unit)$ also
 has dimension $1$ so it is generated by $\eta$.
 \item This follows since $\{\ve\}$ and $\{\eta\}$ are dual basis.
 \item From the proof of (1), we have for any
 $f\in\End(P_\unit)$, $f$ is nilpotent if and only if
 $\ve\circ f=0$. Since $\eta$ is a monomorphism, $f$ is nilpotent
 if and only if $\eta\circ \ve\circ f=0$.
 \item The symmetric pairing $(f,g)\mapsto \mt_{P_\unit}(fg)$ is
 non-degenerate on $\End_\cat(P_\unit)$. From (5),
 $\FK \Lambda_{P_\unit}$ is the orthogonal of $J$. Recall that $J$
 is nilpotent and let $k=\max\{n\in\N:J^nf\neq0\}$ and let
 $g\in J^k$ be such that $gf\neq0$. Then $Jgf=0$ thus $gf$ is
 orthogonal to $J$ and $gf\in\FK^* \Lambda_{P_\unit}$. Up to
 rescaling $g$, we thus have $gf=\Lambda_{P_\unit}$.
\hfill\qed
\end{enumerate}
\renewcommand{\qed}{}
\end{proof}

\subsection{Existence of m-traces and chromatic morphisms}\label{ss:coend}

 In this subsection, we discuss some sufficient conditions for the existence of chromatic morphisms as well as some cases when $\cat$ is chromatic compact.

{\bf Tensor categories.} In~\cite{EGNO}, a \emph{tensor category}~$\ttt$ is defined as a locally
finite $\FK$-linear abelian rigid monoidal category with
$\End_\ttt(\unit)\cong \FK$ where $\FK$ is an algebraically closed
field. A category is \emph{unimodular} if its unit object has a
self dual projective cover. If $\ttt$ is unimodular and pivotal, then
it has a unique (up to scalar) right m-trace on $\Proj$ (see~\cite[Corollary
5.6]{GKP22}). When $\ttt$ is ribbon this right m-trace is an m-trace.

In~\cite{CGPV2023b}, a more general notion of ``left'' and ``right'' chromatic morphisms is defined and proved to exist for general rigid finite tensor categories. We provide a sketch of proof for the special case of interest below.

\begin{theorem}
If $\ttt$ is a finite, unimodular, ribbon tensor category then it admits a chromatic morphism, and the notions
of factorizable and twist non-degenerate coincide with the usual ones, given for example in~{\rm \cite{DRGGPMR2022}}.
\end{theorem}
\begin{proof}[Sketch of proof.]
The finite
tensor category $\ttt$ has a coend
\smash{$\coend=\int^{V\in\cat}V^*\otimes V$} with dinatural transformations
$i_V\colon V^*\otimes V\to \coend$. The coend is a Hopf algebra object in
$\cat$ and every object $V\in\ttt$ has a structure of right
$\coend$-comodule given by
\[
\ci_V=\Bigl(V\tto {\lcoev_V\otimes\Id}V\otimes V^*\otimes V\tto
{\Id\otimes i_V}V\otimes\coend\Bigr)
\]
compatible with the monoidal
structure: the product $m_{\!\coend}\colon \coend\otimes\coend\to\coend$ is
used to define the coaction on a tensor product. Moreover, the coend
is known to have (unique up to a scalar) right integrals
$\lambda\colon \unit\to \coend$ and $\Lambda\colon \coend\to\unit$ with
$\Lambda\circ\lambda=\Id_\unit$ (since $\ttt$ is unimodular, $\Lambda$
is a two-sided integral). It is shown in~\cite{CGPV2023b}
(see Lemma 3.2 and take into account that in the unimodular case $\alpha=\unit$) that for a
good choice of $\Lambda$, we have for any projective $P\in\ttt$:
\begin{equation}
 \label{eq:Lambda_coend}
 (\Id_P\otimes\Lambda)\circ\ci_P=\Lambda_P,
\end{equation}
where $\Lambda_P$ is defined in equation~\eqref{eq:cop} with the
copairing of the m-trace. Let $G$ be a projective generator, letting $\epsilon_G\colon G\to \unit$ be any surjective morphism, then the
map $i_G$ is an epimorphism thus the map
$\lambda\circ\ve_G\colon G\to\coend$ factors through it and there exists a
map $f_\lambda\colon G\to G^*\otimes G$ such~that%
\begin{equation}
 \label{eq:flambda}
 \lambda\circ\ve_G=i_G\circ f_\lambda .
\end{equation}
 Let
\smash{$\wt f_\lambda=\bigl(\rev_G\otimes\Id_G\bigr)\circ(\Id_G\otimes f_\lambda)$} and
fix any map $e\colon G\to G\otimes G$ such that
$(\ve_G\otimes\Id_G)\circ e
=\Id_G$. Then the map
\begin{equation}
 \label{eq:chrcoend}
\chr=\bigl(\wt f_\lambda\otimes\Id_G\bigr)\circ(\Id_G\otimes e
)
\end{equation}
is a chromatic morphism. Indeed, the integral property of $\lambda$ is
that
$m_\coend\circ(\lambda\otimes \Id_\coend)=\ve_\coend\otimes \lambda$
where $m_\coend$ and $\ve_\coend$ are the product and counit of
$\coend$. Then we make a graphical proof:
\begin{align*}
\epsh{fig19}{9ex}
\putc{79}{50}{${\chr}$}
\putc{22}{50}{$\ms{\Lambda_{G\otimes G^*}}$}
&=\epsw{fig19-1}{15ex}
\putc{79}{28}{${\chr}$}
\putc{42}{66}{$\ms{\Lambda_{G\otimes G^*}}$}
\begin{array}{c}{\text{\tiny\eqref{eq:Lambda_coend}}}\\=\\{\text{\tiny\eqref{eq:chrcoend}}}\end{array}\epsw{fig19-2}{14ex}
\put(-10.5ex,-4.8ex){$\ms{f_\lambda}$}
\put(-7.5ex,-9ex){$\ms{e}$}
\put(-6.5ex,10.2ex){$\ms{\Lambda}$}
\putc{58}{50}{$\ms{\ci}$}\putc{31}{50}{$\ms{\ci}$}
\putc{58}{78}{$\ms{m_{\!\coend}}$}
\stackrel{\eqref{eq:flambda}}=\epsw{fig19-3}{14ex}
\putc{49}{09}{$\ms{e}$}\putc{34}{29}{\ms{\ve_G}}\putc{24}{52}{\ms{\lambda}}
\putc{58}{78}{$\ms{m_{\!\coend}}$}\putc{58}{50}{$\ms{\ci}$}
\putc{58}{94}{$\ms{\Lambda}$}\\
&=\epsw{fig19-4}{17.5ex}
\putc{62}{11}{$\ms{e}$}\putc{51}{33}{\ms{\ve_G}}\putc{06}{57}{\ms{\lambda}}
\putc{68}{58}{$\ms{\ci}$}
\putc{06}{76}{$\ms{\Lambda}$}\putc{71}{94}{\ms{\ve_{\!\coend}}}
=\epsh{fig20}{8ex}\putrc{0}{90}{$\ms{G}$}\putlc{100}{90}{$\ms{G}$}\ .
\end{align*}
The first equality comes from the duality property of $\Omega$ (see Lemma~\ref{P:Omega-nat}); the second is expressing~$\Lambda_{G\otimes G^*}$ as the coaction of $\Lambda$. This coaction on a tensor product is given by the product of the coactions. The third equality is the defining property of $f_\lambda$ and the fourth is the integral property of $\lambda$.

The coend in a braided tensor category $\ttt$ also gives an alternative description of the morphisms
$\Delta_0$ and $\Delta_\pm$ as follows: Let $\theta\colon \coend\to\coend$ be defined by
$\theta\circ i_V=i_V\circ(\Id_{V^*}\otimes\theta_V)$, note that~$\theta$ is not $\theta_\coend$. Then
$\Delta_{\pm}=\ve_\coend\circ\theta^{\pm1}\circ\lambda$. Let
$\omega\colon \coend\otimes\coend\to\unit$ be the Hopf pairing defined by
\[\omega\circ(i_U\otimes i_V)=\bigl(\lev_U\otimes\lev_V\bigr)\circ(\Id_{U^*}\otimes(c_{V^*,U}\circ
c_{U,V^*})\otimes\Id_V)\] and let
$\Delta_0=\omega\circ(\lambda\otimes\Id_\coend):\coend\to\unit$. Then
for any projective object $P$, $\Delta_0^P=(\Id_P\otimes\Delta_0)\circ\ci$, indeed
\begin{gather*}
\begin{split}
&\epsh{fig21-0}{9ex}\putc{50}{50}{$\ms{\Delta_0^P}$}\putlc{60}{90}{$\ms{P}$}
=\epsh{fig21-1}{18ex}\putc{63}{28}{$\ms{\chr_P}$}\putlc{9}{50}{$\ms{G}$}
=\epsh{fig21-2}{18ex}\putc{54}{24}{$\ms{\chr_P}$}\putc{63}{44}{$\ms{\ci}$}
\putc{30}{56}{$\ms{\ci}$}\putc{86}{87}{$\ms{\omega}$}
=\epsh{fig21-3}{18ex}\putc{55}{17}{$\ms{e}$}\putc{62}{40}{$\ms{\ci}$}
\putc{15}{41}{\ms{\ve_G}}\putc{15}{62}{\ms{\lambda}}\putc{86}{87}{$\ms{\omega}$}
=\epsh{fig21-4}{17ex}\putc{32}{32}{$\ms{\ci}$}\putc{79}{77}{$\ms{\Delta_0}$}\;\; .
\end{split}
\end{gather*}
Thus, $\ttt$ is factorizable if and only if
$\Delta_0=\Lambda$ which is equivalent (see~\cite[Lemma 2.7]{DRGGPMR2022})
to the non-degeneracy of $\omega$. We thus recover the usual notion
of factorizable and twist non-degenerate given for example in
\cite{DRGGPMR2022}.
\end{proof}

{\bf Hopf algebras.} A particular example is when $\ttt$ is the category of finite-dimensional left
modules over a finite-dimensional unimodular ribbon Hopf algebra
$H$. By~\cite[Theorem~1]{BBG17b}, it has a non-degenerate left
m-trace on $\Proj$, which is also a right m-trace since $H$ is ribbon.
The module $H$ with its left regular action is a projective generator
and the map $\Lambda_{H \otimes H^*}$ is the action of the two sided
cointegral. Indeed, as a Hopf algebra, $\coend$ can be identified
with $H^*$
where the multiplication is
 twisted by the braiding of $\ttt$.
An explicit formula for a chromatic morphism $\chr\colon H\otimes H\to H\otimes H$ is
given by
\begin{equation}\label{E:FormulaChHopf}
x\otimes y\mapsto \lambda\bigl(S\bigl(y_{(1)}\bigr) g x\bigr) y_{(2)} \otimes y_{(3)},
\end{equation}
where $g$ is the pivotal element, $\lambda$ is the right integral, $S$
is the antipode, and \smash{$y_{(1)} \otimes y_{(2)} \otimes y_{(3)}$} is the
double coproduct of $y$ (see~\cite[Lemma 6.3]{CGPT20}). In this case
the morphism $\Delta_0^P$ is given by the action of the element
$\Delta_0=\lambda \otimes\Id_H(R_{21}R_{12})$, where $R\in H\otimes H$
is the $R$-matrix of $H$. Note that $\Delta_0$ being non-zero is one
of the two conditions for $H-{\rm mod}$ to be called 3-modular in
\cite[Theorem~3.7.3]{L94}.

When working with Hopf algebras, it may
seem more natural to express everything in terms of $H$ and to avoid
mentioning $P_\unit$. One can define a \emph{gluing morphism for $H$}
as an endomorphism $\gm_H\colon H\to H$ satisfying
\smash{$\gm_H\circ \Delta_0^H=\Lambda_H$}. Writing $P_\unit$ as an idempotent
$i \circ p$ in $H$, it is easy to see that the existence of a gluing
morphism $\gm_H$ for $H$ implies the existence of a gluing morphism
$\gm:= p\circ \gm_H\circ i$. Similarly, the existence of $\gm$ implies
that of $\gm_H:= i\circ\gm\circ p$
(there is up to automorphism a unique indecomposable summand of
$P_\unit$ in $H$ which contains the unique left ideal~$\FK\Lambda$ isomorphic to $\unit$).
However, the
existence of $\gm_H$ is a priori not equivalent to $\Delta_0$ being
non-zero.

{\bf Fusion categories.} Recall that the global dimension of a ribbon fusion category $\cat$ is
the sum of the squares of the dimensions of its simple
objects. It is
shown in~\cite[Theorem~7.21.12]{EGNO} that it is non-zero when $\FK$
has characteristic zero.
\begin{proposition}\label{prop:semisimpleCompact}
 A ribbon fusion $($possibly non-modular$)$ category $\cat$ is chromatic non-degen\-er\-ate if and only if it has non-zero
 global dimension. In this case, it is moreover chromatic compact.
\end{proposition}
\begin{proof}
 Let $
\{S_i\}_{i\in I}$ be a set of representatives of the isomorphism
classes of simple objects of~$\cat$.
Observe that here $G=\oplus_{i\in I} S_i$ is
a generator of $\Proj=\cat$ and
the quantum trace $\mt=\operatorname{qTr}_\cat$ is
a~non-degenerate m-trace on $ \cat$ (and the other choices differ by an invertible scalar). It follows~that
\[
\biggl\{x_i=\frac{1}{\qdim(S_i)}\lcoev_{S_i}\biggr\}_{i\in I} \qquad\text{and}\qquad \bigl\{y_i=\rev_{S_i}\bigr\}_{i\in I}
\]
 are dual bases of $\Hom_\cat(\unit,G \otimes G^*)$ and $\Hom_\cat(G\otimes G^*,\unit)$, respectively. Using the expansion
\smash{$\Omega_{G\otimes G^*}=\sum_{i\in I} x_i\otimes_\kk y_i$}, it is straightforward to check that
\[\chr_P=(\oplus_{i\in I} \qdim(S_i)\Id_{S_i})\otimes \Id_P\]
is a chromatic morphism for $G$ based on $P$ (it is essentially the
Kirby color tensor the identity of $P$).

Let $d(\cat)= \sum_{i\in I} \qdim(S_i)^2$ be the global dimension of
$\cat$.
As $P_\unit=\unit$, the
morphism \smash{$\Delta_0^{P_\unit}$} is multiplication by $d(\cat)$, and indeed $\CC$ chromatic non-degenerate if and only if it has non-zero
 global dimension. In this case, a gluing morphism
is given by \smash{$\frac{1}{d(\cat)}{\rm Id}_\unit$}, and
$\cat$ is chromatic compact with $\zeta = d(\cat)$.
\end{proof}

In particular, the parameter $\zeta$ we called global dimension coincides with the usual notion in the semisimple case and for the usual quantum trace.

{\bf Symmetric categories.} It seems that having semi-simple M\"uger center is a good criterion for being chromatic compact in characteristic 0. On one extreme, we have seen above that if~$\cat$ has trivial M\"uger center (i.e., is factorizable) then it is chromatic compact. We consider the other extreme, the symmetric case. 
\begin{proposition}\label{P:symcase}
If $\cat$ is a symmetric monoidal tensor category and ${\rm char}(\FK)=0$, then the following are equivalent: $\cat$ is chromatic non-degenerate, $\cat$ is chromatic compact, $\cat$ is semi-simple.
\end{proposition}
\begin{proof}
First, let us observe that $\cat$ admits a fiber functor. By Deligne's theorem, see~\cite[Theorem~9.11.4]{EGNO}, it is enough to check that $\cat$ has sub-exponential growth. Find $C\in \mathbb{N}$ such that $G\otimes G \subseteq G^{\oplus C}$ and let $L$ be the length of $G$. Then the length of $G^{\otimes n}$ is at most $(C.L)^n$ and~$\cat$ has sub-exponential growth.

Using Tannakian reconstruction, $\cat \simeq H-{\rm mod}^{fd}$ for some Hopf algebra $H$. (Note that, despite the fact that $\cat$ is symmetric, it is not clear that $H$ is cocommutative as it might be a~super-Hopf algebra.)
We use the chromatic morphism from equation~\ref{E:FormulaChHopf} (it does not matter which chromatic morphism we use by Lemma~\ref{prop:redtoblue}). We can ignore the double braiding in the definition and compute \smash{$\Delta_0^H(1)=\lambda(1)1$}. If $\cat$ is chromatic non-degenerate, it has to be non-zero, hence $\lambda(1) \neq0$. This is equivalent to semisimplicity of $H^*$ by~\cite{Radford} which is equivalent to semisimplicity of $H$ in characteristic~0 \cite{RadfordLarson}.

 To conclude, we have seen in Proposition~\ref{prop:semisimpleCompact} that semi-simple implies chromatic compact, and it is immediate that chromatic compact implies chromatic non-degenerate.
\end{proof}

\section{Skein relations and algebraic properties}
Throughout this paper, all manifolds are smooth and oriented and all
diffeomorphisms are orientation preserving, unless otherwise stated.

\subsection{Blue, red and green graphs}
In this article, we will consider admissible skeins in closed 3-manifolds. It is common to represent the topology of the 3-manifold $M$ by links in $S^3$ using Kirby calculus. In our setting, this might cause confusion because there are now two kinds of links: the skeins that live in the 3-manifold, and the links depicting its topology. We take the convention to represent the first ones in blue, and the latter in green. We will actually need a third color: the red components. They are a shorthand for a possibly complicated skein (generalizing the Kirby color) and can be turned into actual skeins using the ``red-to-blue modification". It is easy to get the green and red components mixed up, because the red skeins are used to depict the effect of surgery on skeins. Let us stress here that
\begin{itemize}\itemsep=0pt
\item[{}] \textit{blue and red graphs and links} are decorations embedded in some 3-manifold, whereas
\item[{}] \textit{green links} represent the 3-manifold itself, using Kirby calculus.
\end{itemize}

We consider graphs whose edges are colored by one of the colors blue, red or green each representing different structures: the blue part will be a $\cat$-colored ribbon graph with coupons in the sense of Turaev.
 The red part is an unoriented framed
link in $M$ which is not $\cat$-colored. Graphs made of the disjoint union of a blue $\cat$-colored ribbon graph and a red set of non-oriented framed circles are called \emph{bichrome graphs}, see~\cite{CGPT20}.
The green is an unoriented framed
link which is used as a notation for the topology of $M$. This means that $M$
is identified with the result of the $S^1$-surgery on the green link, see Figure~\ref{fig:OrangeAndRedColor}.
\begin{figure}[t]
$$
 \epsh{fig22-3}{12ex}
 \putr{-2}{40}{$\ms{V}$}
 \putc{71}{68}{${f}$}
 \qquad\qquad\qquad\qquad\qquad
 \epsh{fig22-4}{12ex}
 \putr{-2}{40}{$\ms{V}$}
 \putc{71}{68}{${f}$}
$$
\caption{The graph on the left is a ribbon graph $T$ in $S^2\times S^1$ where the green circle represents the topology of $S^2\times S^1$. The graph on the right is a bichrome graph in~$S^3$.
}\label{fig:OrangeAndRedColor}
\end{figure}
Sliding a blue or red edge on an green circle should be thought as
an isotopy in $M$.
For a~disconnected 3-manifolds $M$, we will
use the symbol~$\sqcup$ to separate the different components of~$M$.

A bichrome graph in a manifold
$M$ is admissible if
every connected component of $M$ contains a blue edge colored by
a projective object. A \emph{red to blue modification} of a bichrome graph
is an~operation in an annulus given by
\begin{equation}
 \label{eq:redtoblue}
 \epsh{fig3}{15ex}\put(1,15){\ms{P}}\longrightarrow
 \epsh{fig7}{15ex}\put(-16,-2){{${{\chr}}_P$}}
 \put(-27,16){\ms G} \put(-2,25){\ms P} \;,
\end{equation}
where $\chr_P$ is any chromatic morphism based on a projective object $P$.
Here we allow the $P$-colored edge to be replaced by several parallel
strands with at least one colored by a projective object (indeed if $P\in \Proj$ and $V,W\in \cat$ are any objects, then $V\otimes P\otimes W\in \Proj$).

\subsection{Skein relations and admissible skein modules}
For this subsection, let $M$ be a manifold of dimension $n=2,3$.
An \emph{admissible ribbon graph} in~$M$ is a
$\cat$-colored ribbon graph in $M$ where each connected component of
$M$ contains at least one edge colored with a projective object.

A \emph{box} in $M$ is a proper embedding of $[0,1]^{ n} $ in $M$.
Its top and bottom faces are $[0,1]^{n-1}\times \{1\}$ and
$[0,1]^{n-1}\times \{0\}$, respectively. A ribbon graph is
\emph{transverse} to a box if it intersects the box transversally on
the lines $[0,1] \times \{0,1\}$ if $n=2$ and
$[0,1] \times \bigl\{\frac12\bigr\} \times \{0,1\}$ if $n=3$.

Given admissible graphs $\Gamma_1, \dots,\Gamma_k$ in $M$ and $a_1, \dots,a_k\in \FK$, the linear combination $a_1\Gamma_1 + \cdots + a_n\Gamma_n$ is a \emph{projective skein relation} if there is a box $B$ in a
connected component $M_0$ of $M$ and admissible graphs $\Gamma'_1, \dots,\Gamma'_k$ in $M$ such that
\begin{itemize}\itemsep=0pt
\item $\Gamma'_i$ is isotopic to $\Gamma_i$ (as a $\cat$-colored graph in $M$) for $1 \leq i \leq k$,
\item the $\Gamma'_i$s coincide outside $B$: $\Gamma'_i\cap (M \setminus B)=\Gamma'_j\cap (M \setminus B)$ for all $1 \leq i,j \leq k$,
\item $a_1 F\bigl(\Gamma'_1 \cap B\bigr) + \cdots + a_k F\bigl(\Gamma'_k \cap B\bigr)=0$ (as a morphism in $\cat$),
\item each $\Gamma'_i$
has a
$\Proj$-colored edge in $M_0$ not entirely contained in the box $B$.
\end{itemize}
Two linear combinations are \emph{projective skein
 equivalent} if their difference is a projective skein relation.

\begin{definition}
 The \emph{admissible skein module} $\TSkein_{\cat}(M)$ is
 the $\FK$-span of admissible ribbon graphs in $M$ modulo
 the span of projective skein relations.\footnote{Since we only work with the ideal of projective objects, in this paper, we denote by $\TSkein_\cat$ the skein module $\mathcal{S}_{\Proj}$ of~\cite{CGP22}.}
\end{definition}

\begin{theorem}
For any manifold $M$ of dimension $2$ or $3$, the vector space $\TSkein_\cat(M)$ is finite-dimensional.
\end{theorem}
\begin{proof}
In the case of surfaces this theorem is~\cite[Theorem~5.10]{CGP22}. Let $M$ be a connected 3-manifold, decomposed into one $0$-handle, $g$ index $1$ handles, and some index $2$ and $3$ handles.
Then $\TSkein_\cat(M)$ is generated by admissible skeins in the genus $g$ handlebody $H_g$ formed by the handles of index $0$ and $1$. Then we conclude by observing that $H_g=\Sigma\times [-1,1]$ for some orientable surface $\Sigma$, therefore
\[
\dim\TSkein_\cat(H_g)=\dim\TSkein_\cat(\Sigma)<\infty.
\]
Finally, for non connected manifolds, the skein modules are tensor products of those of the connected components.
\end{proof}

By~\cite[Proposition 2.2]{CGP22}, the assignment $M\to \TSkein_{\cat}(M)$ extends to a functor $\TSkein_{\cat}\colon \Emb_n \to \Vect$ where $\Vect$ is the category of finite-dimensional vector spaces and $\Emb_n$ is the category whose objects are oriented $n$-dimensional manifolds and morphisms are isotopy classes of orientation preserving proper embeddings.

The following lemmas are from~\cite{CGPV2023a}, for completeness we restate them here in a slightly different language.
\begin{lemma}\label{prop:redtoblue}
Let $L$ be an admissible bichrome graph in $M$ and let $L_{1}$ and $L_2$ be two $\cat$-ribbon graphs each obtained by using red to blue modifications to change every red component of $L$ to blue. Then $L_1$ and $L_2$ are projective skein equivalent.
\end{lemma}
\begin{proof}
In this proof, we write $T \,\dot{=}\, T'$ if $T$ and $T'$ are projective skein equivalent graphs in~$M$. We need to show that two red to blue modifications of a red circle at
 different places with different chromatic morphisms are projective skein equivalent. Let $\chr_P\in\End_\cat(G\otimes P)$ and
 $\chr_Q\in\End_\cat(G'\otimes Q)$ be two chromatic morphisms based on $P$
 and $Q$, respectively, where $G$ and~$G'$ are projective generators.

Suppose that we have two modifications of the type ``red to blue'' which are made
on different strands and using opposite orientations for
the red circle. Then we have (with implicit summation):\vspace*{-1ex}
 \begin{gather*}
 \epsh{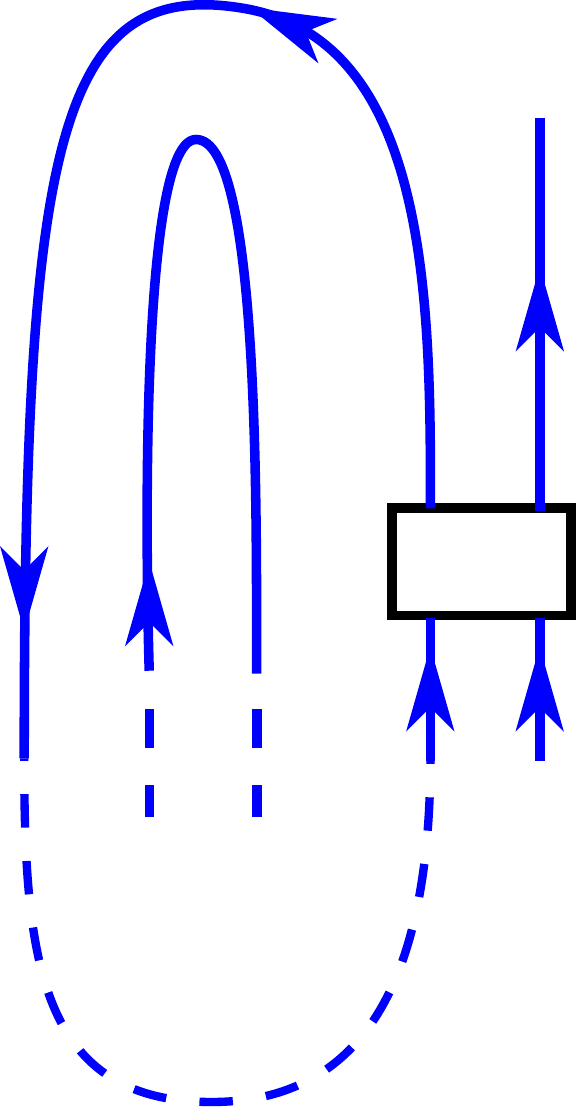}{24ex}
 \putc{84}{49}{$\ms{{{\chr}}_P}$}
 \putlc{31}{45}{$\ms{Q}$}
 \ \dot{=}\
 \epsh{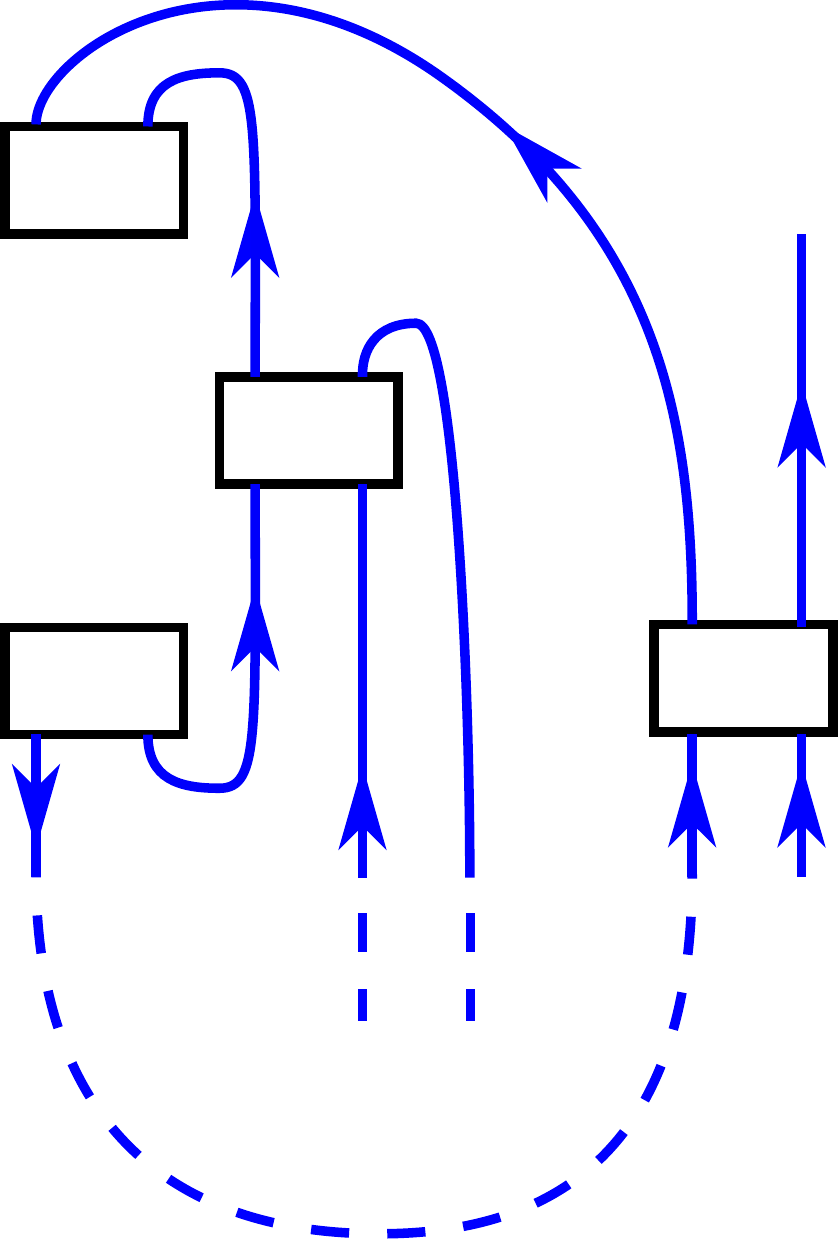}{24ex}
 \putc{11}{86}{$\ms{x_i}$}
 \putc{11}{45}{$\ms{x^i}$}
 \putc{36}{66}{$\ms{{{\chr}}_Q}$}
 \putc{89}{45}{$\ms{{{\chr}}_P}$}
 \ \dot{=}
 \epsh{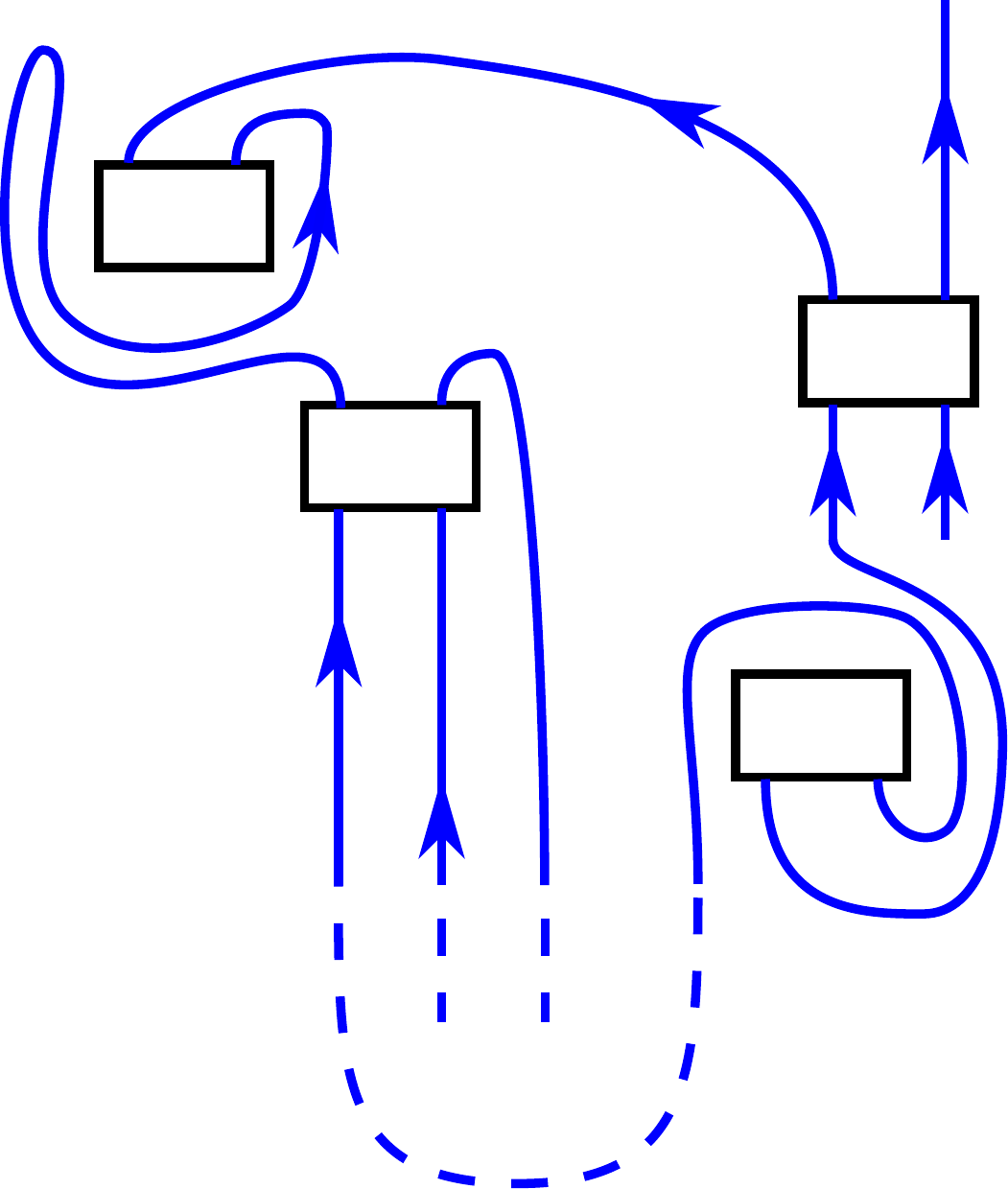}{24ex}
 \putc{17}{82}{$\ms{x_i}$}
 \putc{38}{62}{$\ms{{{\chr}}_Q}$}
 \putc{88}{70}{$\ms{{{\chr}}_P}$}
 \putc{81}{39}{$\ms{x^i}$}
 \ \dot{=}
 \epsh{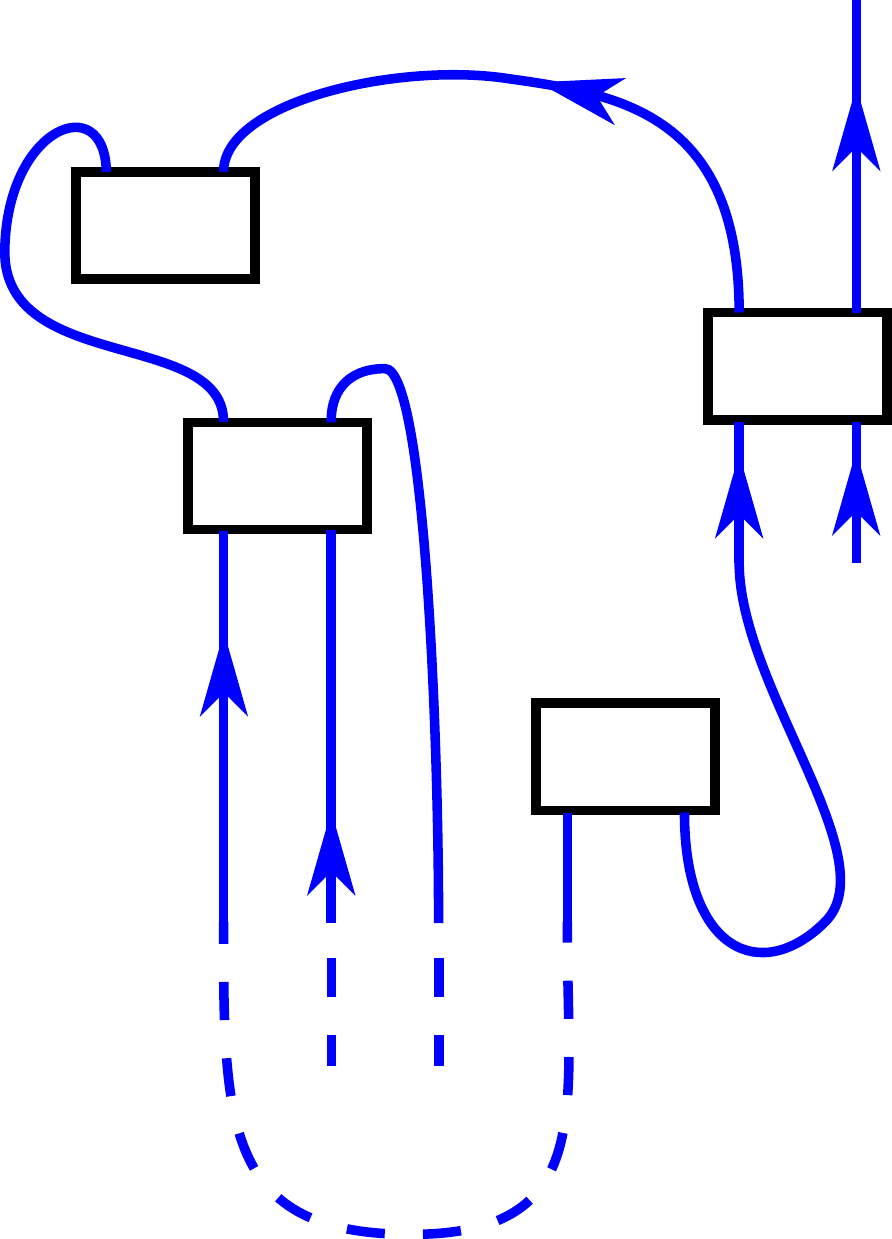}{24ex}
 \putc{18}{82}{$\ms{\wt x_i}$}
 \putc{31}{62}{$\ms{{{\chr}}_Q}$}
 \putc{90}{70}{$\ms{{{\chr}}_P}$}
 \putc{71}{40}{$\ms{\wt x^i}$}
 \ \dot{=}\
 \epsh{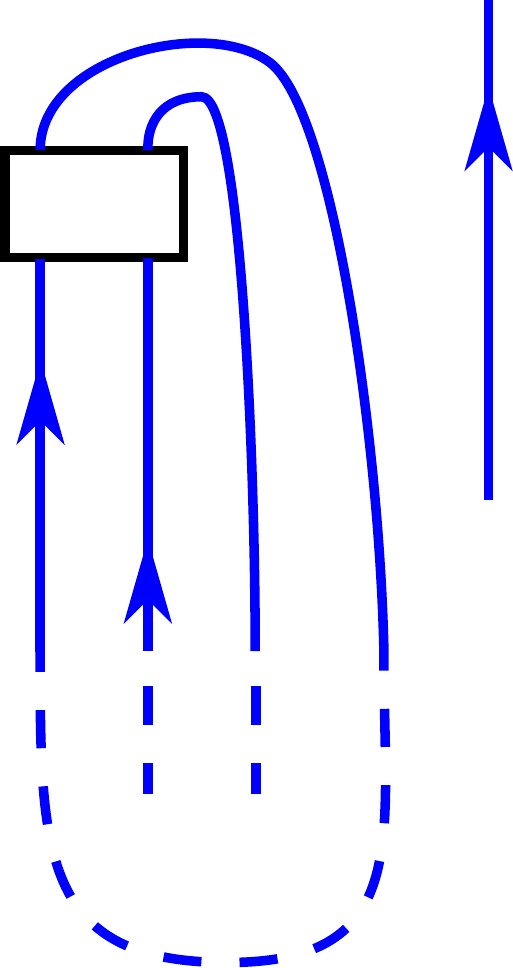}{18ex}
 \putc{18}{79}{$\ms{{{\chr}}_Q}$}
 \putlc{100}{86}{$\ms{P}$}
\end{gather*}
 where $\wt x_i$ and ${\wt x}^i$ are the dual basis obtained from
 $x_i$ and $x^i$ by the rotation property
 of Lemma~\ref{P:Omega-nat}. If the two modifications are happening on the same side of the red circle, then by the above argument they are both equal to a third modification happening on the opposite side of the red circle.
\end{proof}

\begin{lemma}\label{lemma:slidingOverRed}
Let $L$ be an admissible bichrome graph in $M$ and let $L'$ be the admissible bichrome graph obtained by sliding a red or blue edge of $L$ over a red circle of $L$ $($via a Kirby II move, see equation~\eqref{eq:Kirby2}$)$. If $L_1$ and $L_2$ are ribbon graphs obtained by applying a red to blue modification on each red component of~$L$ and~$L'$, respectively, then $L_1$ and $L_2$ are projective skein equivalent. \end{lemma}
\begin{proof}
We first consider the case of sliding a blue edge colored by
 $P\in\Proj$ on a red circle:
 \begin{align*}
 \epsh{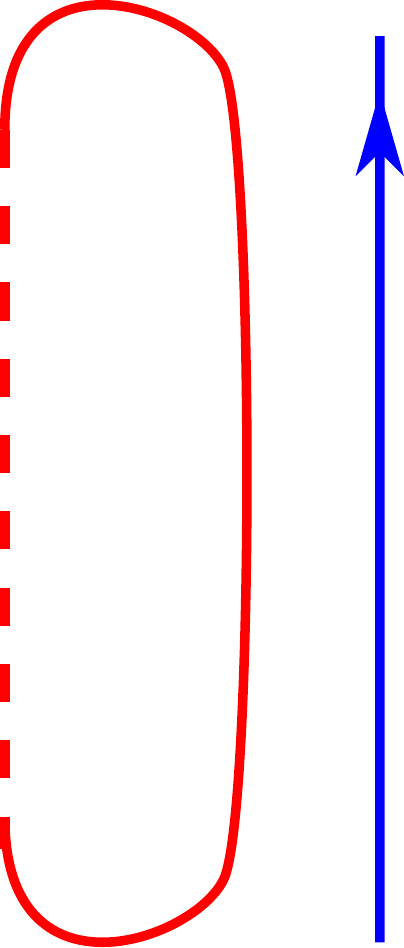}{16ex}
 \putlc{101}{85}{$\ms{P}$}
 &{}\;\;\dot{=}\;\;
 \epsh{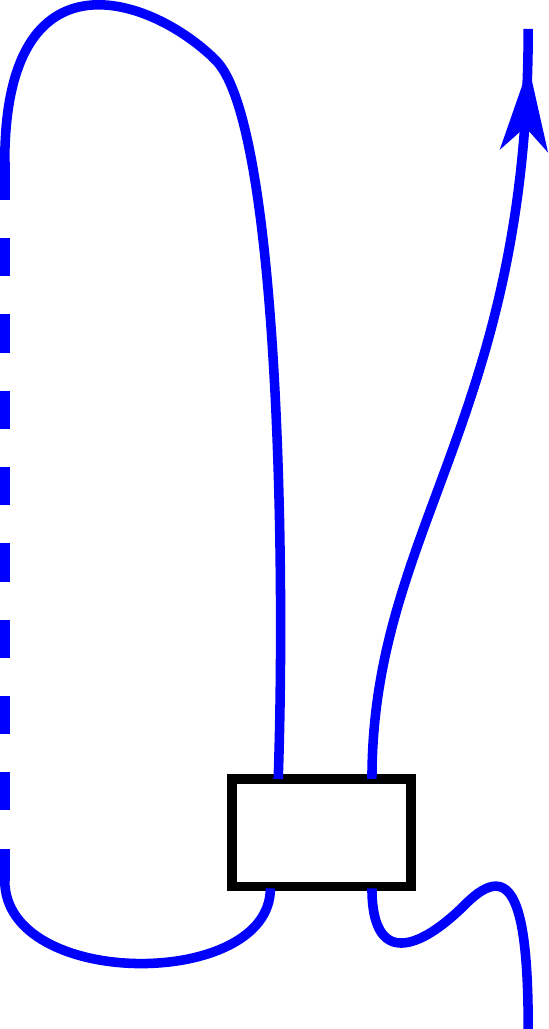}{18ex}
 \putc{58}{19}{$\ms{{{\chr}}_{P}}$}
 \dot{=}\;\;
 \epsh{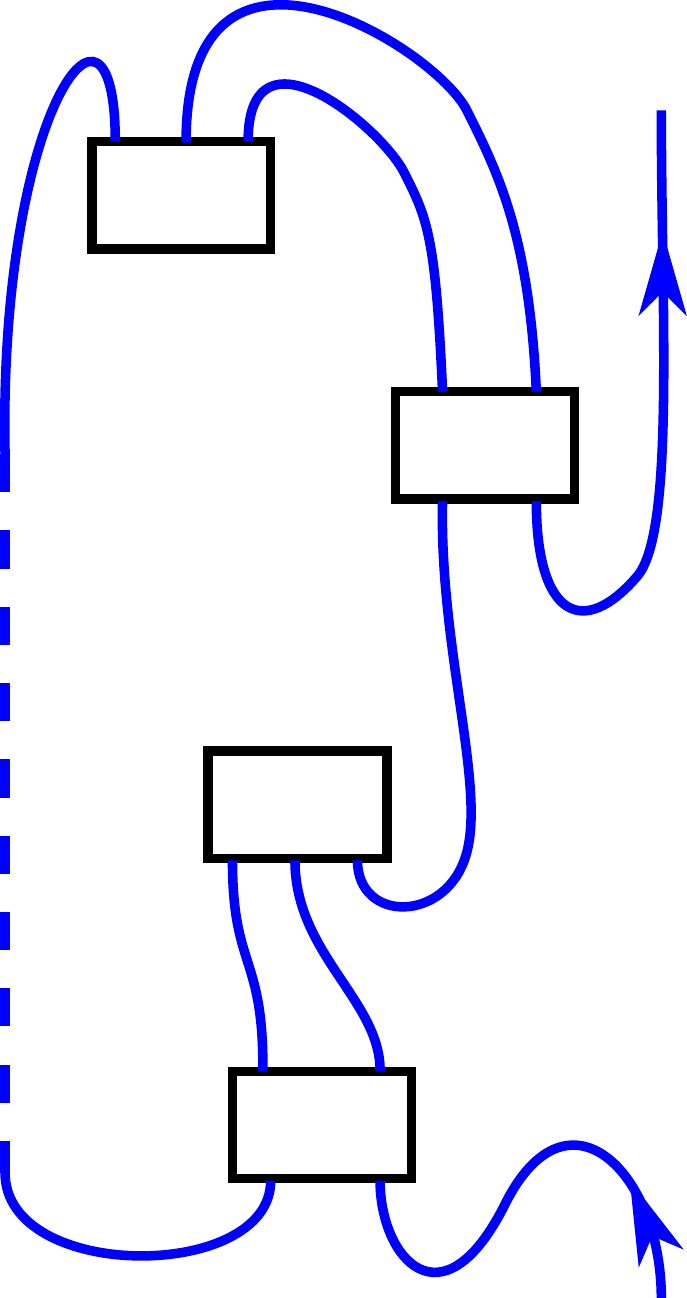}{24ex}
 \putc{26}{86}{$\ms{x_i}$}
 \putc{43}{38}{$\ms{x^i}$}
 \putc{70}{66}{$\ms{{{\chr}}_{P^*}}$}
 \putc{46}{13}{$\ms{{{\chr}}_{P}}$}
 \dot{=}\;\;
 \epsh{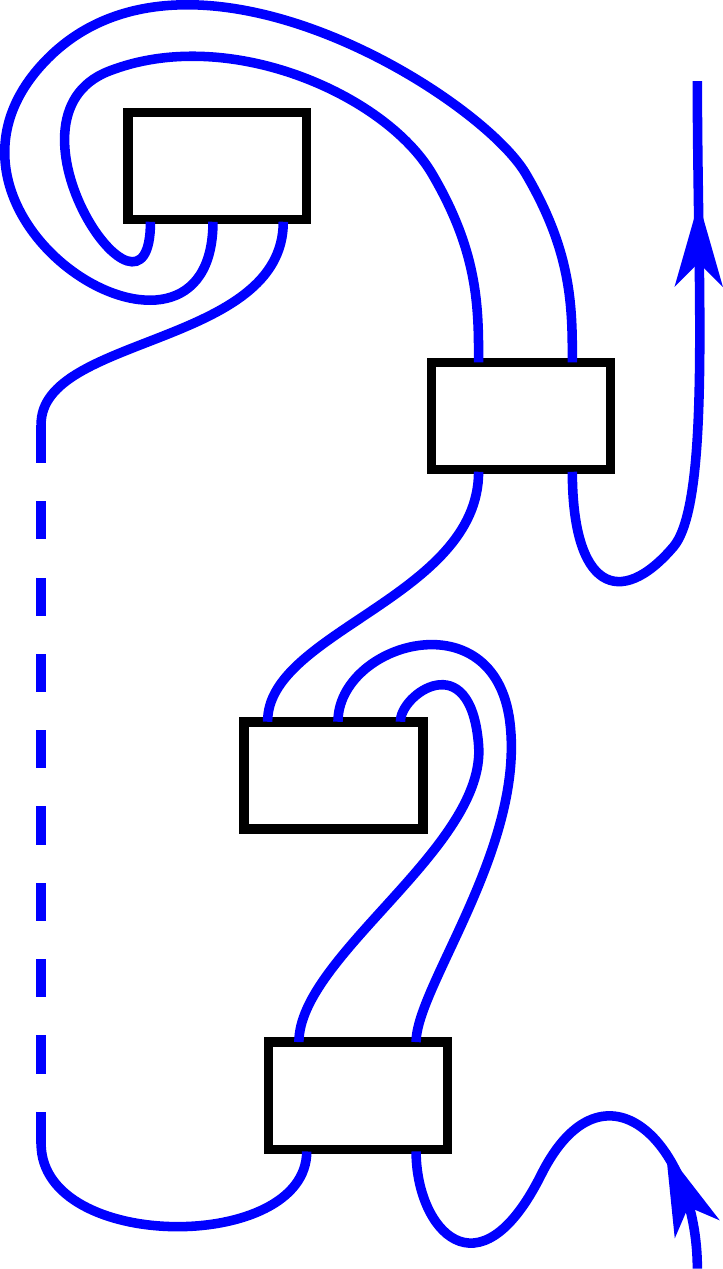}{24ex}
 \putc{30}{87}{$\ms{x^{*i}}$}
 \putc{72}{67}{$\ms{{{\chr}}_{P^*}}$}
 \putc{46}{39}{$\ms{{x^*}_i}$}
 \putc{49}{13}{$\ms{{{\chr}}_{P}}$}\\
 &{}\dot{=}\;\;
 \epsh{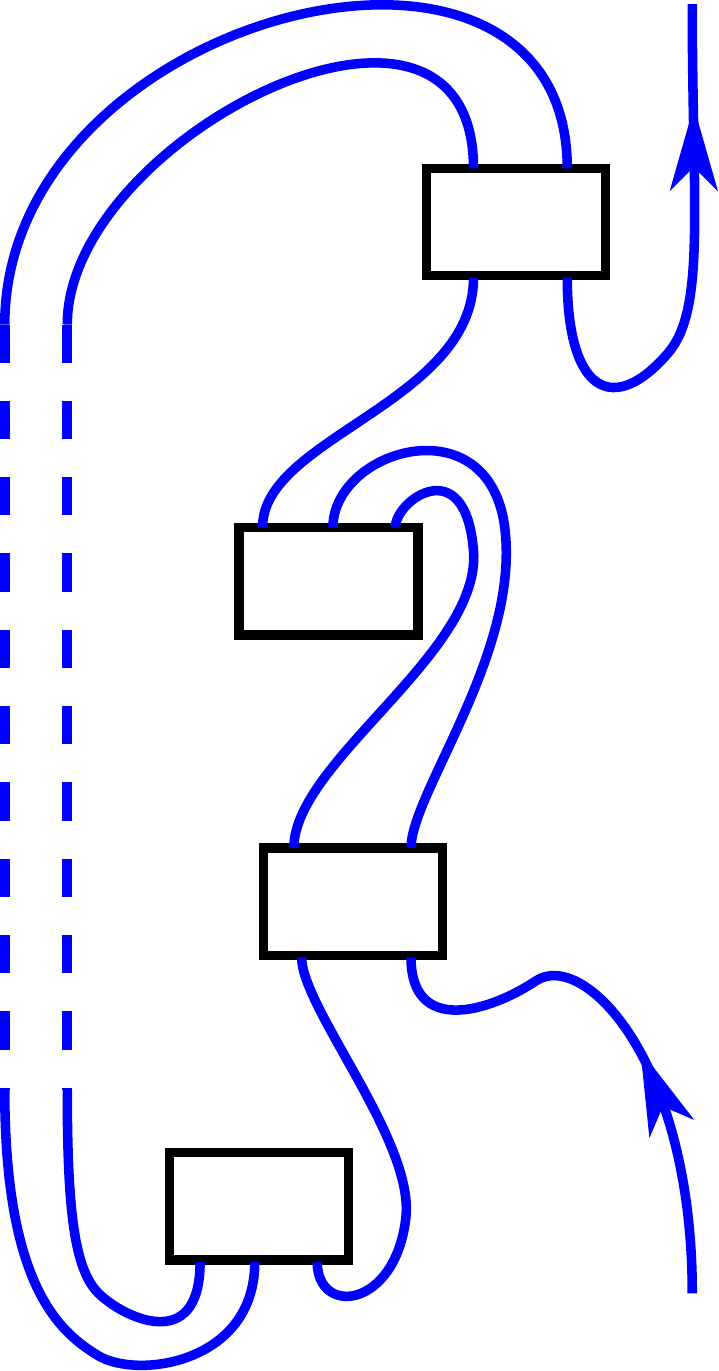}{24ex}
 \putc{36}{12}{$\ms{x^{*i}}$}
 \putc{48}{34}{$\ms{{{\chr}}_{P}}$}
 \putc{45}{58}{$\ms{{x^*}_i}$}
 \putc{72}{84}{$\ms{{{\chr}}_{P^*}}$}
 \dot{=}\;\;
 \epsh{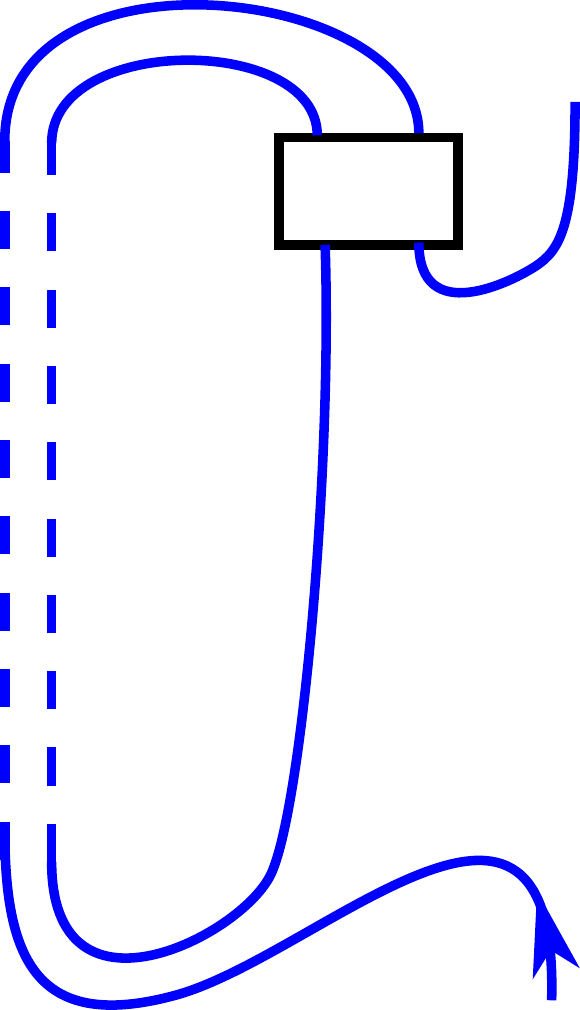}{18ex}
 \putc{64}{81}{$\ms{{{\chr}}_{P^*}}$}
 \dot{=}\;\;
 \epsh{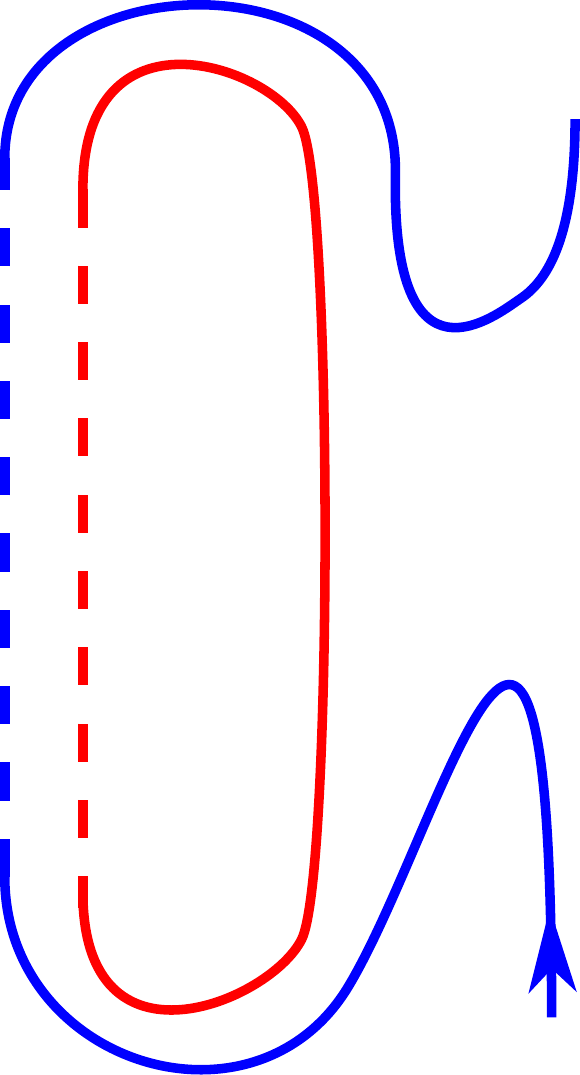}{16ex}
 \putrc{88}{11}{$\ms{P}$}\ ,
 \end{align*}
 where \smash{${x^*}_i$} and \smash{${x^*}^i$} are the dual basis defined by
 \smash{${x^*}^i=(x_i)^*\circ(\phi_{G}\otimes\Id_{P^*\otimes {G}^*})$} and
 ${x^*}_i=\bigl(\smash{\phi_{G}^{-1}}\otimes\Id_{P^*\otimes {G}^*}\bigr)\circ\smash{\bigl(x^i\bigr)^*}$. Notice this implies a red circle can be made blue then slid over a red circle. Now
 if we want to slide a blue edge colored by $V\in\cat$ we can fuse it with an edge colored with a projective $Q$ to this edge creating an edge colored by $V\otimes Q$ (which is projective) and two coupons. Then we can slide the $V\otimes Q$-colored edge as in the above computation. After moving one of coupons along the red circle to the other coupon, we can remove both coupons and then unslide the edge colored with $Q$.
\end{proof}

\subsection{Properties of the morphisms of a chromatic category}
Next we prove Lemma~\ref{L:Delta}.
\begin{proof}[Proof of Lemma~\ref{L:Delta}] By Lemma~\ref{prop:redtoblue}, we
 can give a meaning to the evaluation by $F$ of $F'$ of admissible
 ribbon graphs with red components as the value of any red to blue
 modification of it, see equation \eqref{eq:redtoblue}. Thus we have that
 \begin{gather*}
 \Delta_+\ve=F\bp{\epsh{fig22-5}{3ex}\ \epsh{fig22-7}{6ex}
 \putc{50}{81}{$\ms{\ve}$}\ }\in\FK\ve,\qquad
 \Delta_-\ve=F\bp{\epsh{fig22-6}{3ex}\ \epsh{fig22-7}{6ex}
 \putc{50}{81}{$\ms{\ve}$}\ }\in\FK\ve \qquad \text{and}\\
 \Delta^P_0=F\bp{\epsh{fig22-8}{8ex}\putlc{58}{90}{$\ms{P}$}}\in\End_\cat(P)
 \end{gather*}
 does not depend on the choice of any chromatic morphism.
\end{proof}

\begin{lemma}\label{L:Delta2}
 The morphisms \smash{$\Delta_0^P$} is
 a natural morphism in $P\in\Proj$ \big(that is
 \smash{$f\circ\Delta_0^P=\Delta_0^Q\circ f$} for any \smash{$P\tto fQ\in\Proj$}\big) and
 \smash{$\bigl(\Delta_0^P\bigr)^*=\Delta_0^{P^*}$} for any $P\in\Proj$.
\end{lemma}
\begin{proof}
 Since
 \[
 \Delta^P_0=F\bp{\epsh{fig22-8}{6ex}\putlc{58}{90}{$\ms{P}$}}\in\End_\cat(P),
 \]
 the morphism
 $\Delta^P_0$ is clearly natural in $P\in\Proj$. Finally, since the
 red circle is not oriented,
 \[\bigl(\Delta_0^P\bigr)^*=F\bp{\epsh{fig22-10}{8ex}\putlc{58}{10}{$\ms{P}$}}
 =F\bp{\epsh{fig22-11}{8ex}\putlc{58}{10}{$\ms{P}$}}
 =F\bp{\epsh{fig22-12}{8ex}\putlc{58}{10}{$\ms{P}$}}=\Delta_0^{P^*}.
\tag*{\qed}
\]
\renewcommand{\qed}{}
\end{proof}

\begin{lemma}\label{L:fact>twistnd}
 If $\cat$ is factorizable, then it is twist non-degenerate
 and $\Delta_+\Delta_-=\zeta$.
\end{lemma}
\begin{proof}
 We have
 \begin{equation*}
 \zeta\ve=\epsh{fig22-9}{5ex}\ \epsh{fig22-7}{6ex}
 \putc{50}{81}{$\ms{\ve}$}=\Delta_+\Delta_-\ve,
 \end{equation*}
 where the left equality is obtained by making the $1$-framed
 component blue, applying the factorizability condition and then using the defining property of the chromatic morphism. The right equality is
 obtained by sliding the 0-framed red unknot on the $1$-framed red
 one.
\end{proof}

\begin{proposition}
 The dual of a gluing morphism is conjugate to a gluing morphism by
 any isomorphism $P_\unit\simeq P_\unit^*$.
\end{proposition}
\begin{proof}
 For the existence: since \smash{$\Delta_0^{P_\unit}\neq0$} by Lemma~\ref{L:prop-proj}\,(6),
 there exist \smash{$\gm\in\End_\cat(P_\unit)$} such
 that \smash{$\Delta_0^{P_\unit}\gm=\Lambda_{P_\unit}$}. Since
 \smash{$\Delta_0^{P_\unit}$} is central,
 \smash{$\gm\Delta_0^{P_\unit}=\Lambda_{P_\unit}$} and taking the dual, we get
 \[
 (\Lambda_{P_\unit})^*=\Lambda_{P_\unit^*}=\bigl(\Delta_0^{P_\unit}\bigr)^*\gm^* =\bigl(\Delta_0^{P_\unit^*}\bigr)\gm^*
 \]
 (where the first equality is due to Lemma~\ref{P:Omega-nat} and the last to Lemma~\ref{L:Delta2}).
 Now if we conjugate by an isomorphism $\psi\colon P_\unit\to P_\unit^*$ we
 get \smash{$\Lambda_{P_\unit}=\psi\circ \Lambda_{P^*_\unit}\circ \psi^{-1}=\Delta_0^{P_\unit}\bigl(\psi^{-1}\gm^*\psi\bigr)$}, where the last equality follows by Lemma~\ref{L:Delta2}.
\end{proof}

\begin{lemma}\label{lem:invertibleg}
 The category $\cat$ has a gluing morphism which is an isomorphism of
 $P_\unit$ if and only~if%
 \[\Delta_0^{P_\unit}=\zeta\Lambda_{P_\unit} \qquad
 \text{for some scalar }\zeta\in\FK^* \text{$($i.e., iff $\cat$ is chromatic compact$)$.}\]
 In this case, \smash{$\zeta^{-1}\Id_{P_\unit}+n$} is a gluing morphism for any nilpotent $n\in {\rm End}(P_1)$.
\end{lemma}
\begin{proof}
 Let $\gm$ be an invertible gluing morphism. Then $\gm^{-1}=\zeta\Id+n$ for some $n$ nilpotent and $\zeta\in\FK^*$. Then \smash{$\Delta_0^{P_\unit}=\gm^{-1}\Lambda_{P_\unit}=\zeta\Lambda_{P_\unit}$}.
\end{proof}

\section{Presentation of manifolds}\label{sec:presman}
\subsection{Surgery presentation of 3-manifolds containing ribbon graphs} It is well known that any closed 3-manifold can be represented by surgery along a link in $S^3$ and that two such presentations of the same manifold are related by the following two moves. A~Kirby I move which replaces a link $L$ with itself disjoint union a unknot $U_{\pm}$ with framing $\pm 1$:
\begin{equation*}
 L\quad\longleftrightarrow\quad L\sqcup U_{\pm}
\end{equation*}
and
 a Kirby II move which replaces a component $L_i$ of a link with a connected
sum of $L_i$ with a parallel copy of a different component $L_j$:
\begin{equation}
 \label{eq:Kirby2}
 \epsh{K2c}{12ex}\ \longleftrightarrow\ \epsh{K2d}{12ex}\ .
\end{equation}
The result of surgery along a link $L$ is a 3-manifold
$S^3(L)$ uniquely defined up to isotopy. If $L$ and $L'$ are related by a
Kirby move, it induces a canonical diffeomorphism
\smash{$S^3(L)\overset\sim\to S^3(L')$} up to isotopy.

Now given a pair $(M,T)$ where $M$ is a closed 3-manifold containing an admissible $\cat$-colored ribbon graph $T$.
 We say the ribbon graph $L \cup T \subset S^3$ is a surgery presentation of $(M,T)$ if~$L$ is a link surgery presentation representing $M$. The components of $L$ are called the surgery components of $L\cup T$.
 We have the following theorem (see~\cite{CGP14}).
\begin{theorem}\label{T:SurgPresRelatedbyKirby}
 For $i=1,2$, let $L_i\cup T_i$ be a surgery presentation in $S^3$ of a $3$-manifold $M_i$ containing a $\cat$-colored ribbon graph $T_i$. Let
 $f\colon M_1\to M_2$ be an orientation preserving diffeomorphism such that
 $f(T_1)=T_2$ as $\cat$-colored ribbon graphs. Then $L_1\cup T_1$ and
 $L_2\cup T_2$ are related by a sequence of orientation changes of
 the surgery components, Kirby I moves away from $T_i$, Kirby II
 moves on the surgery components and Kirby II moves obtained by
 sliding an edge of $T_i$ on a component of the surgery link such
 that the induced diffeomorphism between $M_1$ and $M_2$ is isotopic
 to $f$.
\end{theorem}
\subsection{Presentation of the (3+1)-cobordism categories
}\label{S:Top}
A generator and relation presentation of the $n$-dimensional oriented cobordism category
is given in~\cite{Juh2018}. The (3+1)-TQFT construction given here
uses this work. Here we recall
 some of the definitions and results of~\cite{Juh2018}, in the case of 3-manifolds, i.e., $n=3$.

Let $\man$ be the category whose objects are closed oriented 3-manifolds and whose morphisms are
orientation preserving diffeomorphisms.
 Let $\cob$ be the category whose objects are closed oriented 3-manifolds and morphisms are equivalence classes (up to oriented diffeomorphism that are the identity on the boundary) of oriented cobordisms.
 Finally, let $\nc\cob$ (resp. $\cob'$) be the largest subcategory of $\cob$ such that each component of every cobordism has a non-empty source (resp.\ has a non-empty source and target).
 Here we add the non-compact cobordism category that
 was not consider
 by Juh\'asz but that we will use later.

Let $M$ be a closed oriented 3-manifold. For
$k\in \{0,1,2,3\}$, a \emph{framed k-sphere} in $M$ is an orientation reversing embedding $\Sp=\Sp^k\colon S^k \times D^{3-k} \hookrightarrow M$.
Then we can perform surgery on~$M$ along $\Sp$ by removing the interior of the image of $\Sp$ and gluing in $D^{k+1} \times S^{2-k}$, getting a well defined topological manifold $M(\Sp)$ which, using the framing of the sphere, can be endowed with a canonical smooth structure.
 Juh\'asz considers additional types of framed sphere, namely $\Sp=\Sp^{-1}=0$ and $\Sp=\varnothing$, where $M(0)=M \sqcup S^3$ and
$M(\varnothing)=M$. Given a framed k-sphere $\Sp$,
 we denote by $W(\Sp)$ an associated smooth cobordism from $M \to M(\Sp)$ (unique up to diffeomorphism relative to the boundary see for instance~\cite[Theorem~3.13]{Mil1965}), $W(0)=(M\times[-1,1])\sqcup D^4$ and~${W(\varnothing)=M\times [-1,1]}$.

Let $\mathcal{G}$ be the directed graph described as follows. The vertices correspond to objects of $\man$. There are two kinds of edges of $\mathcal{G}$: 1) if $d\colon M \to M'$ is a diffeomorphism then there is an edge~$e_d$ going from $M$ to $M'$, 2) if $\Sp$ is a $k$-framed sphere in $M$ then
there is an edge $e_{M,\Sp}$ from $M$ to~$M(\Sp)$. Let $\nc{\mathcal G}$ (resp. $\mathcal{G'}$) be the subgraph of $\mathcal{G}$ obtained by removing the edges $e_{M,\Sp}$ where $\Sp=0$ (resp. where $\Sp=0$ or $\Sp$ is a framed $3$-sphere). Let $\mathcal{F}(\mathcal{G})$ (resp.\ $\mathcal{F}(\nc{\mathcal G})$, resp.\ $\mathcal{F}(\mathcal{G}')$) be the free category generated by $\mathcal{G}$ (resp.\ $\nc{\mathcal G}$, resp.\ $\mathcal{G}'$).

In~\cite[Definition~1.4]{Juh2018}, Juh\'asz considers a set of relations $\mathcal{R}$ in $\mathcal{F}(\mathcal{G})$ which we recall now. If~$w$ and~$w'$ are words consisting of composing arrows, then we write $w\sim w'$ if
$(w,w')\in \mathcal{R}$.
\begin{enumerate}\itemsep=0pt
\item[(R1)] For diffeomorphisms $d$ and $d'$ whose composition is defined, we have the relation $e_{d \circ d'} \sim e_d\circ e_{d'}$. We also have the relations $e_{M,\varnothing}\sim e_{\Id_M}$ and $e_d\sim e_{\Id_M}$ if $d\colon M \to M$ is a~diffeomorphism isotopic to the identity.
\item[(R2)] If $d\colon M \to M'$ is an orientation preserving diffeomorphism between 3-manifolds and~$\Sp$ a~framed sphere in $M$, then let $\Sp'=d\circ \Sp$ be the framed sphere in $M'$; then
let $d^\Sp\colon M(\Sp)\to M'(\Sp')$ be the induced diffeomorphism. Then the commutativity of the following diagram defines a relation:
\[\xymatrix{
M \ar[d]_{d} \ar[r]^{e_{M,\Sp}} & M(\Sp) \ar[d]^{d^\Sp} \\
M' \ar[r]^{e_{M',\Sp'}} & M'(\Sp').}\]
\item[(R3)] If $\Sp$, $\Sp'$ are disjoint framed sphere in an oriented 3-manifold $M$, then
$M(\Sp)(\Sp')= M(\Sp')(\Sp)$ and we denote this 3-manifold by $M(\Sp,\Sp')$. The commutativity of the following diagram defines a relation:
\[\xymatrix{
M \ar[d]_{e_{M,\Sp'}} \ar[r]^{e_{M,\Sp}} & M(\Sp) \ar[d]^{e_{M(\Sp),\Sp'}} \\
 M(\Sp')\ar[r]^{e_{M(\Sp'),\Sp}} & M(\Sp,\Sp').}\]

 \item[(R4)]
 Let $\Sp$ be a framed $k$-sphere in $M$ and $\Sp'$ a framed $k'$-sphere in $M(\Sp)$. If the attaching sphere \smash{$\Sp'\bigl(S^{k'}\times \{0\}\bigr)\subset M(\Sp)$} intersects the belt sphere
 $\{0\}\times S^{-k+2} \subset M(\Sp)$ once transversally, then there is a diffeomorphism (well defined up to isotopy)
 $\phi \colon M \to M(\Sp,\Sp')$ (defined in~\cite[Definition~2.17]{Juh2018}) and the following is a relation:
 \[
 e_{M(\Sp),\Sp'} \circ e_{M,\Sp} \sim \phi.
 \]
\item[(R5)] If $\Sp\colon S^k \times D^{3-k} \hookrightarrow M $ is a framed $k$-sphere, then there is a relation $ e_{M,\Sp} \sim e_{M,\bar\Sp} $ where~$\bar \Sp$ is the framed $k$-sphere given by $\bar\Sp(x,y)= \Sp(r_{k+1}(x),r_{3-k}(y))$ with $x\in S^k \subset \R^{k+1}, y\in D^{3-k}\subset \R^{3-k}$ and $r_m(x_1,x_2, \dots, x_m)=(-x_1,x_2, \dots, x_m)$.
\end{enumerate}
Let $\nc{\mathcal{R}}$ (resp.\ $\mathcal{R'}$) be the subset of
$\mathcal{R}$ consisting of relations where all involved edges are in
$\nc {\mathcal{G}}$ (resp.~$\mathcal{G'}$).
\begin{definition}[{\cite[Definition~1.5]{Juh2018}}]
 Let $c\colon \mathcal{G}\to \cob$ be the map which is the identity on
 vertices, assigns the cylindrical cobordism $c_d$ to a
 diffeomorphism $d$ (see~\cite[Definition~2.3]{Juh2018}) and
 assigns the cobordism $W(\Sp)$ to the edge $ e_{M, \Sp}$. This
 extends to a symmetric strict monoidal functor
 $c\colon \mathcal{F}(\mathcal{G})\to \cob$ (resp
 $c\colon \mathcal{F}(\nc{\mathcal{G}})\to \nc{\cob}$,
 resp. $c\colon \mathcal{F'}({\mathcal{G'}})\to {\cob'}$).
\end{definition}

Recall that given a category $\cat$ and a set of relations $\sim$ on
its morphisms, the quotient category~$\cat/{\sim}$ has the same objects
as $\cat$ and
equivalence classes of morphisms of $\cat$
under the monoidal congruence relation generated by $\sim$
as morphisms.
\begin{theorem}[{\cite[Theorem~1.7]{Juh2018}}]
The functor $c\colon \mathcal{F}(\mathcal{G})\to \cob$ induces a functor
$\mathcal{F}(\mathcal{G})/\mathcal{R}\to \cob$,
which is an isomorphism of symmetric monoidal categories.
Furthermore there is also an induced isomorphism of symmetric monoidal categories
$c\colon \mathcal{F}(\mathcal{G'})/\mathcal{R'}\to \cob'$.
\end{theorem}

\begin{remark}
 The theorem is stated in a slightly different way by Juh\'asz as an
 isomorphism $\mathcal{F}(\mathcal{G'})/\mathcal{R}\to \cob'$ but a careful
 reading of its proof shows that the $0$-$1$-handle cancellations
 and the $(n-1)$-$n$-handle cancellations can be avoided so that one
 only needs to consider the equivalence relation generated by
 $\mathcal R'$ (this is explicitly stated in~\cite[Theorem~2.24]{Juh2018}) . The
 other difference is that, in order to have a monoidal category, we have added
 the empty manifold in~$\mathcal{G'}$ and $\cob'$ but it has no maps
 to any other object.
\end{remark}

\begin{corollary}
 The functor $c\colon \mathcal{F}(\nc{\mathcal{G}})\to \nc{\cob}$ induces
 a functor
 \[
 \mathcal{F}(\nc{\mathcal{G}})/\nc{\mathcal{R}}\to \nc\cob,
 \]
 which is an isomorphism of symmetric monoidal categories.
\end{corollary}
\begin{proof}
 This corollary follows from Juh\'asz's argument in~\cite{Juh2018}
 using parameterized Cerf decomposition. For an argument based on
 the statements of Juh\'asz's theorems, one can easily translate in
 dimension 4 the one made in dimension~3 of~\cite[Corollary~4.3]{CGPV2023a}.
\end{proof}

\section{3-manifold invariant}\label{sec:3mfld}
In this section, we assume $\cat$ is twist non-degenerate.
Recall the definition of the scalars $\Delta_\pm$ given in Section~\ref{SS:Def-gluingmorph}.
\begin{theorem}\label{T:Exist3ManInv}
Let $(M,T)$ be a pair where $M$ is a closed $3$-manifold containing an admissible $\cat$-colored ribbon graph $T$. Let $L \cup T \subset S^3$ be a surgery presentation of $(M,T)$. If $L^{\rm blue}$ is a~$\cat$-colored ribbon graph obtained by making each red component of $T \cup L$ blue using a red to blue modification, then
\begin{equation}\label{E:Def3ManInv}
\ThreeManInv'_\cat(M,T)=\dfrac{F'(L^{\rm blue})}{\Delta_+^r\Delta_-^s}
\end{equation}
only depends on the diffeomorphism class of $(M,T)$, where $(r,s)$ is the signature of the linking matrix of $L$.

In particular, if $M$ is connected, the scalar
\[
\ThreeManInv_\cat(M):=\ThreeManInv'_\cat(M,\Gamma_0),
\] where
$\Gamma_0$ defined in \eqref{GammaZero} is contained in a ball in $M$, is an invariant of $3$-manifolds.
\end{theorem}
\begin{proof}
Lemma~\ref{prop:redtoblue} implies that any
way of making a
 surgery presentation blue only depends on the surgery presentation. This lemma also implies that using a red to blue modification on a~unknot with $\pm1$ framing with any chromatic morphism produces the same the scalar $\Delta_\pm$. Thus, it is enough to show the invariant is well defined for any two surgery presentations
which are related by a Kirby I or II move as in Theorem~\ref{T:SurgPresRelatedbyKirby}. Lemma~\ref{lemma:slidingOverRed} implies any of the Kirby~II~moves in Theorem~\ref{T:SurgPresRelatedbyKirby} hold. Finally, since the category is twist non-degenerate, then the normalization in equation \eqref{E:Def3ManInv} implies invariance under any Kirby I move.
\end{proof}

When $T$ is not necessarily admissible, and $M$ is connected define
$\ThreeManInv_\cat(M,T)=\ThreeManInv'_\cat(M,T\sqcup \Gamma_0)$, where
\[\Gamma_0=\begin{tikzpicture}[baseline = 5pt]
\node[draw, rectangle, minimum height = 0.4cm, minimum width = 0.7cm] (eta) at (0,0){$\eta$};
\node[draw, rectangle, minimum height = 0.4cm, minimum width = 0.7cm] (epsilon) at (0,1){$\varepsilon$};
\draw (eta) -- (epsilon) node[midway,sloped]{$>$} node[midway, right]{$P_\unit$};
\end{tikzpicture}\]
is contained in a ball in $M$. If $(M,T)$ and $(M',T')$ are 3-manifolds with $\cat$-colored ribbon graphs such that $T'$ is admissible, then
\begin{equation*}
\ThreeManInv'_\cat\bigl((M,T)\sharp\bigl(M',T'\bigr)\bigr)=\ThreeManInv_\cat(M,T)\ThreeManInv'_\cat\bigl(M',T'\bigr),
\end{equation*}
where $\sharp$ stands for the connected sum along balls not intersecting $T$ nor $T'$.

\begin{theorem} The $3$-manifold invariant $\ThreeManInv'$ recovers the following invariants as special cases:
 \begin{enumerate}\itemsep=0pt
 \item[$1.$] If $\catd$ is semisimple modular with the standard categorical trace, then
 \[\operatorname{WRT}_\catd=\ThreeManInv'_\catd=\ThreeManInv_\catd,\]
 where $\operatorname{WRT}_\catd$ is the Witten--Reshetikhin--Turaev
 invariant associated to $\catd$, see~{\rm \cite{RT, Tu}}.
 \item[$2.$] If $\catd$ is the category of finite-dimensional left
modules over a finite-dimensional unimodular ribbon Hopf algebra with right integral
$\lambda$, equipped with the modified trace induced by $\lambda$ $($see
{\rm\cite[{\it Theorem} 1]{BBG17b})}, then
\[
\operatorname{H}_H(M)=\mathbb{D}^{-1-b_1}\ThreeManInv_\catd(M,\varnothing) \qquad \text{and}\qquad \operatorname{H}'_\catd(M,T) =\mathbb{D}^{-1-b_1}\ThreeManInv'_\catd(M,T),
\]
where
$b_1$ is the first Betti number of $M$, $\operatorname{H}_H$ is the Hennings invariant defined in~{\rm \cite{H96}} and~$\operatorname{H}'_\catd$ is its renormalized version defined in~{\rm \cite{DRGPM18}}. Here both invariants are normalized using a choice of square root $\mathbb{D}$ of $\Delta_+\Delta_-$.
\item[$3.$] If $\catd$ is an abelian finite unimodular $\FK$-linear ribbon category $($note in the rest of the paper we do not require our category to be abelian$)$, then
\[\operatorname{L}_\catd(M)=\mathbb{D}^{-1-b_1}\ThreeManInv_\catd(M,\varnothing) \qquad\text{and}\qquad \operatorname{L}'_\catd(M,T) =\mathbb{D}^{-1-b_1}\ThreeManInv'_\catd(M,T),\]
where
$b_1$ is the first Betti number of $M$, $\operatorname{L}_\catd$ is the Lyubashenko invariant defined in~{\rm \cite{L94}} and $\operatorname{L}'_\catd$ its renormalized version defined in~{\rm \cite{DRGGPMR2022}}. Both invariants are normalized using a~choice of square root $\mathbb{D}$ of $\Delta_+\Delta_-$.
 \end{enumerate}
\end{theorem}
 \begin{proof}
We prove the last equality, which recovers the other ones as special cases. Let ${L\cup T}$ be a~surgery presentation of $(M,T)$. Let $F_\lambda$ be the extension to bichrome graphs of the R-T functor~$F$ and let $F_\lambda'$ be its renormalization with the m-trace, see~\cite{DRGGPMR2022, DRGPM18}. From~\cite[Theorem~3.8]{DRGGPMR2022}, we~have%
\[
\operatorname{L}_\catd'(M,T)= \mathbb{D}^{-1 - \ell} \delta^{- \sigma(L)} F'_{\lambda}(L\cup T),
\]
where $\sigma(L)=r-s$ is the signature linking matrix of a surgery presentation $L$ of $M$, \smash{$\delta=\frac{\Delta_+}{D}=\frac{D}{\Delta_-}$} and $\ell=b_1+r+s$.
The main property of the map $F'_\lambda$ is that it gives a well defined meaning to a red circle that can be made blue using the red-to -blue operation of~\cite{DRGGPMR2022}.
Now~\cite{CGPV2023b} (see~\cite[Lemma 6.3]{CGPT20} in the Hopf-algebraic case) says that this red to blue operation is exactly the chromatic morphism given in equation \eqref{E:FormulaChHopf}. Thus, with the notation of Theorem~\ref{T:Exist3ManInv} it follows that \smash{$F'_{\lambda}(L\cup T)=F'\bigl(L^{\rm blue}\bigr)$}.
Finally, it is easy to show that \smash{$\mathbb{D}^{-1 - \ell} \delta^{- \sigma(L)} =\frac{\mathbb{D}^{-1-b_1}} {\Delta_+^r\Delta_-^s}$} for a surgery presentation $L$.
\end{proof}

\section{The (non-compact) (3+1)-TQFT}
In this section, we assume that $\cat$ is
chromatic non-degenerate.

\subsection{Construction of TQFT and 4-dimensional invariants}
We extend the functor $\TSkein_\cat\colon\man\to \Vect$ to a
functor $\TSkein_\cat\colon \mathcal{F}(\nc{\mathcal{G}})\to \Vect$
(resp.~$\TSkein_\cat\colon \mathcal{F}(\mathcal{G})\to \Vect$ if
$\cat$ is chromatic compact)
by assigning to each $\Sp$-surgery a
linear map between skein modules.

Let $M$ be a closed 3-manifold. For $k\in \{0,\dots ,4\}$ and $\Sp^{k-1}$ a framed sphere in $M$, recall from Section~\ref{S:Top} the cobordism $W\bigl(\Sp^{k-1}\bigr)$ which is given by gluing a $k$-handle on $M\times [-1,1]$. We define the associated morphism
\[\chi _{M,{\Sp^{k-1}}}\colon\ \TSkein_\cat(M)\to\TSkein_\cat\bigl(M\bigl({\Sp^{k-1}}\bigr)\bigr)\]
as follows.

\begin{figure}[t]
 \centering
 $T \quad \mapsto \quad \zeta\cdot\ T\sqcup$
\begin{tikzpicture}[baseline = 0pt]
\draw[gray] (-1.4,0) arc(180:540:1.5);
\fill[gray!10] (-1.4,0) arc(180:540:1.5);
\node[color = gray] at (1.4,1.2){$S^3$};
\draw[blue] (0,-0.7)node[rectangle, draw=black, fill=white]{\ \color{black}$\eta$\ \ } --(0,0.7)node[pos = 0.5, right]{$P_\unit$} node [pos = 0.5, sloped]{$>$} node[rectangle, draw=black, fill=white]{\ \color{black}$\varepsilon$\ \ };
\end{tikzpicture}
\caption{The 0-handle, or birth map, adds a standard skein in the new $S^3$ component. It is only defined when $\CC$ is chromatic compact.}
 \label{fig:0handle}
\end{figure}
\textbf{0-handle}: We only consider 0-handles when $\cat$ is
chromatic compact and so $\gm=\zeta^{-1}\Id_{P_\unit}$ is a gluing morphism.
Let $\Sp^{-1}\colon \varnothing\hookrightarrow M$ be a framed $-1$-sphere. Recall $\Gamma_0$ is the ribbon graph with a unique edge from a
coupon colored with $\eta$ to a coupon colored by $\ve$ (see the right-hand side of Figure~\ref{fig:0handle}). Then there exists a \emph{birth map}:
\begin{align*}
\chi _{M,{\Sp^{-1}}}\colon\ \TSkein_\cat(M)&{}\to\TSkein_\cat\bigl(M\sqcup S^3\bigr),\\
 T &{}\mapsto \zeta\cdot\ T\sqcup\Gamma_0
\end{align*}
sending a skein in $M$ to its disjoint union with $\bigl(S^3,\zeta\Gamma_0\bigr)$, see Figure~\ref{fig:0handle}.

\begin{figure}[t]
 \centering
\begin{tikzpicture}[baseline = 0pt,xscale = 0.8]
\fill[gray!10] (0,1) arc(180:360:1 and 0.8) -- (2,2) -- (0,2) -- (0,1);
\fill[gray!10] (0,-1) arc(180:0:1 and 0.8) -- (2,-2) -- (0,-2) -- (0,-1);
 \fill[gray!10] (0,2) arc(180:540:1 and 0.4);
 \draw[gray] (0,2) arc(180:540:1 and 0.4);
 \draw[gray] (0,2) arc(180:360:1 and 0.2);
 \draw[dashed, gray] (0,2) arc(180:0:1 and 0.2);

 \fill[gray!10] (0,-2) arc(180:540:1 and 0.4);
 \draw[gray] (0,-2) arc(180:540:1 and 0.4);
 \draw[gray] (0,-2) arc(180:360:1 and 0.2);
 \draw[dashed, gray] (0,-2) arc(180:0:1 and 0.2);

 \draw[gray] (0,1) arc(180:360:1 and 0.8);
 \draw[gray] (0,-1) arc(180:0:1 and 0.8);

 \draw[gray] (0,1)--++(0,1);
 \draw[gray] (0,-2)--++(0,1);
 \draw[gray] (2,1)--++(0,1);
 \draw[gray] (2,-2)--++(0,1);
 \node[gray] at (2.7,2.2) {$D^3\times S^0$};

\node[rectangle, draw, fill=white] (Up) at (1,0.7) {$\eta_\unit$};
\node[rectangle, draw, fill=white] (Do) at (1,-0.7) {$\varepsilon_\unit$};
\draw[blue] (0.5,2) node[rotate = 40, above]{$\vdots$} ..controls (0.5, 1.8) and (0.8,1.8).. (1, 1.5) -- (Up);
\draw[blue] (1.5,2) node[rotate = -40, above]{$\vdots$} ..controls (1.5, 1.8) and (1.2,1.8).. (1, 1.5);
\node[rectangle, draw, fill=white] at (1,1.5){$T_+$};

\draw[blue] (1,-2) node{$\vdots$} -- (1, -1.5) -- (Do);
\node[rectangle, draw, fill=white] at (1,-1.5){$T_-$};
\end{tikzpicture}
$\longmapsto$ \hspace{1cm}
\begin{tikzpicture}[baseline = 0pt,xscale = 0.8]
\fill[gray!10] (0,-2) rectangle (2,2);
 \fill[gray!10] (0,2) arc(180:540:1 and 0.4);
 \draw[gray] (0,2) arc(180:540:1 and 0.4);
 \draw[gray] (0,2) arc(180:360:1 and 0.2);
 \draw[dashed, gray] (0,2) arc(180:0:1 and 0.2);

 \fill[gray!10] (0,-2) arc(180:540:1 and 0.4);
 \draw[gray] (0,-2) arc(180:540:1 and 0.4);
 \draw[gray] (0,-2) arc(180:360:1 and 0.2);
 \draw[dashed, gray] (0,-2) arc(180:0:1 and 0.2);

 \draw[gray] (0,0) arc(180:540:1 and 0.4);
 \draw[gray] (0,0) arc(180:360:1 and 0.2);
 \draw[dashed, gray] (0,0) arc(180:0:1 and 0.2);

 \draw[gray] (0,-2)--++(0,4);
 \draw[gray] (2,-2)--++(0,4);
 \node[gray] at (2.7,2.2) {$S^2\times D^1$};

\node[inner sep = 0pt, outer sep = 0pt] (P) at (1,0) {};
\draw[blue] (0.5,2) node[rotate = 40, above]{$\vdots$} ..controls (0.5, 1.8) and (0.8,1.8).. (1, 1) -- (P);
\draw[blue] (1.5,2) node[rotate = -40, above]{$\vdots$} ..controls (1.5, 1.8) and (1.2,1.8).. (1, 1);
\node[rectangle, draw, fill=white] at (1,1){$T_+$};

\draw[blue] (1,-2) node{$\vdots$} -- (1, -1) -- (P) node[rectangle, draw=black, fill=white]{\color{black}$\gm$};
\node[rectangle, draw, fill=white] at (1,-1){$T_-$};
\end{tikzpicture}
\caption{The 1-handle, or gluing map, connects two skeins living in disjoint disks by replacing an $\varepsilon$ and an $\eta$ coupon by a $\gm$ coupon.}
 \label{fig:1handle}
\end{figure}
\begin{figure}[t]
$$
\epsh{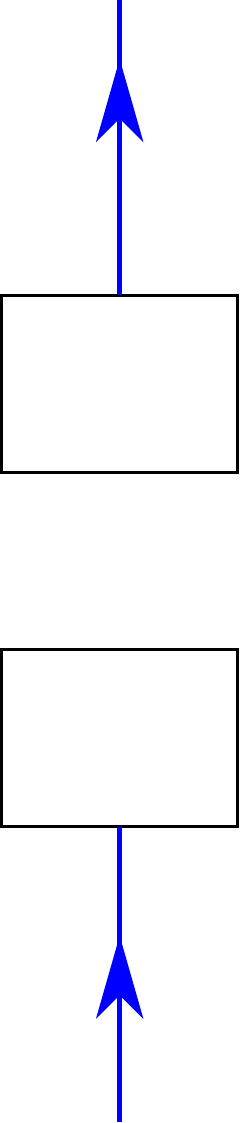}{14ex}
\putc{47}{66}{$\ms{\eta}$}
\putc{47}{34}{$\ms{\ve}$}
\putlc{70}{90}{$\ms{P_\unit}$}
\putlc{72}{12}{$\ms{P_\unit}$}
\longmapsto
\epsh{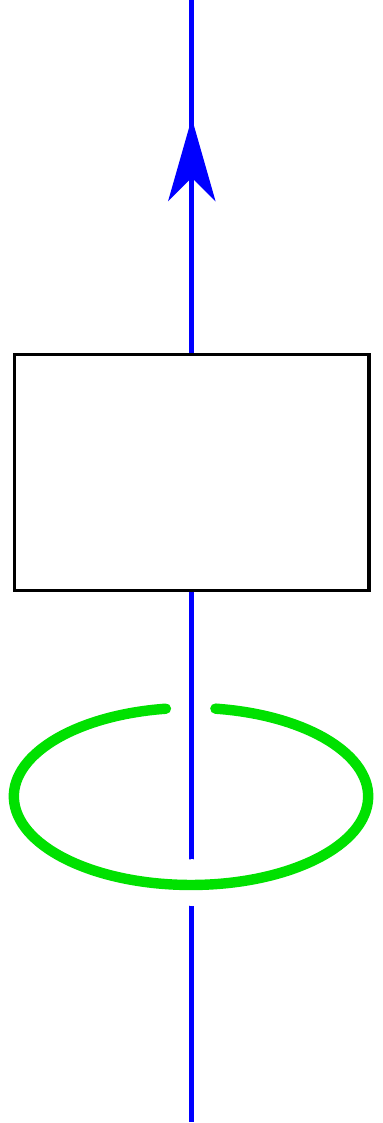}{14ex}
\putc{49}{58}{$\ms{\gm}$}
\putlc{61}{85}{$\ms{P_\unit}$}
\qquad\text{ or }\qquad
\epsh{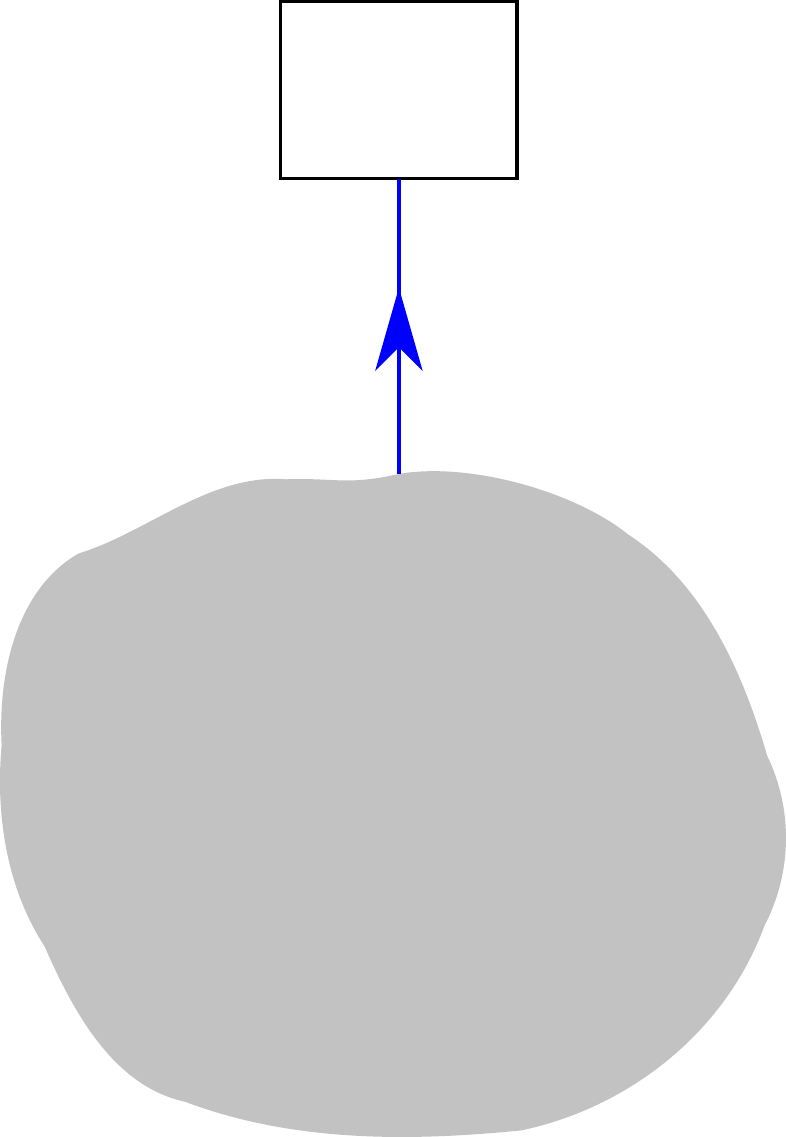}{12ex}
\putc{49}{30}{A}
\putc{50}{93}{$\ms{\ve}$}
\putlc{56}{70}{$\ms{P_\unit}$}
\sqcup
\epsh{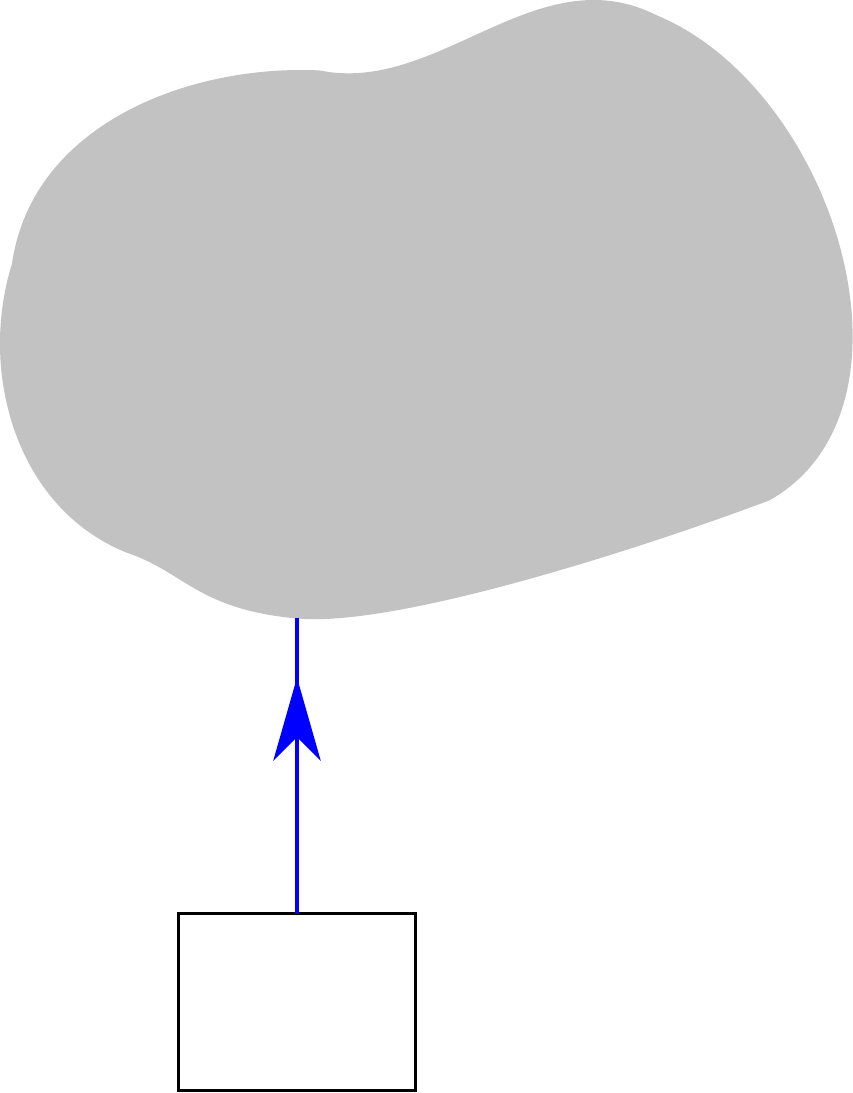}{12ex}
\putc{46}{71}{$B$}
\putc{34}{8}{$\ms{\eta}$}
\putlc{40}{33}{$\ms{P_\unit}$}
\longmapsto
\epsh{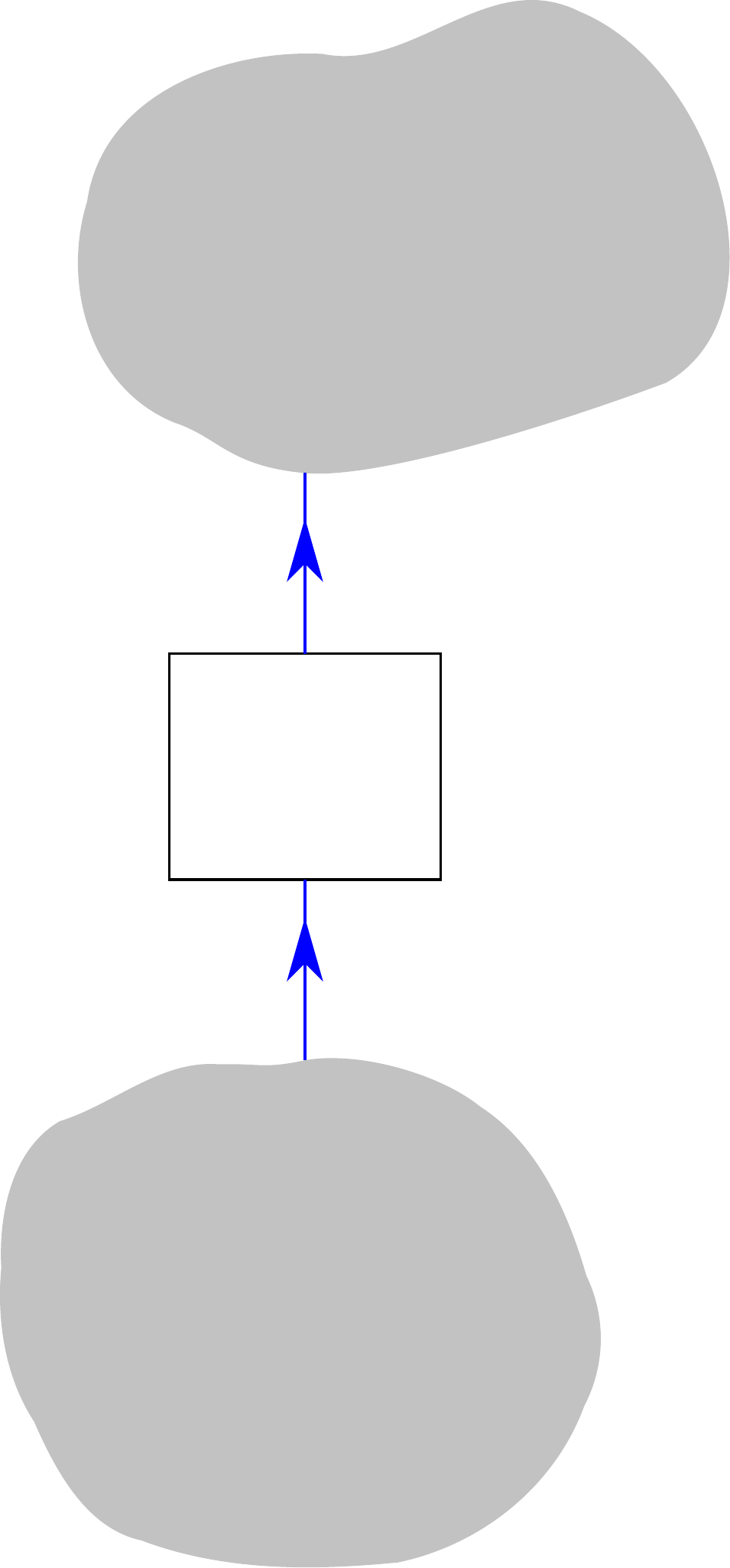}{16ex}
\putc{55}{84}{B}
\putc{41}{51}{$\ms{\gm}$}
\putc{40}{16}{A}
\putlc{49}{65}{$\ms{P_\unit}$}
\putlc{49}{38}{$\ms{P_\unit}$}
$$
 \caption{The gluing map $\chi_{M,\Sp^0}$
 is depicted by
 two different
 representations depending on whether $\Sp^0$ is embedded in a unique
 connected component of $M$ (left) or not (right).}\label{F:glue}
\end{figure}
\textbf{1-handle}: Given a framed sphere $\Sp^0$ in $M$ there exists a \emph{gluing map}:
\[
\chi _{M,{\Sp^0}}\colon\ \TSkein_\cat(M)\to\TSkein_\cat\bigl(M\bigl({\Sp^0}\bigr)\bigr),
\]
which glues two edges terminating on coupons colored by $\eta$ and $\varepsilon$ by a gluing morphism as represented in Figure~\ref{fig:1handle}. We drew the disks and spheres there, but for computations it is often useful describe this operation on a surgery presentation of~$M$. If the gluing $\Sp^0$ is not contained in a single component of $M$, let $L$ and $L'$ be surgery presentation of connected components of~$M$, then $L\cup L'$ is a surgery presentation of $M\bigl(\Sp^0\bigr)$; else a surgery presentation $L$ of $M$ is turned into a surgery presentation of $M\bigl(\Sp^0\bigr)$ by adding a disjoint unknot, $L\to L\cup {\textcolor{green}{O}}$. We represent the 1-handle map in these two situations in Figure~\ref{F:glue}.

\begin{remark}
 In order to understand why we turned towards such a construction,
 the reader should think of the case where the $1$-handle is glued to
 two distinct components of the $3$-manifold, $M_1\sqcup M_2$. In
 this case, one wants to get a linear map from
 $\TSkein(M_1\sqcup M_2)$ to $\TSkein(M_1\#M_2)$. But a surgery
 presentation of $M_1\#M_2$ is obtained by taking the disjoint union
 in $S^3$ of the surgery presentations of $M_1$ and $M_2$, thus the
 simple idea of embedding the skeins $\Gamma_1\subset M_1$ and
 $\Gamma_2\subset M_2$
 separately in $M_1\#M_2$ would not work: indeed, if $M_2=S^3$, it would yield the $0$ vector in
 $\TSkein(M_1\#M_2)=\TSkein(M_1)$ because one can apply the admissible skein
 relation given by evaluating via the Reshetikhin--Turaev functor the
 skein $\Gamma_2$ which is~$0$ (as it contains a projective color).
 Therefore, one needs to ``glue'' the two skeins in order for them to
 form a single skein in $M_1\#M_2$ not separated by a sphere. The
 glueing morphism serves this purpose, but of course not any morphism
 will as the operation of glueing the two skeins must be well
 defined: the properties we demanded on $g$ are sufficient to ensure
 this.
\end{remark}

Let us describe this morphism in more detail.
Let $x$, $y$ be two distinct points of a 3\nobreakdash-mani\-fold~$M$. Let~$B_x$,~$B_y$
be neighborhood of $x$ and $y$ both oriented and parameterized by~$B^3$
and let ${\Sp^0}$ be the framed 0-sphere $B_x\sqcup B_y$. Let
\smash{$M'=M\setminus (B_x\sqcup B_y)\stackrel i\hookrightarrow M$} be the
inclusion and $C\simeq S^2\times [0,1]$ be the cylinder such that
$M\bigl({\Sp^0}\bigr)=M'\cup_\partial C$. We put in this cylinder a skein $\Gamma_\gm$ with a single coupon
colored by any gluing morphism $\gm$ and an incoming and an outgoing edge
parallel to $(1,0,0)\times [0,1]$, framed in the direction $(0,0,1)$. We will say that a skein $T$ in $M$ is in good position with respect to $\Sp^0$ if $B_x\cap T$ consists of a
planar ribbon graph in $\R^+\times\R\times\{0\}\cap B_x$ consisting of
a unique edge oriented from $(1,0,0)\in\partial B_x$ towards a coupon
colored by $\ve$ and if $B_y\cap T$ consists of a planar ribbon graph
in $\R^+\times\R\times\{0\}\cap B_y$ consisting of a unique edge
oriented from a coupon colored by $\eta$ towards
$(1,0,0)\in\partial B_y$. The map $\chi_{M,{\Sp^0}}$ assigns to a
skein~$T$ in good position
with respect to
 $\Sp^0$ the skein
$(M',
T\cap M')\cup_\partial (C,\Gamma_\gm)$.
\begin{proposition}\label{prop:1handlewelldef}
 The linear map $\chi_{M,\Sp^0}$ is well defined and does not depend
 on the ordering of~$\{x,y\}$ nor
 on the gluing morphism $\gm$.
\end{proposition}
\begin{proof}
 First we note that the admissible skein module is generated by skeins
 in $M$ where every component of $M$ contains a coupon colored by
 $\ve$ and a coupon colored by $\eta$. Indeed, consider a~box
 containing a part of an edge colored by $P\in\Proj$ whose image by
 the RT-functor is~$\lev_P$ and apply Lemma~\ref{L:prop-proj}\,(1) to
 show that a skein relation can be used to make appear a~coupon
 colored by $\ve\colon P_\unit \to \unit$. Let us choose an isomorphism $\psi\colon P_\unit\to P_\unit^*$
 normalized so that $\eta^*\circ\psi=\ve$ then a coupon colored by
 $\ve$ is skein equivalent to a graph with two coupons colored by~$\psi$ and~$\eta$. So applying this procedure twice, we can ensure the presence of a $\epsilon$-colored coupon and of a~$\eta$-colored coupon in each connected component of $M$.

 Now, up to isotopy of the skein, the definition of
 $\chi_{M,{\Sp^0}}$ only depends a priori on the choice of the two
 coupons colored by $\ve$ and $\eta$, and on the choice of a gluing
 morphism $\gm$: we will now prove independence on these data. Let $\gm'$ be an other gluing morphism and consider
 the element obtained by using $\gm'$ instead of $\gm$ and two
 different coupons colored with $\ve$ and $\eta$. Then we have
 if $\Sp^0$ is embedded in a unique connected component,
 \begin{gather*}\epsh{fig23-3.pdf}{14ex}
 \putc{47}{66}{$\ms{\eta}$}
 \putc{47}{34}{$\ms{\ve}$}
 \putlc{70}{90}{$\ms{P_\unit}$} \putlc{72}{12}{$\ms{P_\unit}$}\hspace{3ex}
 \epsh{fig23-3.pdf}{14ex}
 \putc{47}{66}{$\ms{\eta}$}
 \putc{47}{34}{$\ms{\ve}$}
 \putlc{70}{90}{$\ms{P_\unit}$}
 \putlc{72}{12}{$\ms{P_\unit}$}
 \quad\longmapsto\quad
 \epsh{fig23-3.pdf}{14ex}
 \putc{47}{66}{$\ms{\eta}$}
 \putc{47}{34}{$\ms{\ve}$}\hspace{3ex}
 \epsh{fig23-4.pdf}{14ex}
 \putc{49}{58}{$\ms{\gm}$}
 \quad=\quad
 \epsh{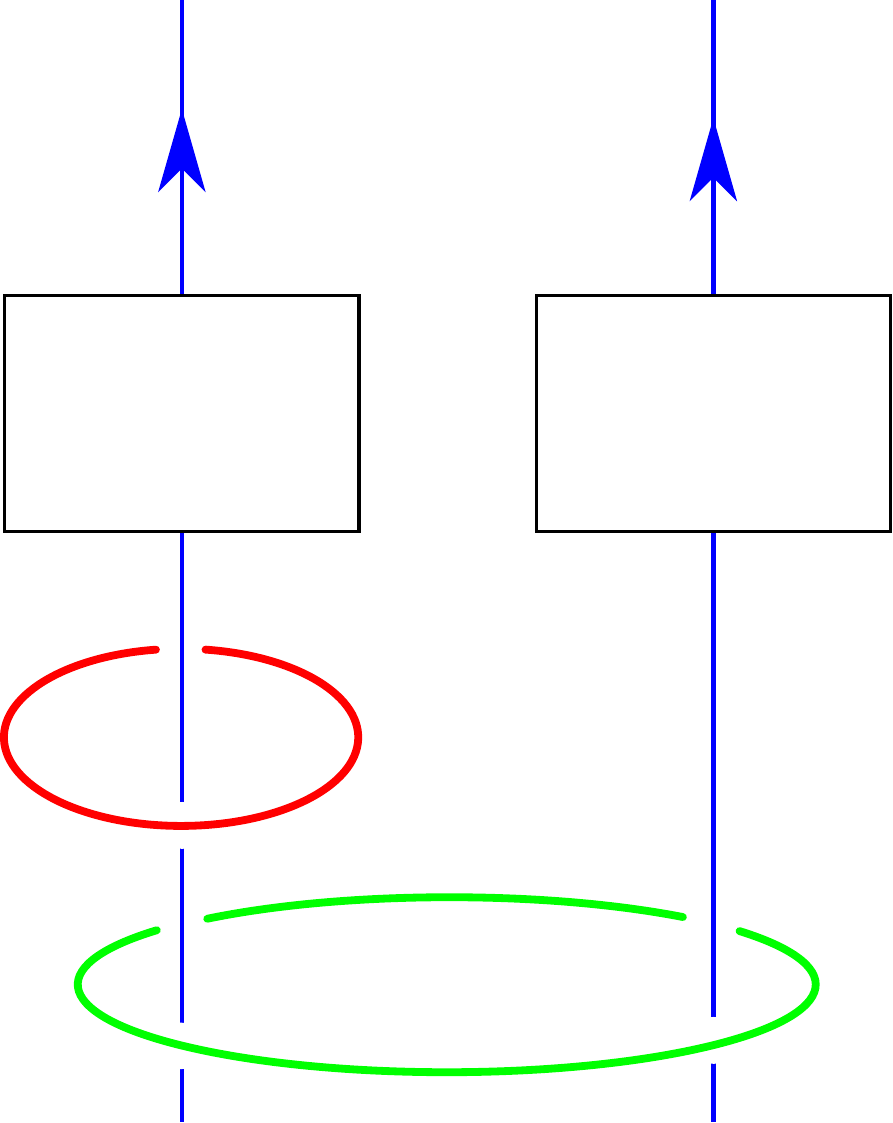}{14ex}
 \putc{19}{63}{$\ms{\gm'}$}
 \putc{80}{63}{$\ms{\gm}$}
 \quad=\quad
 \epsh{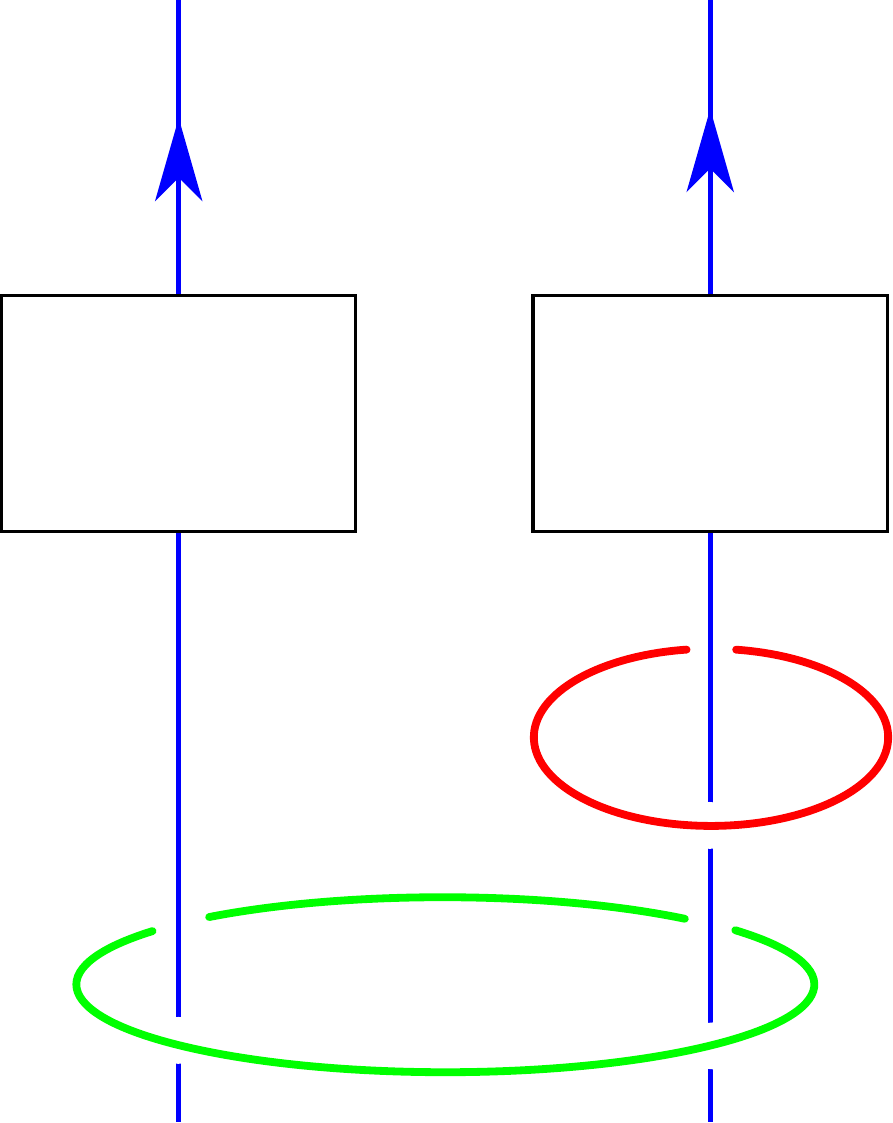}{14ex}
 \putc{19}{63}{$\ms{\gm'}$}
 \putc{79}{64}{$\ms{\gm}$}
 \quad=\quad
 \epsh{fig23-4.pdf}{14ex}
 \putc{49}{58}{$\ms{\gm'}$}\hspace{3ex}
 \epsh{fig23-3.pdf}{14ex}
 \putc{47}{66}{$\ms{\eta}$}
 \putc{47}{34}{$\ms{\ve}$}\ ,
 \end{gather*}
 where the first and last equalities are skein
 equivalences given by definition of gluing morphisms and the
 middle one is an isotopy of the red circle in the belt 2-sphere
 created by gluing the 1-handle. Similarly, if the
 surgery is
 connecting two different components of $M$, the representation of the
 equivalence is similar without the green circles but with the
 separating belt 2-sphere represented by the horizontal
 plane.

 The map $\chi_{M,{\Sp^0}}$ preserves skein relations as we can
 always choose coupons $\ve$ and $\eta$ outside a~fixed box.

 Finally, reversing the orientation of the sphere $\Sp^0$ that is
 interchanging $x$ and $y$ does not change the map since
 $\eta=\psi^{-1}\ve^*$, $\ve=\psi\eta^*$ and $\psi^{-1}\gm^*\psi$ is
 also a gluing morphism.
\end{proof}

\begin{figure}
 \centering
\begin{tikzpicture}[baseline = 0pt]
 \fill[gray!10] (0,0) arc(180:540:2 and 1.5);
 \draw[blue] (2.5,-1.3) node[rotate = 25]{$\vdots$}.. controls (2.2,-0.8) and (2.5, -0.5).. (3.5,0.3) node[pos = 0.5, sloped]{$>$}node[pos = 0.5, above]{$P$};
 \draw[gray] (0,0) arc(180:540:2 and 1.5) node[pos = 0.65, right=5pt]{$S^1\times D^2$};
 \draw[line width = 5pt, gray!10] (0.5,0) arc(180:540:1.5 and 1);
 \draw[gray] (0.5,0) arc(180:540:1.5 and 1) node[pos = 0.75, above=-3pt]{$\scriptstyle S^1\times \{0\}$};
 \node[rectangle, draw, fill=white] at (3.5,0.4){$T$};
 \begin{scope}[xshift = 0.6cm,xscale = 0.7]
 \fill[white] (1,0).. controls (1.75, 0.2) and (2.25, 0.2).. (3,0).. controls (2.25, -0.3) and (1.75, -0.3)..(1,0);
 \draw[gray] (0.5, 0.3) .. controls (0.75,0.1) ..(1,0).. controls (1.75, -0.3) and (2.25, -0.3).. (3,0).. controls (3.25,0.1).. (3.5,0.3);
 \draw[gray] (1,0).. controls (1.75, 0.2) and (2.25, 0.2).. (3,0);
 \end{scope}
\end{tikzpicture} $\longmapsto$
\begin{tikzpicture}[baseline = 5pt]
\fill[gray!10] (-1,-1)..controls (-0.8,-0.5) and (-0.8,0.5).. (-1,1) -- (1,1) ..controls (0.8,0.5) and (0.8,-0.5).. (1,-1) -- (-1,-1);
 \draw[gray] (-1,-1)..controls (-0.8,-0.5) and (-0.8,0.5).. (-1,1) node[above]{$D^2\times S^1$};
 \draw[gray] (1,-1)..controls (0.8,-0.5) and (0.8,0.5).. (1,1);

 \draw[dashed, gray] (-0.85,0) arc(180:0:0.85 and 0.2);

 \draw[red] (0,-1)--(0,1);
 \draw[red, dashed] (0,1).. controls (0,1.8) and (2,1.8).. (2,1) node[pos = 0.5,above]{$\{0\}\times S^1$} --(2,-1)..controls (2,-1.8) and (0,-1.8)..(0,-1);

 \draw[blue] (0.6,-0.8) node[rotate = 25]{$\vdots$}.. controls (0.3,-0.1) and (0.6, 0).. (0.6,0.6) node[pos = 0.2, sloped]{$<$}node[pos = 0.2, right]{$P$};
 \node[rectangle, draw, fill=white] at (0.6,0.6){$T$};

 \draw[line width = 3pt, gray!10] (-0.8,0) arc(180:360:0.8 and 0.2);
 \draw[gray] (-0.85,0) arc(180:360:0.85 and 0.2);
\end{tikzpicture}
\quad := \quad
\begin{tikzpicture}[baseline = 5pt]
\fill[gray!10] (-1,-1)..controls (-0.8,-0.5) and (-0.8,0.5).. (-1,1) -- (1,1) ..controls (0.8,0.5) and (0.8,-0.5).. (1,-1) -- (-1,-1);
 \draw[gray] (-1,-1)..controls (-0.8,-0.5) and (-0.8,0.5).. (-1,1);
 \draw[gray] (1,-1)..controls (0.8,-0.5) and (0.8,0.5).. (1,1);

 \draw[blue] (0,-1)--(0,1);
 \draw[dashed, blue] (0,1).. controls (0,1.8) and (2,1.8).. (2,1) node[pos = 0.5,above]{$G$}node[pos = 0.5]{$>$} --(2,-1)..controls (2,-1.8) and (0,-1.8)..(0,-1);

 \draw[blue] (0.6,-0.8)node[rotate = 25]{$\vdots$}.. controls (0.3,-0.1) and (0.6, 0).. (0.6,0.6) node[pos = 0.2, sloped]{$<$}node[pos = 0.2, right]{$P$};
 \node[rectangle, draw, fill=white] at (0.6,0.6){$T$};
 \node[rectangle, draw, fill=white] at (0.3,0){$\ \chr_P\ $};
\end{tikzpicture}
\begin{tikzpicture}

\end{tikzpicture}
\caption{The 2-handle, or knot-surgery map. For depiction purposes we have represented $S^1\times D^2$ embedded in $S^3$, and $D^2 \times S^1$ as its complement. Note however that the orientations do not quite match, as they both should be on the same side of $S^1 \times S^1$.}
 \label{fig:2handle}
\end{figure}
\begin{figure}[t]
 \[
 \epsh{fig23-8}{10ex} \longmapsto \epsh{fig23-9}{10ex}
 \qquad\text{ and }\qquad
 \epsh{fig23-11}{10ex} \longmapsto \epsh{fig23-10}{10ex}
 \]
 \caption{The knot-surgery map $\chi_{M,\Sp^1}$, two alternative
 representations: on the left we choose a representation of $M$
 where $\Sp^1$ is a meridian of a green knot; a presentation for $M\big(\Sp^1\big)$ is then obtained by forgetting the green knot in the presentation of $M$, but the map on skeins consists of adding a red component along that $\Sp^1$. On the right, the surgery presentation of $M\big(\Sp^1\big)$ is obtained by adding the green circle \big(which is $\Sp^1$\big) and the map on skeins consists in adding also its red meridian.}\label{F:surgery}
\end{figure}
\textbf{2-handle:} Given a framed sphere $\Sp^1$ in $M$,
there exists a \emph{knot-surgery map}:
\[
\chi _{M,{\Sp^1}}\colon\ \TSkein_\cat(M)\to\TSkein_\cat\bigl(M\bigl({\Sp^1}\bigr)\bigr)
\]
adding a red circle along the meridian of the surgery knot, see Figure~\ref{fig:2handle}. Again, instead of drawing solid tori, it will be useful to have a diagrammatic representation based on a surgery presentation of $M$, see Figure~\ref{F:surgery}.

Let $C=-B^2\times S^1$ where the sign of $B^2$ means reversing
orientation and $O_r\subset C$
be a red ribbon knot of the form
$[-0.1,0.1]\times\{0\}\times S^1$. Let $\Sp^1\simeq S^1\times B^2$ be
a framed knot in $M$, $M'=M\setminus \bigl(S^1\times B^2\bigr)$ and
$M''=M'\cup_\partial C$. Let \smash{$\TSkein_\cat(M')\tto i\TSkein_\cat(M)$} and \smash{$\TSkein_\cat(M')\tto {i''}\TSkein_\cat(M'')$} be the
maps induced by the inclusions. We define $\chi_{M,\Sp^1}$ to be the
map that sends a skein $i(T)$ to~$i''(T)\cup O_r$.
 Observe that this map is well-defined on all of $\TSkein_\cat(M)$ because each skein in $M$ can be isotoped
 out of
 $C$.
\begin{remark}
A reader who is used to the standard skein theoretical approach to ${\rm SU}(2)$-Reshetikhin--Turaev TQFTs might be confused by Figure~\ref{F:surgery}. Let us explain why this operation seems to be the inverse of what one would usually do.
In the standard RT theory, suppose one is given a skein $s$ in a $3$-manifold $X$ and that the manifold $X$ is obtained by surgery on a link~$L$ contained in a $3$-manifold $Y$. Then one can draw $s$ in $Y\setminus L$ and color $L$ by the Kirby color: this gives the same vector in the RT TQFT as $(X,s)$ but represented in $Y$ as $(Y,s\cup L_\Omega)$ (where by~$L_\Omega$ we denote $K$ colored by the Kirby color). This is the same as what is depicted in the left hand side of Figure~\ref{F:surgery}.
Thus one does a surgery ``from $Y$ to $X$'' (because the surgery link $L$ is in~$Y$) but then performs a ``pull back of a skein $s$ from $X$ to $Y$'' (i.e., one ends up drawing a skein in~$Y$).
Here instead we want to ``push forward the skeins'' (because we want to build a~covariant functor), i.e., we want to start from a skein which is depicted in $Y$ and draw its image in $X$.
In~the~right-hand side of the figure $M=Y$ is the manifold before the surgery on the green knot~$L$ and $X=M(S)$ is the one obtained after the surgery on the knot, and we want to ``push forward'' the skeins ``from $M=Y$ to $M(S)=X$''. Therefore, to obtain a skein in $X=M(S)$ one has to operate the ``opposite surgery'' which corresponds to coloring the meridian of L via the red color (corresponding to the Kirby color in the standard case).
\end{remark}
\begin{proposition}
 The linear map $\chi_{M,\Sp^1}$ is well-defined.
\end{proposition}
\begin{proof}
If $T_1,T_2\in \TSkein_\cat(M')$ are such that $i(T_1)=i(T_2)$ then $T_1$ and $T_2$ differ by isotopies in $M'$, slidings through meridian discs of $C$ and skein relations which, up to isotopy, can be supposed to be supported in a box disjoint from $C$.
 Then $i''(T_1)\sqcup O_r$ and $i''(T_2)\sqcup O_r$ differ by isotopies in~$i''(M')$, skein relations in $i''(M')$ and sliding of edges on the created red component $O_r$, which by Lemma~\ref{lemma:slidingOverRed} preserves the class in $\TSkein_\cat(M'')$.
\end{proof}

\begin{figure}[t]
 \centering
\begin{tikzpicture}[baseline = 0pt,xscale = 0.8]
\fill[gray!10] (0,-2) rectangle (2,2);
 \fill[gray!10] (0,2) arc(180:540:1 and 0.4);
 \draw[gray] (0,2) arc(180:540:1 and 0.4);
 \draw[gray] (0,2) arc(180:360:1 and 0.2);
 \draw[dashed, gray] (0,2) arc(180:0:1 and 0.2);

 \fill[gray!10] (0,-2) arc(180:540:1 and 0.4);
 \draw[gray] (0,-2) arc(180:540:1 and 0.4);
 \draw[gray] (0,-2) arc(180:360:1 and 0.2);
 \draw[dashed, gray] (0,-2) arc(180:0:1 and 0.2);

 \draw[gray] (0,0) arc(180:540:1 and 0.4);
 \draw[gray] (0,0) arc(180:360:1 and 0.2);
 \draw[dashed, gray] (0,0) arc(180:0:1 and 0.2);

 \draw[gray] (0,-2)--++(0,4);
 \draw[gray] (2,-2)--++(0,4);
 \node[gray] at (2.7,2.2) {$S^2\times D^1$};

\node[inner sep = 0pt, outer sep = 0pt, blue] (P) at (1,0) {};
\draw[blue] (0.5,2) node[rotate = 40, above]{$\vdots$} ..controls (0.5, 1.8) and (0.8,1.8).. (1, 1) -- (P) node[right]{$P$};
\draw[blue] (1.5,2) node[rotate = -40, above]{$\vdots$} ..controls (1.5, 1.8) and (1.2,1.8).. (1, 1);
\node[rectangle, draw, fill=white] at (1,1){$T_+$};

\draw[blue] (1,-2) node{$\vdots$}-- (1, -1) -- (P)node[pos = 0.98, sloped]{$>$};
\node[rectangle, draw, fill=white] at (1,-1){$T_-$};

\end{tikzpicture}
$\longmapsto$ \hspace{1cm} $\displaystyle \sum_i \quad$
\begin{tikzpicture}[baseline = 0pt,xscale = 0.8]
\fill[gray!10] (0,1) arc(180:360:1 and 0.8) -- (2,2) -- (0,2) -- (0,1);
\fill[gray!10] (0,-1) arc(180:0:1 and 0.8) -- (2,-2) -- (0,-2) -- (0,-1);
 \fill[gray!10] (0,2) arc(180:540:1 and 0.4);
 \draw[gray] (0,2) arc(180:540:1 and 0.4);
 \draw[gray] (0,2) arc(180:360:1 and 0.2);
 \draw[dashed, gray] (0,2) arc(180:0:1 and 0.2);

 \fill[gray!10] (0,-2) arc(180:540:1 and 0.4);
 \draw[gray] (0,-2) arc(180:540:1 and 0.4);
 \draw[gray] (0,-2) arc(180:360:1 and 0.2);
 \draw[dashed, gray] (0,-2) arc(180:0:1 and 0.2);

 \draw[gray] (0,1) arc(180:360:1 and 0.8);
 \draw[gray] (0,-1) arc(180:0:1 and 0.8);

 \draw[gray] (0,1)--++(0,1);
 \draw[gray] (0,-2)--++(0,1);
 \draw[gray] (2,1)--++(0,1);
 \draw[gray] (2,-2)--++(0,1);
 \node[gray] at (2.7,2.2) {$D^3\times S^0$};

\node[rectangle, draw, fill=white] (Up) at (1,0.7) {$x_i$};
\node[rectangle, draw, fill=white] (Do) at (1,-0.7) {$x^i$};
\draw[blue] (0.5,2) node[rotate = 40, above]{$\vdots$}..controls (0.5, 1.8) and (0.8,1.8).. (1, 1.5) -- (Up);
\draw[blue] (1.5,2) node[rotate = -40, above]{$\vdots$}..controls (1.5, 1.8) and (1.2,1.8).. (1, 1.5);
\node[rectangle, draw, fill=white] at (1,1.5){$T_+$};

\draw[blue] (1,-2)node{$\vdots$} -- (1, -1.5) -- (Do);
\node[rectangle, draw, fill=white] at (1,-1.5){$T_-$};
\end{tikzpicture}
\caption{The 3-handle, or cutting map. Here $\Omega_P = \sum_i x^i \otimes x_i$ is the copairing.}
 \label{fig:3handle}
\end{figure}
\begin{figure}[t]
\[
\epsh{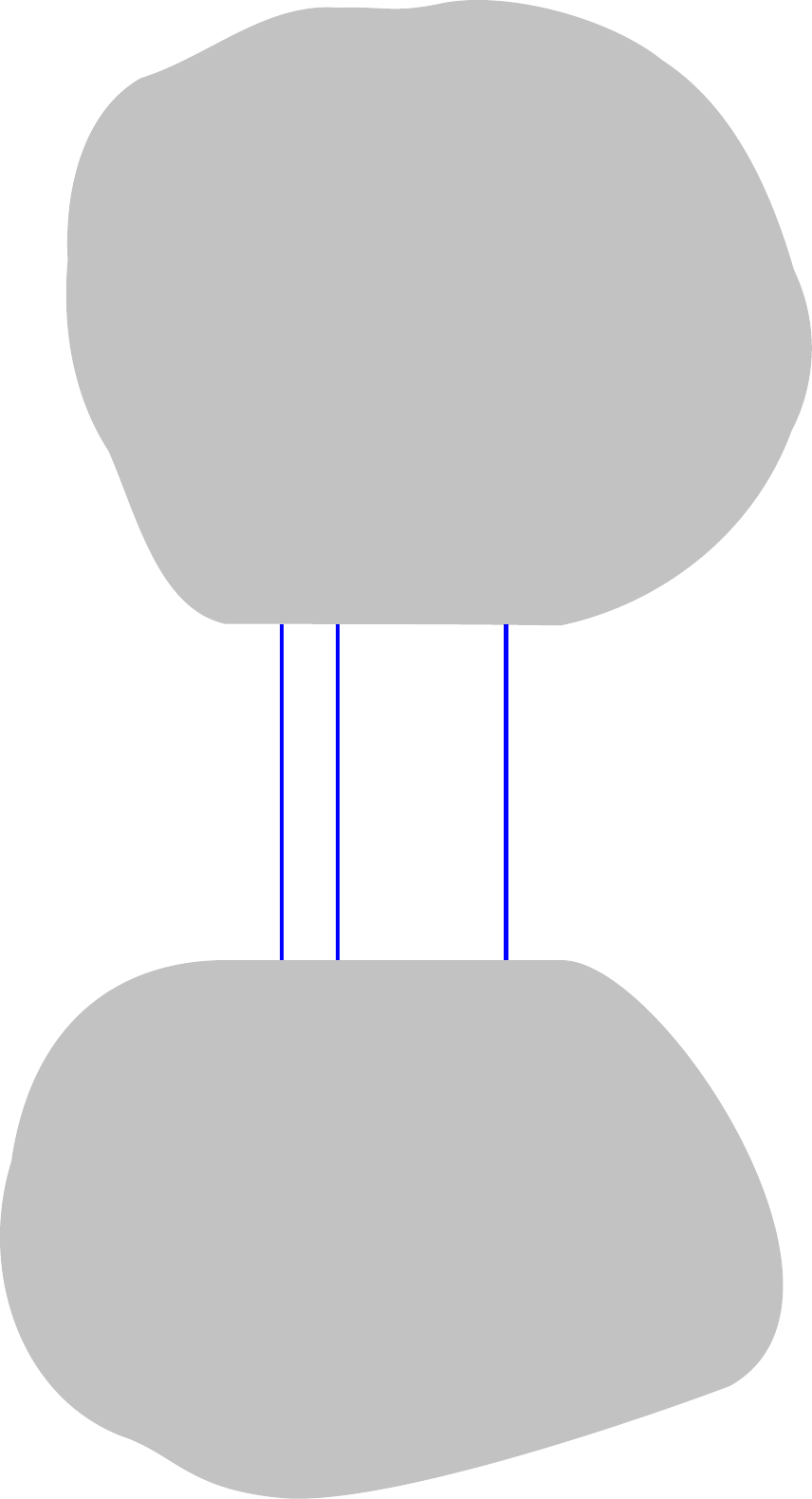}{20ex}
\putc{52}{48}{$\ms{\cdots}$}
\putc{52}{80}{A}
\putc{41}{17}{B}\longmapsto\sum_i
\epsh{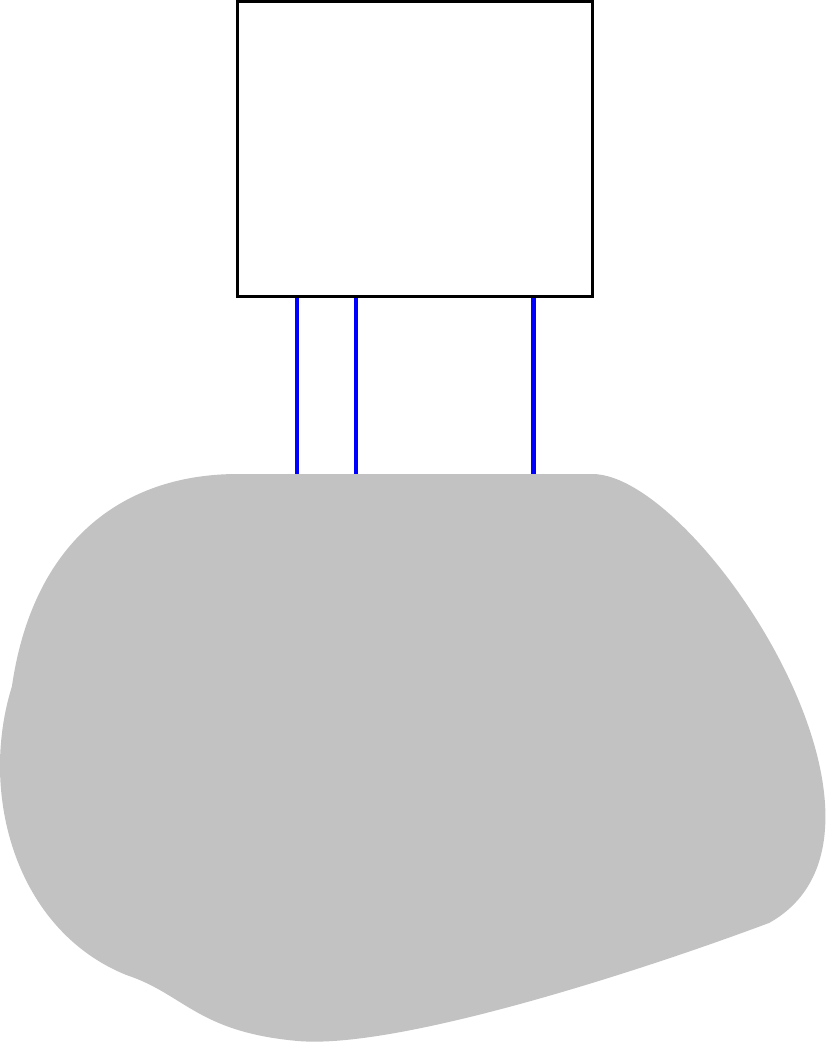}{14ex}
\putc{49}{87}{$\ms{x^i}$}
\putc{55}{64}{$\ms{\cdots}$}
\putc{44}{28}{B}
\sqcup
\epsh{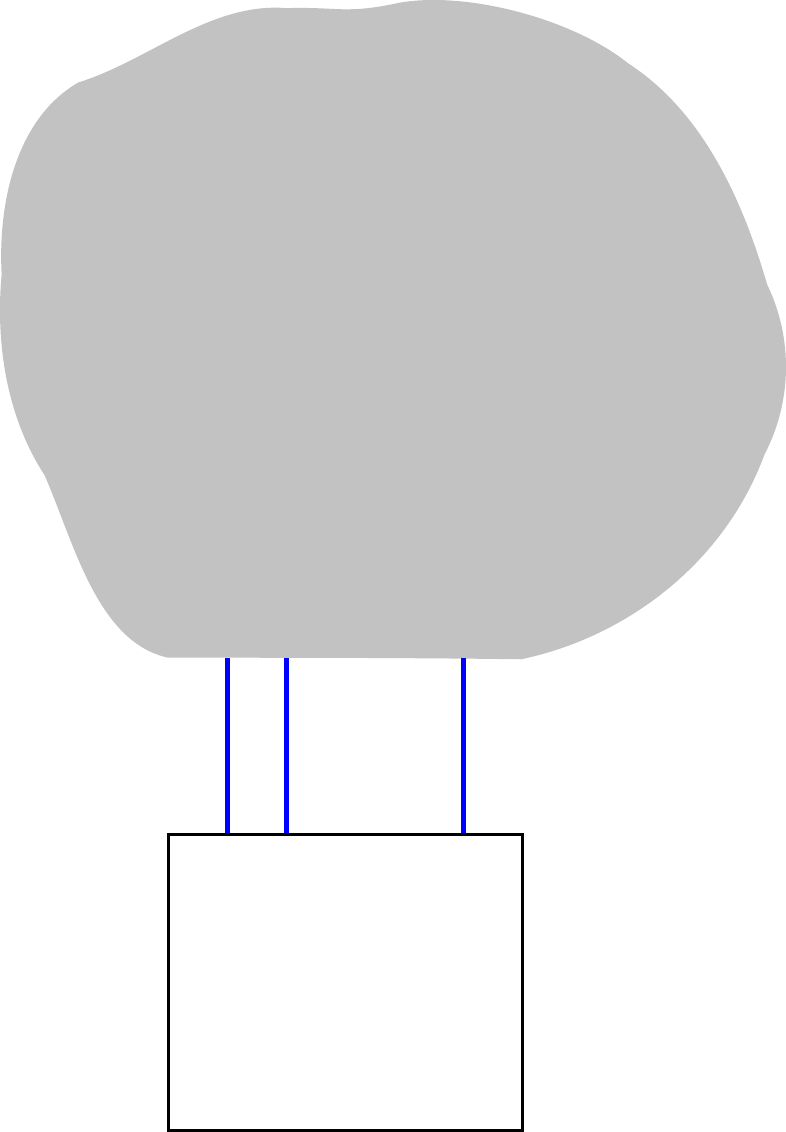}{14ex}
\putc{42}{14}{$\ms{x_i}$}
\putc{50}{34}{$\ms{\cdots}$}
\putc{46}{72}{A}
\qquad\text{ or }\qquad
\epsh{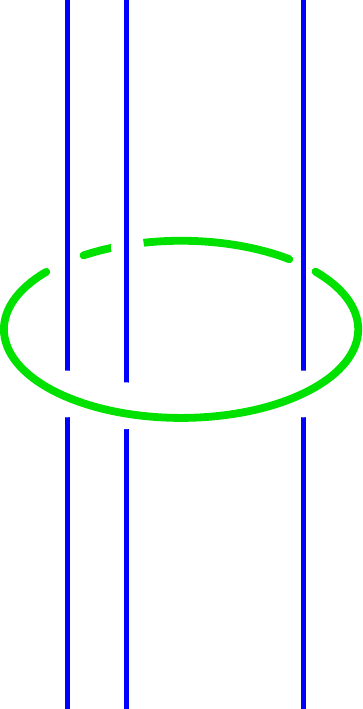}{14ex}
\putc{58}{17}{$\ms{\cdots}$}\longmapsto\sum_i
\epsh{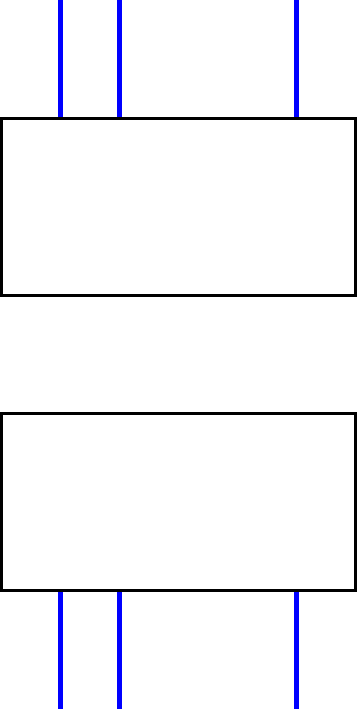}{14ex}
\putc{47}{71}{$\ms{x_i}$}
\putc{58}{92}{$\ms{\cdots}$}
\putc{48}{30}{$\ms{x^i}$}
\putc{57}{9}{$\ms{\cdots}$}
\]
 \caption{The cutting map $\chi_{M,\Sp^2}$: two representations
 depending on whether $\Sp^2$ is a separating (left) or a non-separating
 sphere in $M$ (right).}\label{F:cut}
\end{figure}
\textbf{3-handle: }
Given a framed sphere $\Sp^2$ in $M$,
there exists a \emph{cutting map}:
\[\chi _{M,{\Sp^2}}\colon\ \TSkein_\cat(M)\to\TSkein_\cat\bigl(M\bigl({\Sp^2}\bigr)\bigr)\]
sending parallel strands passing through the cutting sphere $S^2$ to the copairing $\Omega$, see Figure~\ref{fig:3handle}. A~representation based on a surgery presentation of $M$ is given in Figure~\ref{F:cut}. The impact on a~surgery presentation depends on whether $\Sp^2$ is separating. If it is separating, then $M$ is a~connected sum of $M_1$ and $M_2$ and can be presented as surgery on $L_1\cup L_2$. A surgery presentation of $M\bigl(\Sp^2\bigr) = M_1\sqcup M_2$ is given by$ L_1\sqcup L_2$. If $\Sp^2$ is not separating, let $\gamma$ be an embedded circle intersecting $\Sp^2$ once, then a small neighborhood of $\Sp^2\cup \gamma$ is diffeomorphic to a punctured $S^2 \times S^1$ and $M$ is obtained as a connected sum $M= M'\# S^2 \times S^1$. So we may obtain a surgery presentation of $M$ where $\Sp^2$ is represented by a green unknot. The effect of the 3-handle on skeins in both situations is depicted in Figure~\ref{F:cut}.

We say that the skein is in standard position with respect to $\Sp^2$ if its
intersection consists in~$n$ parallel edges in a rectangle \big(i.e., a disc of the form $\alpha\times [0,1]\subset \Sp^2\times [0,1]$ for some simple arc $\alpha\subset \Sp^2$\big) with at
least one edge colored by a projective module. We
now consider a~skein in standard position. Then the image by the
RT-functor of this rectangle is the identity of~$P$ for some
$P\in\Proj$. The cutting map $\chi_{M,\Sp^2}$ replaces the framed
sphere by the sums of graphs in two balls each containing a unique coupon
 colored with the dual basis of $\Hom_\cat(P,\unit)$ and~$\Hom_\cat(\unit,P)$.
\begin{proposition}
 The linear map $\chi_{M,\Sp^2}$ is well defined.
\end{proposition}
\begin{proof}
 We refer here to the proof of~\cite[Lemma 3.3]{CGPV2023a}
 which is completely similar. The main idea is that the naturality of
 $\Omega$ implies that the images of isotopic skeins are skein
 equivalent.
\end{proof}

\begin{figure}[t]
 \centering
\begin{tikzpicture}[baseline = 0pt]
\draw[gray] (-1,0.2) arc(180:540:1.5);
\fill[gray!10] (-1,0.2) arc(180:540:1.5);
\node[color = gray] at (1.2,1.2){$S^3$};
\node[rectangle, draw, fill=white] (T) at (0,0) {$T_{cut}$};
\draw[blue] (T) ..controls (0,1) and (1,1).. (1,0) .. controls (1,-1) and (0,-1)..(T) node[pos = 0.1, right]{$P$} node [pos = 0.1, sloped]{$<$};
\end{tikzpicture} $\quad \mapsto \quad \mt_P(T_{cut})$
\caption{The 4-handle}
 \label{fig:4handle}
\end{figure}
\textbf{4-handle:} Given a framed sphere $\Sp^3$ in $M$, the map $\chi_{M,\Sp^3}$ corresponding to
filling of a~3-sphere of $M=M'\sqcup S^3$ is given by
\[(M,\Gamma)\mapsto F'\bigl(\Gamma\cap S^3\bigr)\bigl(M',\Gamma\cap M'\bigr)\in\TSkein_\cat\bigl(M'\bigr).\]
See Figure~\ref{fig:4handle}.

\begin{theorem}\label{T:S}
 There exists a unique symmetric monoidal functor{\samepage
 \[
 \TSkein_\cat\colon\ \nc\cob\to\Vect
 \]
 extending $\TSkein_\cat\colon \man\to\Vect$
 such that $\TSkein_\cat(e_{\Sigma,\Sp})=\chi_{\Sigma,\Sp}$.}

 If $\cat$ is chromatic compact, then the functor extends to a symmetric monoidal functor on $\cob$:
 \[
 \TSkein_\cat\colon\ \cob\to\Vect.
 \]
\end{theorem}
\begin{proof}
 We only need to prove that the relation (R1)--(R5) are satisfied by
 $\TSkein_\cat$.
 \begin{enumerate}\itemsep=0pt
 \item[(R1)] Since $\TSkein_\cat\colon \man\to\Vect$ is functorial we have
 $\TSkein_\cat(e_{d\circ d'})=\TSkein_\cat(e_{d})\circ
 \TSkein_\cat(e_{d'})$. Also, since elements of
 $\TSkein_\cat(M)$ are defined by ribbon graphs up to isotopy,
 we clearly have $\TSkein_\cat(e_{d})=\Id$ if $d$ is isotopic to
 $\Id_\Sigma$.
 \item[(R2)] Since the construction of the maps $\chi_{M,\Sp}$
 are local, they are covariant under diffeomorphisms of the pair
 $({M,\Sp})$.
 \item[(R3)] Again, since the construction of the maps
 $\chi_{M,\Sp}$ are local, they commute for disjoint framed
 spheres.
 \item[(R4)] The 2-3-handle cancellations reduces to the chromatic
 identity \eqref{eq:chrP} as shown in Figure~\ref{F:R4c}. Indeed since the attaching framed
 2-sphere of the 3-handle intersects the belt circle of the
 $2$-handle once, the attaching circle for the 2-handle bounds a
 disc in the intermediate 3-manifold. This is why we can represent
 the green circle in Figure~\ref{F:R4c} as an unknot.

 The
 1-2-handle cancellations reduces to the defining property of
 the gluing map. Indeed, the sphere $S^2$ created by the 1-handle can not
 be separating since it is intersected once by the attaching
 $\Sp^1$ of the 2-handle. This means that we can represent the map
 $\chi_{M,\Sp^1}$ as in the left-hand side of Figure~\ref{F:glue}
 and the map $\chi_{M,\Sp^2}$ is then the left hand-side of Figure~\ref{F:surgery}
 turning the green unknot into red.

 The 3-4-handle
 cancellation relies on the fact that evaluating $F'$ on a cut
 3-ball is a skein relation.

 Finally, in the compact case, the
 0-1-handle cancellation is obvious since we can choose
 $\gm=\zeta^{-1}\Id_{P_\unit}$ as gluing morphism.
 \begin{figure}[t]
$$
\epsh{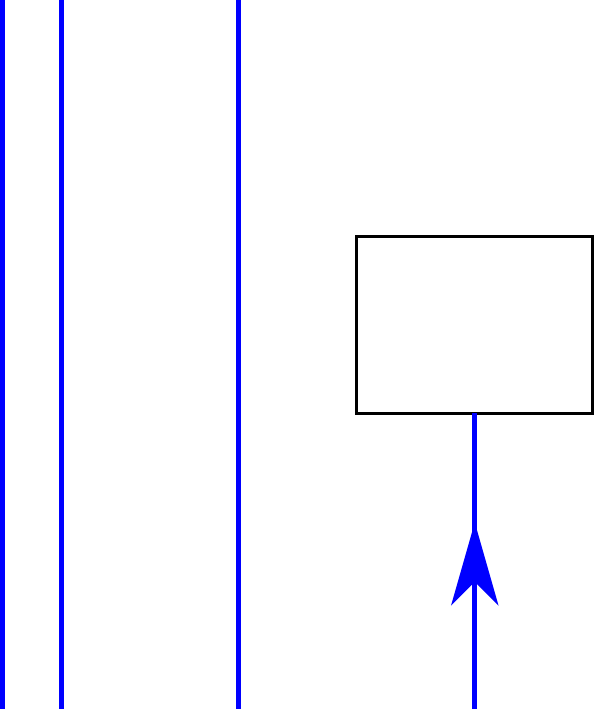}{14ex}
\putc{24}{51}{${\cdots}$}
\putc{80}{55}{${\ve}$}
\putlc{86}{18}{$\ms{P_\unit}$}
\longmapsto
\epsh{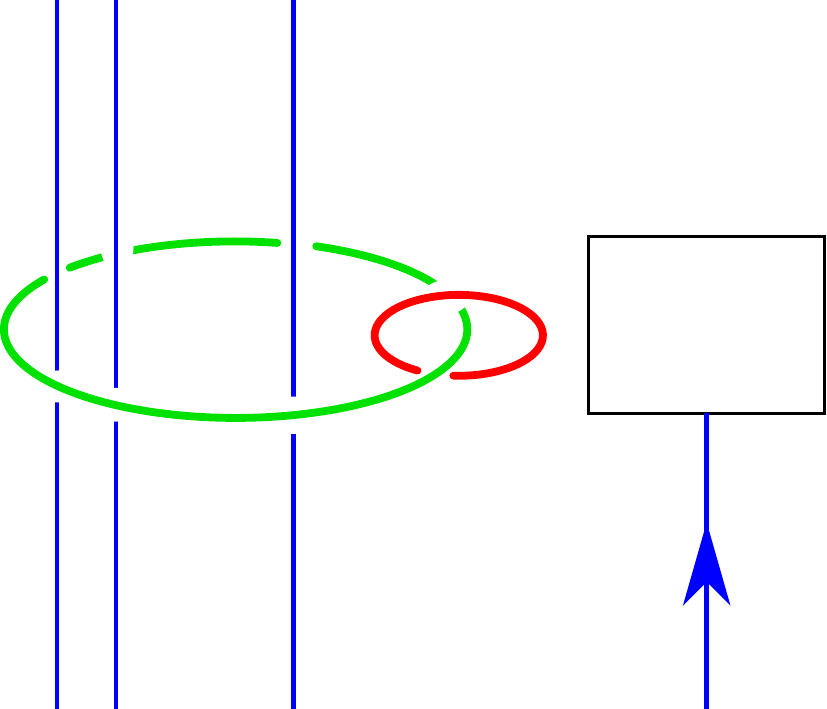}{14ex}
\putc{26}{19}{$\cdots$}
\putc{85}{55}{${\ve}$}
\putlc{89}{20}{$\ms{P_\unit}$}
\ \dot=
\epsh{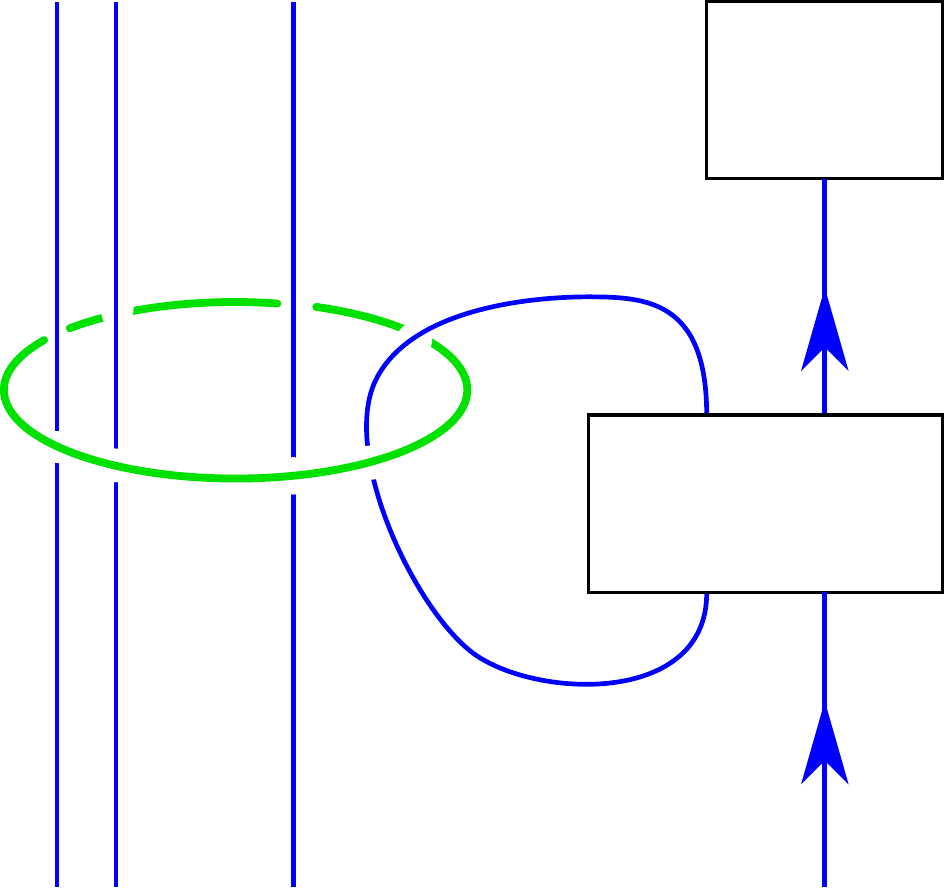}{14ex}
\putc{23}{22}{$\cdots$}
\putc{82}{43}{${\chr_{P_\unit}}$}
\putc{86}{91}{${\ve}$}
\putlc{92}{16}{$\ms{P_\unit}$}
\longmapsto
\epsh{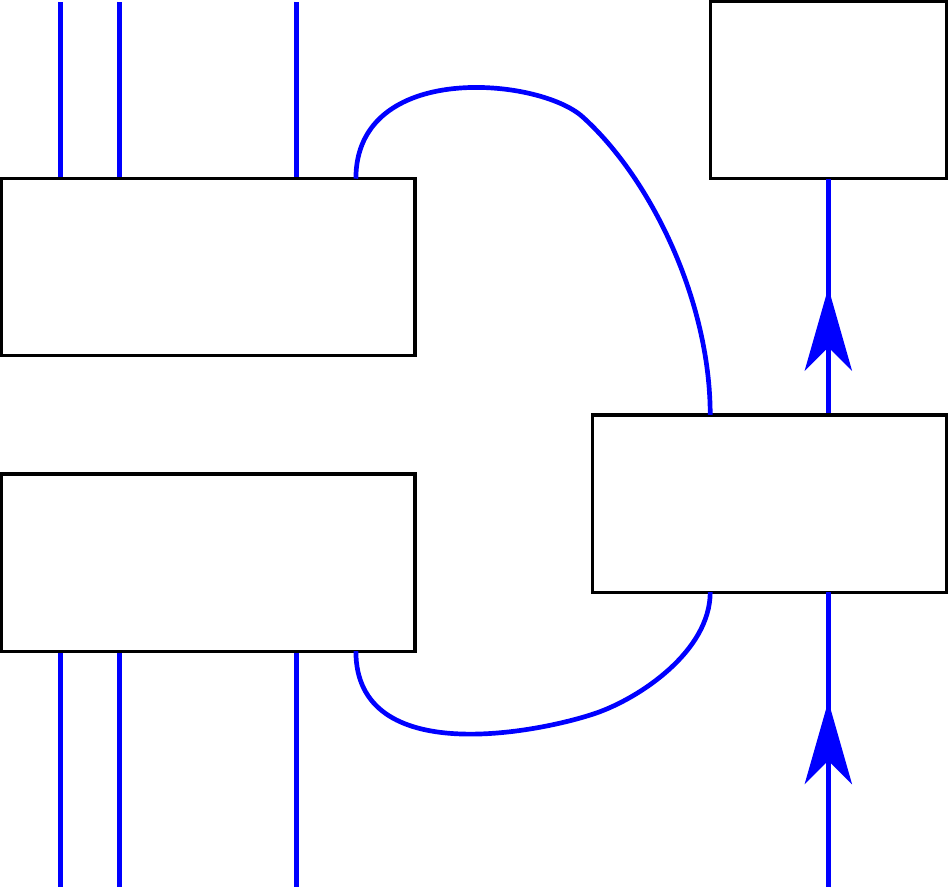}{14ex}
\putc{23}{13}{$\cdots$}
\putc{20}{37}{${x^i}$}
\putc{20}{70}{${x_i}$}
\putc{81}{43}{${\chr_{P_\unit}}$}
\putc{87}{90}{${\ve}$}
\putlc{93}{15}{$\ms{P_\unit}$}
\ \dot=
\epsh{fig24-0.pdf}{14ex}
\putc{24}{51}{${\cdots}$}
\putc{80}{55}{${\ve}$}
\putlc{86}{18}{$\ms{P_\unit}$}
$$
 \caption{The cancellation of a 2-handle by a 3-handle.}\label{F:R4c}
\end{figure}
 \item[(R5)] The maps $\chi_{M,\Sp}$ do not depend on the
 orientation of $\Sp$.
\hfill\qed
\end{enumerate}
\renewcommand{\qed}{}
\end{proof}

Let us illustrate the construction on two 4-manifolds:
\begin{example}
 The 4-manifold $S^3\times S^1$ is obtained by attaching successively exactly one~0,~1,~3 and~4-handle. The 0-handle introduces a tendril $\Gamma_0$ in $S^3$ times the global dimension $\zeta$. The 1-handle will glue the two ends of $\Gamma_0$ and turn it into a circle in $S^1\times S^2$ which goes around the $S^1$, with a coupon colored by $\gm$. The 3-handle will cut this circle and introduce a cutting morphism, which for $P_\unit$ is the composition of $\eta$ and $\varepsilon$. Finally, the 4-handle will evaluate this skein in $S^3$ using the modified trace,
 \begin{equation*}
1 \overset{\text{0-handle}}{\longmapsto} \zeta\begin{tikzpicture}[baseline = 8pt]
 \node[draw, rectangle, minimum height = 0.4cm, minimum width =
 0.7cm] (eta) at
 (0,0){$\eta$}; \node[draw, rectangle, minimum height = 0.4cm,
 minimum width = 0.7cm] (epsilon) at
 (0,1){$\varepsilon$}; \draw[blue] (eta) -- (epsilon)
 node[midway,sloped]{$>$} node[black, midway, right]{$P_\unit$};
\end{tikzpicture}
\overset{\text{1-handle}}{\longmapsto}
\zeta\begin{tikzpicture}[baseline = -8pt, yscale = 0.75]
 \node[draw, rectangle, minimum height = 0.4cm, minimum width =
 0.7cm] (gm) at (0,0){$\gm$};
\draw[blue] (gm)--++(0,0.5) arc(0:180:0.5) -- ++(0,-2)node[midway,sloped]{$>$} node[black, midway, right]{$P_\unit$} arc(180:360:0.5)--(gm) ;
\draw [line width = 5pt,white] (0.1,-0.5) arc (85:-245:0.5);
\draw [green] (0.1,-0.5) arc (85:-245:0.5);
\end{tikzpicture}
\overset{\text{3-handle}}{\longmapsto} \zeta\ \begin{tikzpicture}[baseline = 20pt]
 \node[draw, rectangle, minimum height = 0.4cm, minimum width =
 0.7cm] (eta) at
 (0,0){$\eta$}; \node[draw, rectangle, minimum height = 0.4cm,
 minimum width = 0.7cm] (gm) at
 (0,1){$\gm$}; \node[draw, rectangle, minimum height = 0.4cm,
 minimum width = 0.7cm] (epsilon) at
 (0,2){$\varepsilon$}; \draw[blue] (eta) -- (gm) node[midway,sloped]{$>$} node[black, midway, right]{$P_\unit$} -- (epsilon)
 node[midway,sloped]{$>$} node[black, midway, right]{$P_\unit$};
\end{tikzpicture}
\overset{\text{4-handle}}{\longmapsto} \zeta \mt(\eta \circ \varepsilon \circ \gm) = 1.
 \end{equation*}
 The topology of the ambient 3-manifolds is read in the green part, i.e., we go from the empty 3-manifold to $S^3$ to $S^2\times S^1$ to $S^3$ to the empty.

 If we do not suppose that $\cat$ is chromatic compact, we can simply remove the 0-handle and the 4-cobordism $S^3\times S^1$ minus a ball maps the tendril $\Gamma_0$ to $\mt(\eta \circ \varepsilon \circ \gm)$, i.e., to the scalar $\lambda$ such that $\varepsilon \circ\gm = \lambda \varepsilon$.
\end{example}
\begin{example}
 The 4-manifold $\CP^2$ is obtained by attaching successively exactly one 0, 2 and 4-handle. Again the 0-handle introduces $\zeta$ times the tendril $\Gamma_0$ in $S^3$. The 2-handle is attached along a 1-framed unknot and introduces its meridian in red in the surgered manifold. This manifold is $S^3$ and the meridian is a $(-1)$-framed unknot. It is turned blue using the tendril and the 4-handle will evaluate this skein in $S^3$ using the modified trace,
 \begin{gather*}
1 \overset{\text{0-handle}}{\longmapsto} \zeta\begin{tikzpicture}[baseline = 8pt]
 \node[draw, rectangle, minimum height = 0.4cm, minimum width =
 0.7cm] (eta) at
 (0,0){$\eta$}; \node[draw, rectangle, minimum height = 0.4cm,
 minimum width = 0.7cm] (epsilon) at
 (0,1){$\varepsilon$}; \draw[blue] (eta) -- (epsilon)
 node[midway,sloped]{$>$} node[black, midway, right]{$P_\unit$};
\end{tikzpicture}
\overset{\text{2-handle}}{\longmapsto}
\zeta\hskip-13pt\begin{tikzpicture}[baseline = 8pt]
\draw[green] (0.2,0) .. controls (2,0) and (0,2).. (0,0) ..controls (0,-1) and (-1,0).. (-0.2,0);
\draw[line width = 5pt, white] (0.5,0.2) arc (190:-140:0.5);
\draw[red] (0.5,0.2) arc (190:-140:0.5);
\begin{scope}[xshift = 2cm]
 \node[draw, rectangle, minimum height = 0.4cm, minimum width =
 0.7cm] (eta) at
 (0,0){$\eta$}; \node[draw, rectangle, minimum height = 0.4cm,
 minimum width = 0.7cm] (epsilon) at
 (0,1){$\varepsilon$}; \draw[blue] (eta) -- (epsilon)
 node[midway,sloped]{$>$} node[black, midway, right]{$P_\unit$};
\end{scope}
\end{tikzpicture}
=
\zeta\hskip-16pt\begin{tikzpicture}[baseline = 8pt]
\draw[green] (0.2,0) .. controls (2,0) and (0,2).. (0,0) ..controls (0,-1) and (-1,0).. (-0.2,0);
\begin{scope}[xshift = 1.2cm, xscale = -1, rotate = 90]
\draw[red] (0.2,0) .. controls (2,0) and (0,2).. (0,0) ..controls (0,-1) and (-1,0).. (-0.2,0);
\end{scope}
\begin{scope}[xshift =2.8cm]
 \node[draw, rectangle, minimum height = 0.4cm, minimum width =
 0.7cm] (eta) at
 (0,0){$\eta$}; \node[draw, rectangle, minimum height = 0.4cm,
 minimum width = 0.7cm] (epsilon) at
 (0,1){$\varepsilon$}; \draw[blue] (eta) -- (epsilon)
 node[midway,sloped]{$>$} node[black, midway, right]{$P_\unit$};
\end{scope}
\end{tikzpicture}
\overset{\text{4-handle}}{\longmapsto} \zeta \Delta_-.
 \end{gather*}
 Again in the non-compact, we remove the 0-handle and the 4-cobordism $\CP^2$ minus a ball maps the tendril $\Gamma_0$ to $\Delta_-$.
\end{example}
We now extract (even in the non-compact case) two scalar invariants of
$4$-manifolds: $\TSkein_\cat(W,T)$ for manifolds with an admissible
graph in the boundary and $\dot\TSkein_\cat(W)$ for connected closed
4-manifolds.
\begin{definition}
 Let $W$ be an oriented compact 4-manifolds with no closed components.
 A $\cat$-ribbon graph $T\subset(-\partial W)$ is \emph{admissible} if for each
 component $M$ of $-\partial W$, $T\cap M$ is admissible, i.e., if $T$
 represents an admissible skein of $\TSkein_\cat(-\partial W)$ (where
 the minus sign is for opposite orientation).
 If $T\subset(-\partial W)$ is admissible, then
 define the invariant
 \[\TSkein_\cat(W,T)=\TSkein_\cat\bigl(\wt W\bigr)(T),\]
 where \smash{$\wt W$} is $W$ seen
 as a cobordism from $-\partial W$ to
 $\varnothing$.
\end{definition}
\begin{definition}
 Let $W$ be a
 connected closed 4-manifold.
 Define
 \[
 \dot\TSkein_\cat(W)=\TSkein_\cat\bigl(\dot W,\Gamma_0\bigr)\in\FK,
 \]
 where
 \smash{$\dot W=W\setminus B^4$} is a once punctured $W$.
\end{definition}
If
$\cat$ is chromatic compact, by definition of the maps $\chi_{M,\Sp^0}$ (see Figure~\ref{fig:0handle}), we
have for any closed connected 4-manifold $W$:
\begin{equation*}
 \TSkein_\cat(W)=\zeta\dot\TSkein_\cat(W)\Id_\FK.
\end{equation*}
For example, \smash{$\dot\TSkein_\cat\bigl(S^4\bigr)=1$} and \smash{$\TSkein_\cat\bigl(S^4\bigr)=\zeta\Id_\FK$}.
\subsection{Properties}
Remember that in this section we are assuming that $\cat$ is chromatic non-degenerate so that in particular it has an m-trace $\mt$, chromatic morphism $\chr$ and gluing morphism $\gm$.
\begin{proposition}\label{prop:changeOfModifiedTraceTQFT}
 Let $\kappa\in\FK^*$, then
 \begin{enumerate}\itemsep=0pt
 \item[$(1)$] $\mt':=\kappa\mt$ is a non-degenerate
 m-trace
 on $\Proj$,
 \item[$(2)$] its associated copairing is given by
 \smash{$\Omega'_P=\frac1\kappa\Omega_P$}, and
 \smash{$\Gamma_0'=\frac1\kappa\Gamma_0$},
 \item[$(3)$] $\chr'=\kappa\chr$ is a chromatic morphism associated to
 $\mt'$,
 \item[$(4)$] \smash{$\gm'=\frac{1}{\kappa^2}\gm$} is a gluing morphism, and in the
 compact case $\zeta'= \kappa^2\zeta$.
 \end{enumerate}
 Finally, the TQFT $\TSkein'_\cat$ associated to $\mt'$ satisfies
 \smash{$\TSkein'_\cat(W)=\kappa^{\chi(W)}\TSkein_\cat(W)$} where $\chi$ is the Euler characteristic.
\end{proposition}
\begin{proof}
The first four points are immediate from the definitions. In the compact case, the 0-handle map becomes $\chi _{M,{\Sp^{-1}}}'=\kappa \chi_{M,{\Sp^{-1}}}$, as \smash{$\zeta'\Gamma_0' = \kappa \zeta\Gamma_0$}. The 1-handle map becomes $\chi _{M,{\Sp^{0}}}'=\smash{\frac1\kappa \chi_{M,{\Sp^{0}}}}$ as it maps a $\Omega'_{P_\unit}$ to a $\gm'$. The 2-handle map becomes $\chi _{M,{\Sp^{0}}}'=\kappa \chi_{M,{\Sp^{0}}}$ as $\chr'=\kappa\chr$. The 3-handle map becomes \smash{$\chi _{M,{\Sp^{0}}}'=\frac1\kappa \chi_{M,{\Sp^{0}}}$} as \smash{$\Omega_P'=\frac1\kappa\Omega_P$}. The 4-handle map becomes \smash{$\chi _{M,{\Sp^{0}}}'=\kappa \chi_{M,{\Sp^{0}}}$} as $\mt'=\kappa\mt$.

Therefore, for a 4-bordism $W$ decomposed using $n_i$ $i$-handles, $0\leq i\leq 4$, one has
\[
\TSkein_\cat'(W) = \kappa^{n_4-n_3+n_2-n_1+n_0}\TSkein_\cat(W) = \kappa^{\chi(W)}\TSkein_\cat(W).
\tag*{\qed}
\]
\renewcommand{\qed}{}
\end{proof}

\begin{theorem}\label{T:invertible}
 The TQFT $\TSkein_\cat$ is invertible if and only if $\cat$ is factorizable.
\end{theorem}
\begin{proof}
 First we prove the necessity of the theorem: recall that \smash{$\TSkein_\cat\bigl(S^3\bigr)\simeq \FK$} is generated by
 the skein \smash{$\bigl(S^3,\Gamma_0\bigr)$}. Let $G$ be a projective generator with
 a~unique indecomposable factor \smash{$P_\unit\tto i G\tto p P_\unit$}.
 Then by naturality of $\Lambda$ and since $G$ contains a single copy of $P_1$, we have ${\Lambda_G=i\Lambda_{P_\unit} p}$. Consider the subspace of
 \smash{$\TSkein_\cat\bigl(S^2\times S^1\bigr)$} generated by graphs
 $\{O_f\}_{f\in\End_\cat(G)}$ with a~unique coupon colored by
 $f\in\End_\cat(G)$ and a~unique edge of the form $\{pt\}\times S^1$.
 Consider the two cobordisms \smash{$W_2,W_3\colon S^2\times S^1\to S^3$} given by
 gluing a 2-handle (resp. a 3-handle) to $S^2\times S^1\times[0,1]$
 respectively along $\{pt\}\times S^1\times \{1\}$ and
 $S^2\times \{pt\}\times \{1\}$. Then
 \begin{gather*}
 \TSkein_\cat(W_2)(O_f)=\mt_G\bigl(\Delta_0^Gf\bigr)\in\FK\simeq\TSkein_\cat\bigl(S^3\bigr)
 \!\qquad \text{and} \qquad\!
 \TSkein_\cat(W_3)(O_f)=\mt_G(\Lambda_Gf)\in\FK\simeq\TSkein_\cat\bigl(S^3\bigr).
 \end{gather*}
 In
 particular, for the gluing morphism $\gm$,
 \smash{$\TSkein_\cat(W_2)(O_{i\gm p})=1=\TSkein_\cat(W_3)(O_{\Id})=1$} so
 the two maps are non-zero. If $\TSkein_\cat$ is invertible,
 \smash{$\dim_\FK\bigl(S^2\times S^1\bigr)=1$} and there exists $\zeta\in\FK^*$ such
 that $\TSkein_\cat(W_3)=\zeta\TSkein_\cat(W_2)$. Then for any
 $f\in\End_\cat(G)$,
 \smash{$\mt_G\bigl(\Delta_0^Gf\bigr)=\TSkein_\cat(W_3)(O_f)=
 \zeta\TSkein_\cat(W_2)=\zeta\mt_G(\Lambda_Gf)$}. Finally, by
 non-degeneracy of the m-trace, $\Delta_0^G=\zeta\Lambda_G$.

 Now we prove the sufficiency of the theorem: we suppose $\cat$ is factorizable and we show that
 for any connected 3-manifold $M$,
 $\dim_\FK(\TSkein_\cat(M))=1$. This is true because the image of any
 1-surgery given by a 2-handle can be inverted,
 \begin{align*}
 \epsh{fig23-11}{10ex}
 &\stackrel{\chi_{M,\Sp^1}}{\longmapsto}
 \epsh{fig25-0}{12ex}\stackrel{\chi_{M',\Sp^1}}
 \longmapsto\epsh{fig25-1}{12ex}\\
 &=\epsh{fig25-2}{12ex}=\epsh{fig25-3}{12ex}\dot=\ \zeta\epsh{fig23-11}{10ex}\,.
 \end{align*}
 Here the first map is the image of any $S^1$-surgery on $M$ and the
 second is the image of an~appropriate second $S^1$-surgery; the
 first equality is an isotopy in the manifold obtained by sliding the
 second red curve along the first green curve, the second equality
 comes from the fact that topologically on the level of the $3$-manifolds, a surgery along a meridian of
 a $\Sp^1$-surgery component cancels both components and the last equivalence is a skein
 equivalence due to the factorizability of $\cat$.

Then we check that every cobordism induces an isomorphism. It is
 immediate from the definition that 0-handles and 4-handles are
 isomorphisms. The proof for 2-handles is given above. Note that a 1-handle
 followed by a 3-handle glued on the belt sphere created by the $1$-handle is a scalar times the identity. Indeed, the
 1-handle will introduce a gluing morphism (which is a scalar times
 the identity of $P_\unit$ by assumption) from a pair of coupons
 $\varepsilon$ and $\eta$. Then the 3-handle will cut it, turning it
 back to a pair of coupons $\varepsilon$ and $\eta$. This shows that
 1-handles are injective and 3-handles surjective. Because every
 skein module is 1-dimensional, they are also bijective.
\end{proof}

\begin{theorem}\label{T:4TQFT-3M}
 Assume $\cat$ is twist non-degenerate and chromatic non-degenerate
 and let $M$ be a closed connected $3$-manifold. Fix any connected
 bordism $W$ with $\partial W=-M$ such that the cobordism $W\colon M\to \varnothing$
 is made by gluing $2$-handles and a unique $4$-handle on ${M\times[0,1]}$.
 Let~$(r,s)$ be the signature of $-W$,then for any admissible skein
 $T$ in $M$,
 \begin{equation*}
 \TSkein_\cat(W,T)=\Delta_+^r\Delta_-^s\ThreeManInv'(M,T)\in\FK,
 \end{equation*}
 $($where $\ThreeManInv'$ is defined in Section~$\ref{sec:3mfld})$.
\end{theorem}
\begin{proof}
 Let
 \smash{$-W=B^4\bigcup_{N(L)= \bigsqcup_{i=1}^n(\partial B^2)\times B^2}
 \bigl(\bigsqcup_{i=1}^nB^2\times B^2\bigr)\colon\varnothing\to M$} where $N(L)$ is a
 tubular neighborhood of a $n$ component link $L$ in $S^3$. Then
 $L^{\rm red}\cup T$ is a link presentation in $S^3$ of the pair $(M,T)$
 while $L^{\rm green}\cup T$ represents $T$ in the manifold $M$. Now
 \[
 \TSkein_\cat\bigl(\dot W\bigr)(M,T)=\TSkein_\cat\bigl(\dot W\bigr)(L^{\rm green}\cup
 T)=L^{\rm red}\cup T\subset S^3.
 \]
 Thus
 \smash{$\TSkein_\cat(W,T)=F'\bigl(L^{\rm blue}\cup T\bigr)$} where \smash{$\bigl(L^{\rm blue}\cup T\bigr)$} is
 obtained from $L^{\rm red}\cup T$ by doing red to blue modifications.
 Finally (see~\cite[Proposition 4.5.11]{GS99}), the linking
 matrix of $L$ is the intersection form on $H^2(-W)$ so the signature
 of $-W$ is $(r,s)$.
\end{proof}
\begin{proposition}\label{P:conn} Behavior under connected sums:
 \begin{itemize}\itemsep=0pt
 \item The invariant of closed connected $4$-manifolds
 $\dot\TSkein_\cat(W)$ is multiplicative under connected sum.
 \item If $W$ is a closed connected $4$-manifold and
 $W'\colon M'\to N'\in\cob^{nc}$ $($resp. $W'\in\cob$ if $\cat$ is chromatic compact$)$, both non-empty, then
 \[
 \TSkein_\cat\bigl(W\#W'\bigr)=\dot\TSkein_\cat(W)\TSkein_\cat\bigl(W'\bigr)
 \in\Hom_\FK\bigl(\TSkein_\cat\bigl(M'\bigr),\TSkein_\cat\bigl(N'\bigr)\bigr).
 \]
 \item For non-empty $4$-manifolds $W$, $W'$ containing admissible graphs
 $T$, $T'$ in their boundaries,
 \[\TSkein_\cat\bigl(W\#W',T\cup T'\bigr)=
 \begin{cases}
 \zeta^{-1}\TSkein_\cat(W,T)\TSkein_\cat(W,T)&\text{if $\cat$ is chromatic compact,}\\
 0&\text{else.}
 \end{cases}\]

 \item If $\cat$ is chromatic compact, for two non-empty $4$-cobordisms
 $W\colon\! M\to N$ and ${W'\colon\! M'\to N'}$,
 \[
 \TSkein_\cat\bigl(W\#W'\bigr)=\zeta^{-1}\TSkein_\cat(W)\otimes\TSkein_\cat\bigl(W'\bigr)\colon \
 \TSkein_\cat(M)\otimes\TSkein_\cat\bigl(M'\bigr)\to\TSkein_\cat(N)\otimes\TSkein_\cat\bigl(N'\bigr).
 \]
 \end{itemize}
\end{proposition}
\begin{proof}
 The admissible skein module $\TSkein_\cat\bigl(S^3\bigr)$ is one-dimensional
 and generated by
\[
\Gamma_0=\tikz[baseline = -3pt] \node[scale = 0.6]at (0,0) {\begin{tikzpicture}
 \node[draw, rectangle, minimum height = 0.4cm, minimum width =
 0.7cm] (eta) at
 (0,0){$\eta$}; \node[draw, rectangle, minimum height = 0.4cm,
 minimum width = 0.7cm] (epsilon) at
 (0,1){$\varepsilon$}; \draw[blue] (eta) -- (epsilon)
 node[midway,sloped]{$>$} node[black, midway, right]{$P_\unit$};
\end{tikzpicture}};.
\]
 For a closed connected 4-manifold $W$, the twice
 punctured cobordism $\TSkein_\cat\bigl(\ddot W\bigr)\colon\TSkein_\cat\bigl(S^3\bigr)\to\TSkein_\cat\bigl(S^3\bigr)$ acts as multiplication by the scalar \smash{$\dot\TSkein_\cat(W)$}. Composition corresponds to connected sum for the twice-punctured cobordisms, and to multiplication for the scalars. The second point is obtained by adding a cancelling pair of 3- and 4-handles to $W'$. Then connected sum with $W$ precomposes by $\TSkein_\cat\bigl(\ddot W\bigr)$ before the 4-handle, hence simply multiplies by $\dot\TSkein_\cat(W)$.

 Let $P\colon S^3\sqcup S^3\to S^3$ be the three-dimensional pair of pants,
 namely a 3-punctured $S^4$ which can be seen as a unique 1-handle.
 The cobordism
 $\dot{(W\#W')}\colon (-\partial W)\sqcup (-\partial W')\to S^3$ factors as
 \smash{$\dot{W\#W'}=P\circ \bigl(\dot W\sqcup \dot W'\bigr)$}.

The map
 $\TSkein_\cat(P):\FK\otimes\FK=\FK\to\FK$ is a scalar morphism which
 sends $\Gamma_0\otimes\Gamma_0$ to the unique graph with 3 coupons
 colored by $\eta$, $\gm$ and $\ve$. Since $\ve\circ\gm=0$ unless
 $\gm$ is invertible (i.e., $\cat$ is chromatic compact by Lemma~\ref{lem:invertibleg}), the second case follows. Let us now assume that
 $\cat$ is chromatic compact and let us use
 $\gm=\zeta^{-1}\Id_{P_\unit}$ for the gluing morphism. Then
 $\TSkein_\cat(W\#W',T\cup T')=
 \TSkein_\cat(W,T)\TSkein_\cat(W',T')F'(\TSkein_\cat(P)(\Gamma_0\otimes\Gamma_0))=
 \zeta^{-1}\TSkein_\cat(W,T)\TSkein_\cat(W',T')$.

 For the last statement, since every object of $\cob$ is dualizable
 we can suppose that $N=N'=\varnothing$. Then the statement follows
 from the previous identity since for any
 $T\otimes T'\in\TSkein_\cat(M)\otimes\TSkein_\cat(M')
 \cong\TSkein_\cat(-\partial(W\sqcup W'))$, we have
 $\TSkein_\cat(W\#W')(T\otimes T')=\TSkein_\cat(W\#W',T\cup T')$.
\end{proof}
\begin{proposition}
The category $\cat$ is chromatic compact if and only if $\dot\TSkein_\cat\bigl({S^1\times S^3}\bigr)\neq 0$.
\end{proposition}
\begin{proof}
 A handle decomposition of the punctured bordism
 \smash{$\dot{S^1\times S^3}\colon S^3\to \varnothing$} is given by a~1-handle followed by a
 3-handle glued on its belt sphere and a closing 4-handle. The skein $\Gamma_0$ is sent to a circle with a coupon
 $\gm$ in $\TSkein_\cat\bigl(S^2\times S^1\bigr)$ which is then cut into the
 closure of $\gm \circ \Lambda_{P_\unit}$ in
 $\TSkein_\cat\bigl(S^3\bigr)$. This is non-zero if and only if $\gm$ is
 invertible. The statement follows then by Lemma~\ref{lem:invertibleg}.
\end{proof}
\begin{proposition}
 If $\cat$ is twist non-degenerate or if
 $\dot\TSkein_\cat\bigl(S^2\times S^2\bigr)\neq0$, then $\TSkein$ does not distinguish
 exotic pairs of cobordisms.
\end{proposition}
\begin{proof}
 Since $\dot\TSkein_\cat\bigl(\pm\CP^2\bigr)=\Delta_{\mp}$, the category is twist non-degenerate if and only if
 \[
 \dot\TSkein_\cat\bigl(\CP^2\bigr)\dot\TSkein_\cat\bigl(-\CP^2\bigr)\neq0.
 \]
 As
 said in the introduction, Gompf \cite{Go} showed that two
 homeomorphic compact orientable $4$-manifolds (possibly with
 boundary) become diffeomorphic after some finite sequence of
 connected sums with $S^2\times S^2$; the same is true for connected sums with complex projective planes (or their opposites) since $\bigl(S^2 \times S^2\bigr)\#\CP^2$ is diffeomorphic to $\CP^2\#\CP^2\#\bigl(-{\CP}^2\bigr)$.
The statement then follows from Proposition~\ref{P:conn}.
\end{proof}
\begin{proposition}
 Let $\cat$ be non-semisimple and chromatic compact then ${\TSkein_\cat\bigl(B^2\times S^2,\Gamma_0\bigr)=0}$ \big(where $\Gamma_0$ is the graph of the right-hand side Figure~$\ref{fig:0handle}$, contained in a ball in $\partial B^2\times S^2$\big).
Equivalently,
the
 skein ${\color{red}\mathsf O}\,\cup\Gamma_0$ is zero in
 $\TSkein_\cat\bigl(S^3\bigr)$ $($where ${\color{red}\mathsf O}$ denotes a red unknot$)$.
\end{proposition}
\begin{proof}
As $\cat$ is chromatic compact,
\smash{$\Delta_0^{P_\unit}=\zeta\Lambda_{P_{\unit}}$} for some $\zeta\in \FK^*$. Therefore, ${\color{red}\mathsf O}\,\cup\Gamma_0=\zeta \Gamma_0\sqcup \Gamma_0$ which is zero by the skein relation which evaluates only one $\Gamma_0$ via the RT functor. This proves the second statement. The first follows by observing that $B^2\times S^2$ is obtained by gluing a $2$-handle to $S^1\times S^2\times [-1,1]$ along $S^1\times B^2\times \{1\}$ and then filling the result by a $4$-handle. Then we present $\bigl(S^1\times S^2,\Gamma_0\bigr)$ as ${\color{green}\mathsf O}\,\cup\Gamma_0$ and the first operation consists of changing ${\color{green}\mathsf O}$ to ${\color{red}\mathsf O}$.
\end{proof}

\section{Examples and relations with other works}\label{S:Examples}
\subsection{Semisimple case}\label{SS:ExSS} Using the chain-mail construction of~\cite{Roberts}, we can rewrite our construction in the semi-simple case as a state sum. We then recover the Crane--Yetter--Kauffman 4-manifold invariant associated with a semi-simple fusion category $\cat$. The chain mail construction has been carried out for the state spaces in~\cite{Tham21} in characteristic 0 and we will use this description.

The state-sum 4-manifold invariant was defined for all fusion categories in~\cite{CYK}. It was first only defined in the modular case by Crane and Yetter, and mentioned to extend to a TQFT there. It is well-known that in the modular case the TQFT is invertible, and the associated 4-manifold invariant is classical, namely only depends on the signature and Euler characteristic, see~\cite[Proposition 6.2]{CYK}. Note however that given the extra data of a boundary condition, which corresponds to the empty skein in our description, this TQFT recovers the Reshetikhin--Turaev invariants of 3-manifolds. It was shown in~\cite{BB18} that when the category $\cat$ is not modular, i.e., has non-trivial M\"uger center, it is no longer true that the 4-manifold invariants depend only on the signature and the Euler characteristic, but at least also on the fundamental group. It is however still almost trivial on simply connected manifolds, see~\cite{BB18}.
\begin{theorem} Let $\cat$ be a fusion category over an algebraically closed field of characteristic $0$. Choose $\mt=\operatorname{qTr}$ the standard categorical trace. Then the TQFT $\TSkein_\cat$ coincides with the Crane--Yetter--Kauffman TQFT.
\end{theorem}
\begin{proof}
In the semi-simple case, the admissible skein modules are the usual skein modules.

Let us describe the TQFT $\TSkein_\cat$ in this setting. Let $\{S_i\}_{i\in I}$ be a set of representatives of the isomorphism classes of simple objects of $\cat$. We described in Proposition~\ref{prop:semisimpleCompact} the chromatic morphism $\chr_P=(\oplus_{i\in I} \qdim(S_i)\Id_{S_i})\otimes \Id_P$ and the gluing morphism \smash{$\gm=\frac{1}{d(\cat)}\Id_\unit$}, hence $\zeta = d(\cat)$.
\begin{enumerate}\itemsep=0pt
\item[0.] A 0-handle introduces $d(\cat)\cdot\varnothing$ in the created $S^3$.
\item[1.] A 1-handle on a skein disjoint from the attaching sphere multiplies by \smash{$\frac{1}{d(\cat)}$} without affecting the skein.
\item[2.] A 2-handle introduces a Kirby-colored circle along the attaching sphere.
\item[3.] A 3-handle cuts the strands passing through the canceling 2-handle represented by the green arc by introducing a copairing.
\item[4.] A 4-handle does Reshetikhin--Turaev evaluation on a skein in $S^3$.
\end{enumerate}
This is exactly the description of~\cite[Definition~5.11]{Tham21}. The only non-trivial check is for the 3-handle. Let $V$ denote the color of the strand passing trough the green arc. In the description there, one splits $V$ as a direct sum of simples, and only keeps the $\unit$ components. In our construction, we choose a~basis $\{f_j\}_j$ of $\Hom_\cat(V,\unit)$ and the dual basis $\bigl\{f_j^\star\bigr\}_j$ of $\Hom_\cat(\unit,V)$ with respect to the
m-trace.
Then indeed $f_j^\star\circ f_j$ is an idempotent of $V$ corresponding to a~$\unit$-component, and the two constructions agree.

It is shown in~\cite[Sections 5.2 and 3.4]{Tham21} that the TQFT of skein modules and the construction above coincides with the Crane--Yetter TQFT. Skein modules are introduced there in~Definition~5.6, and the linear maps induced by 4-manifolds in Definition~5.11. The Crane--Yetter state spaces (as outlined by Yetter, see also~\cite[Section 7.1]{BB18}) are introduced in Proposition 3.50, and the linear maps induced by 4-manifolds in Definition~3.46. The fact that this recovers the Crane--Yetter invariants is proven in Theorem~3.61. The isomorphism between the skein and Crane--Yetter state spaces is given in Lemmas~5.22 and~5.24. The fact that this isomorphism is natural and respects 4-cobordisms is Theorem~5.26.
\end{proof}

\subsection[The example of sl\_2]{The example of $\boldsymbol{\mathfrak{sl}_2}$}\label{SS:Ex-sl2}
We study the category of modules over a partially unrolled version of
the small quantum group associated with $\mathfrak{sl}_2$, at roots of
unity. Varying the parameters, this gives examples of possibly
non-factorizable and possibly twist-degenerate chromatic compact
categories. In particular, our construction applies and gives a plain
(3+1)-TQFT $\TSkein_\cat$. We expect this TQFT to be similar to the
construction of~\cite[Section 9.2]{BD21} on 2-handlebodies. In
particular, we expect that a result similar to~\cite[Theorem
8.1]{BD22} applies, and that the associated invariant of closed
connected 4-manifolds only depends on the Euler characteristic,
signature and (in the twist-degenerate case) spin status. The whole
TQFT might be of greater interest though.
\newcommand{\pp}{r}
\begin{definition}
 Let $\FK=\C$ and $m$, $n$, $r$ be positive integers such that $n\vert m$
 and $r\ge2$.
 Let~$q$ be a~primitive $2r$-th
 root of unity and choose \smash{$q^\frac{2}{mn}$} a~primitive $mn{\pp}$-th root
 of unity. Note that~\smash{$(q^\frac{2}{mn})^{\frac{m}{n}}$}
 is a primitive
 $n^2{\pp}$-th root of unity. Let
 \begin{align*}
 &H=\mathfrak{u}^{m,n}_{q}(\mathfrak{sl}_2)
 = \C\biggl\langle E,F,{\smk} \ \bigg|\,
 E^{{\pp}}=F^{{\pp}}=0,\, \smk^{\,mn{\pp}}=1, \,\smk E=q^\frac{2}{m}E\smk, \\
& \hphantom{H=\mathfrak{u}^{m,n}_{q}(\mathfrak{sl}_2)
 = \C\biggl\langle E,F,\smk \,\bigg|\, }{}
 \smk
 F=q^{-\frac{2}{m}}F\smk,\, EF-FE=\frac{K-K^{-1}}{q-q^{-1}} \biggr\rangle,
 \end{align*}
 where $K=\smk^{\, m}$.
 The algebra $H$ can be given the structure of a Hopf algebra with coproduct $\Delta$, counit $\varepsilon$ and antipode $S$ defined by
\begin{alignat*}{4}
&\Delta(E)=1\otimes E + E \otimes K,\qquad && \varepsilon(E)=0,\qquad && S(E)= -EK^{-1},&\\
&\Delta(F)= K^{-1}\otimes F+F\otimes 1,\qquad && \varepsilon(F)=0,\qquad && S(F)=-KF,&\\
&\Delta(\smk)= \smk\otimes \smk,\qquad && \varepsilon(\smk)=1,\qquad && S(\smk)=\smk^{\,-1}.&
 \end{alignat*}
 \end{definition}

 Note that $H$ contains a version of the small quantum group at even
 root of unity as the sub-Hopf-algebra generated be $E$, $F$ and $K$.
 Let
$\cat=H-{\rm mod}$
 be the
 category of finite-dimensional left
 $H$-modules.
\noindent
For $i\in\Z/mn{\pp}\Z$, denote
\[
\smk_{\; i} = \frac{1}{mn{\pp}}\sum_{j=0}^{mn{\pp}-1} q^{\frac{-2ij}{mn}}\smk^{\, j}.
\]
Then
\begin{align*}
&\smk\smk_{\; i} = q^{\frac{2i}{mn}}\smk_{\; i},\qquad \smk_{\; i}\smk_{\; j}=\delta_{i,j}\smk_{\; i}, \qquad \sum_{i=0}^{mn{\pp}-1} \smk_{\; i}=1, \\
&E\smk_{\; i}=\smk_{\; i+n}E, \qquad\text{and}\qquad F\smk_{\; i}=\smk_{\; i-n}F.
\end{align*}
Namely, $\smk_{\; i}$ acts as the projection on the \smash{$q^{\frac{2i}{mn}}$} eigenspace of $\smk$.
\begin{proposition}
 The Hopf algebra $H=\mathfrak{u}^{m,n}_{q}(\mathfrak{sl}_2)$ is
 ribbon where the R-matrix and twist are given by:
 \begin{gather*}
 R=\Biggl(\sum_{i,j=0}^{mn{\pp}-1}
 q^{\frac{2ij}{n^2}}\smk_{\; i}\otimes
 \smk_{\; j}\Biggr)\cdot\Biggl(\sum_{k=0}^{{\pp}-1}
 \frac{\{1\}^{2k}}{\{k\}!}q^\frac{k(k-1)}{2} E^k\otimes
 F^k\Biggr),
 \\
 \theta=K^{{\pp}-1}\sum_{k=0}^{{\pp}-1}
 \frac{\{1\}^{2k}}{\{k\}!}q^\frac{k(k-1)}{2}
 S\bigl(F^k\bigr)\Biggl(\sum_{i=0}^{mn{\pp}-1}
 q^{\frac{-2i^2}{n^2}}\smk_{\; i}\Biggr) E^k.
 \end{gather*}
 where $\{k\} = q^k-q^{-k}$ and $\{k\}! = \{k\} \{k-1\} \cdots \{1\}$.
\end{proposition}
\begin{proof} We first sketch the proof when $q=\exp(i\pi/r)$ and
 \smash{$q^{\frac2{mn}}=\exp(2i\pi/mnr)$}. Then $H$ is a sub-quotient of the
 topological unrolled quantum group (see~\cite{GHP22}). The
 $R$-matrix factors $R=\mathcal H\check R$ where $\check R$ is the
 quasi R-matrix. Then R is an R-matrix since $\mathcal H$ satisfies
 the following relations: $\forall x,y\in H$,
 $\mathcal H(x\otimes y)\mathcal H^{-1}=xK^{|y|/2}\otimes K^{|x|/2}y$
 where $|x|$, $|y|$ are the integral weights of~$x$ and~$y$;
 $\Delta\otimes\Id(\mathcal H)=\mathcal H_{13}\mathcal H_{23}$; and
 $\Id\otimes\Delta(\mathcal H)=\mathcal H_{13}\mathcal H_{12}$.
 Similarly the fact that $\theta$ is a twist follows since
\[
\mathcal T=\sum_{i=0}^{mn{\pp}-1}
 q^{\frac{-2i^2}{n^2}}\smk_{\; i}=m\circ(S\otimes\Id)(\mathcal H_{21})
\]
 satisfies $S(\mathcal T)=\mathcal T$.

 Finally, for the general case, we remark that the
 \smash{${\mathbb Q}\bigl(q^{\frac2{mn}}\bigr)$} subalgebra generated by~$E$,~$F$ and~$\smk$ is
 also ribbon since it is isomorphic through a Galois isomorphism to a
 sub Hopf algebra of the previous case which contains the R-matrix
 and the twist.
\end{proof}

The cointegral is $\Lambda = c\smk_{\;0}E^{{\pp}-1}F^{{\pp}-1}$ for some scalar $c\in\FK^\times$ and the right integral is $\lambda\bigl(\smk^iE^nF^k\bigr)=\frac{mn{\pp}}{c} \delta_{i,m(1-{\pp})}\delta_{n,{\pp}-1}\delta_{k,{\pp}-1}$. In particular, \smash{$\lambda\bigl(\smk_{\; i}F^{{\pp}-1}E^{{\pp}-1}\bigr)=\frac{1}{c} q^\frac{2i({\pp}-1)}{n}$}.
\begin{proposition}
The category $\cat=H-mod$ is
chromatic compact. It is factorizable if and only if $m=n$ and both
$n$ and ${\pp}$ are odd. It is twist degenerate if and only if $n$ is odd and $\pp$ is a~multiple of~$4$.
\end{proposition}
\begin{proof} As discussed in Section~\ref{ss:coend}, $\Delta_0^P$ is
 given by the action of the central element
 $\Delta_0=(\lambda \otimes \Id)(R_{21}R_{12})$. One can compute
\begin{align*}
\Delta_0&{}=(\lambda \otimes \Id)(R_{21}R_{12})\\
&{}=(\lambda \otimes \Id)\biggl(\sum \limits_{i,j,p,s,k,l}\frac{\{1\}^{2k+2l}}{\{k\}!\{l\}!}q^\frac{k(k-1)+l(l-1)}{2}q^\frac{2ij+2ps}{n^2}\smk_p\smk_{\; i-nl}F^lE^k\otimes \smk_s\smk_{\; j+nl}E^lF^k\biggr).
\end{align*}
Each summand is 0 unless $p=i-nl$ and $s=j+nl$, and after applying $\lambda$ they are also 0 unless $k=l={\pp}-1$. We use that \smash{$\{{\pp}-1\}!=q^{{\pp}({\pp}-1)/2}{\pp}$}.
\begin{align*}
(\lambda \otimes \Id)(R_{21}R_{12})&{}= \sum\limits_{i,j=0}^{mn{\pp}-1}\frac{\{1\}^{4({\pp}-1)}}{c{\pp}^2}q^\frac{2ij+2(i-n({\pp}-1))(j+n({\pp}-1))}{n^2}q^\frac{2(i-n({\pp}-1))({\pp}-1)}{n}\\
&\hphantom{= \sum\limits_{i,j=0}^{mn{\pp}-1}}{}\,
\times \smk_{\; j+n({\pp}-1)}E^{{\pp}-1}F^{{\pp}-1}\\
&{}=\frac{\{1\}^{4({\pp}-1)}}{c{\pp}^2}q^{2({\pp}-1)} \sum \limits_{j=0}^{mn{\pp}-1} q^\frac{-2({\pp}-1)j}{n}\Biggl(\sum \limits_{i=0}^{mn{\pp}-1}q^\frac{4i(j+n({\pp}-1))}{n^2}\Biggr)\\
&{}\hphantom{=\frac{\{1\}^{4({\pp}-1)}}{c{\pp}^2}q^{2({\pp}-1)} \sum \limits_{j=0}^{mn{\pp}-1}}{}\,
\times\smk_{\; j+n({\pp}-1)}E^{{\pp}-1}F^{{\pp}-1}.
\end{align*}
The term in parenthesis is $mn{\pp}$ if $j+n({\pp}-1)$ is a multiple of \smash{$\frac{n^2{\pp}}{\gcd(n^2{\pp},2)}$} and 0 otherwise. Let \smash{$m'=\frac{m\gcd(n^2{\pp},2)}{n}$}. Finally,
\[
(\lambda \otimes \Id)(R_{21}R_{12})= mn \frac{\{1\}^{4({\pp}-1)}}{c{\pp}} \sum \limits_{j=0}^{m'-1} (-1)^\frac{-2jn({\pp}-1)}{\gcd(n^2{\pp},2)} \smk_{\; \frac{jn^2{\pp}}{\gcd(n^2{\pp},2)}}E^{{\pp}-1}F^{{\pp}-1}.
\]

 Then $\gm_H$ given by multiplication on the right by $\frac{c^2{\pp}}{mn\{1\}^{4({\pp}-1)}}\smk_{\;0}$ is a gluing morphism for~$H$,~i.e.:
\[
\gm_H \circ \Delta_0^H(1) = (\lambda \otimes \gm_H)(R_{21}R_{12})= \smk_{\;0}E^{{\pp}-1}F^{{\pp}-1}=\Lambda=\Lambda_H(1).
\]
Write $P_\unit$ as an idempotent
$e_{P_\unit}=i_{P_\unit} \circ \pi_{P_\unit}$ in $H$ such that
$\varepsilon \circ \pi_{P_\unit}$ is the counit. The morphism
$\gm = \pi_{P_\unit} \circ \gm_H \circ i_{P_\unit}$ is a gluing
morphism by naturality of the $\Delta_0^P$'s and the $\Lambda_P$'s.
Hence $\cat$ is always chromatic non-degenerate. Actually,
\[
\varepsilon\circ\gm=\varepsilon\circ\pi_{P_\unit}\circ\left(-\cdot \frac{c^2{\pp}}{mn\{1\}^{4({\pp}-1)}}\smk_{\;0}\right)\circ i_{P_\unit}=\frac{c^2{\pp}}{mn\{1\}^{4({\pp}-1)}}\varepsilon
\]
as the counit is multiplicative and is 1 on $\smk_{\;0}$.
By Lemma~\ref{L:prop-proj}, $\gm$ is invertible,
and by Lemma~\ref{lem:invertibleg}, $\cat$~is chromatic compact.

As discussed in Section~\ref{ss:coend}, $\cat$ is factorizable if and only if $(\lambda \otimes \Id)(R_{21}R_{12})$ is a scalar times~$\Lambda$. This happens if and only if $m'=1$, so if and only if $m=n$ and $n$ and ${\pp}$ are odd.

Let us check for twist non-degeneracy,
\begin{align*}
\Delta_-=\overline{\Delta_+}=\lambda(\theta) &{}= \lambda\Biggl(K^{{\pp}-1}\sum_{k=0}^{{\pp}-1} \frac{\{1\}^{2k}}{\{k\}!}q^\frac{k(k-1)}{2} S(F^k)\Biggl(\sum_{i=0}^{mn{\pp}-1} q^{\frac{-2i^2}{n^2}}\smk_{\; i}\Biggr) E^k\Biggr)\\
&{}= \frac{\{1\}^{2({\pp}-1)}}{c{\pp}}\sum\limits_{i=0}^{mn{\pp}-1}(-1)^{{\pp}-1}q^{\frac{-2i^2}{n^2}} \lambda\bigl(K^{2{\pp}-2}F^{{\pp}-1}\smk_{\; i}E^{{\pp}-1}\bigr)\\
&{}= \frac{\{1\}^{2({\pp}-1)}}{(-1)^{{\pp}-1}c{\pp}}\sum\limits_{i=0}^{mn{\pp}-1}q^{\frac{1}{n^2} (-2i^2+6n({\pp}-1)(i-n({\pp}-1)) )}\\
&{}= \frac{\{1\}^{2({\pp}-1)}}{(-1)^{{\pp}-1}c{\pp}}q^{-6({\pp}-1)^2}\sum\limits_{i=0}^{mn{\pp}-1}q^{\frac{2i}{n^2} (-i+3n({\pp}-1) )}.
\end{align*}
This is a quadratic Gauss sum at a $n^2{\pp}$-th root of unity. They are well-studied, we are computing $G\bigl(1,-3n(\pp-1),n^2\pp\bigr)$ in the notations of~\cite[Appendix B]{BD21}. It is recalled there that if $3n({\pp}-1)$ is even, this vanishes if and only if $n^2{\pp} \equiv 2[4]$ which never happens. If $3n({\pp}-1)$ is odd, this vanishes if and only if $4\vert n^2{\pp}$. Hence $\cat$ is twist degenerate if and only if $n$ is odd and $4\vert {\pp}$.
\end{proof}

The algebraic input in the following example is
a generalization of the one used in~\cite{BD22}, where analogous
computation was performed.
\begin{proposition}
 For $n$ odd and $4\vert {\pp}$, the $(3+1)$-TQFT $\TSkein_\cat$
 distinguishes the closed $4$-manifolds $S^2\times S^2$ and
 \smash{$\C P^2\#\overline{\C P^2}$}, which have same signature, Euler
 characteristic and fundamental groups but different spin status. One
 has: 
 \[\TSkein_\cat\bigl(S^2\times S^2\bigr)= \frac{m^3n \gcd(nr,2)\{1\}^{8({\pp}-1)}}{c^4{\pp}^2}
 \qquad\text{and}\qquad \TSkein_\cat\bigl(\C P^2\#\overline{\C P^2}\bigr)=0\]
\end{proposition}
\begin{proof}
 Both 4-manifolds can be obtained by a single 0 handle, two 2-handles
 and a single 4-handle. For $S^2\times S^2$ the 2-handles form a Hopf
 link, whereas for \smash{$\C P^2\#\overline{\C P^2}$} they are two disjoint
 $\pm1$-framed unknots. The 0-handle gives the skein
 $\zeta \Gamma_0$. Adding a red $\pm1$-framed unknot multiplies by
 $\Delta_\pm$, so by 0 here, and
\[
\TSkein_\cat\bigl(\C P^2\#\overline{\C P^2}\bigr)=0.
\]
 Adding a red Hopf link
 multiplies by~$(\lambda \otimes \lambda)(R_{21}R_{12})$ which is
\smash{$\frac{m^2 \gcd(nr,2)\{1\}^{4({\pp}-1)}}{c^2{\pp}}$}. The 4-handle
 evaluates $\Gamma_0$ to 1. So
 \[
 \TSkein_\cat\bigl(S^2\times S^2\bigr)= \zeta. (\lambda \otimes
 \lambda)(R_{21}R_{12}) = \frac{m^3n\gcd(nr,2)\{1\}^{8({\pp}-1)}}{c^4{\pp}^2}.
\tag*{\qed}
\]
\renewcommand{\qed}{}
\end{proof}

\subsection[Characteristic p]{Characteristic $\boldsymbol{p}$}\label{SS:Ex-chp}
We give an example of a category which is chromatic non-degenerate but
not chromatic compact, and therefore gives a non-compact TQFT. The
example we give is very simple and unlikely to give interesting
4-manifold invariant, but the TQFT already shows some very interesting
features. Its associated algebra on $S^2\times S^1$ is non-semisimple,
so it does not fall under Reutter's theorem~\cite{Reutter} showing
that semi-simple TQFTs cannot detect exotic structures.

The proof of Proposition~\ref{P:symcase} hints at this example. In
characteristic $p$, one may find a~cocommutative Hopf algebra $H$
which is non-semisimple but such that $H^*$ is semi-simple. This gives
a symmetric monoidal, non-semisimple and chromatic non-degenerate
category, therefore with non-semisimple M\"uger center.
\begin{definition}
 Let $\FK$ be an algebraically closed field of characteristic $p$, and
 $H=\FK[\Z/p\Z]$. Denote $\alpha$ the generator of $\Z/p\Z$. Let
 $\cat = H$-mod$^{fd}$ be the symmetric monoidal category of finite-dimensional left $H$-modules.
\end{definition}
\begin{proposition}
 The category $\cat$ is chromatic non-degenerate, but not chromatic
 compact. It
 gives rise to a non-compact TQFT $\TSkein_\cat$.
\end{proposition}
\begin{proof}
 The cointegral is \smash{$\Lambda=\sum_{i=0}^{p-1}\alpha^i$}, and the right
 integral is $\lambda=1^*$ in the basis
 $\bigl(1,\alpha,\dots,\allowbreak \smash{\alpha^{p-1}\bigr)}$. We observe indeed that
 $\varepsilon(\Lambda)=p=0$ whereas $\lambda(1)=1\neq 0$, so $H^*$ is
 semi-simple whereas $H$ is not. One computes the central element
 $\Delta_0= \lambda(1)1=1\in H$, thus $\Delta_0^P=\Id_P$
 for any projective. Therefore, the gluing morphism $\gm$ is given by
 $\Lambda_{P_\unit}$ which is not invertible as $\cat$ is
 non-semisimple.
\end{proof}

Note that $\gm=\Lambda_{P_\unit}$ means that the 1-handle map does not
affect the skein. Similarly, $\cat$ being symmetric and $\lambda(1)=1$
implies that a homotopically-trivial red links can be ignored.

As explained in~\cite{Reutter}, the vector space
\smash{$\TSkein_\cat\bigl(S^2\times S^1\bigr)$} has a natural algebra structure induced~by
the cobordism \smash{$\overset{\dots}{S^3}\times S^1$} where
\smash{$\overset{\dots}{S^3}$} is the thrice-punctured sphere.
 Note that this
algebra is non-unital as the TQFT is non-compact.
\begin{proposition}
 The non-unital algebra $\TSkein_\cat\bigl(S^2\times S^1\bigr)$ is
 non-semisimple $($i.e., it is non-semisimple if one freely adjoins a
 unit$)$.
\end{proposition}
\begin{proof}
For $f\colon P\to P$ an endomorphism of a projective object, denote $O_f$ the skein $\{pt\}\times S^1\subseteq S^2\times S^1$ colored by $P$ with a single coupon $f$. The skein module of $S^2\times S^1$ is generated by the $O_f$'s. As the braiding and twist are trivial, the only relation is cyclicity: \smash{$O_{f\circ g}=O_{g\circ f}$} for $f\colon P\to Q$ and $g\colon Q\to P$. A handle decomposition of \smash{$\overset{\dots}{B^3}\times S^1$} is given by a single 1-handle and a single 2-handle, both of which does not affect the skeins. The algebra structure is given by~${O_f . O_g=O_{f\otimes g}}$.

As $H$ is a projective generator of the category, one can restrict to $P=H$ for the generators of $\TSkein_\cat\bigl(S^2\times S^1\bigr)$. Furthermore endomorphisms of $H$ are right multiplications by elements of $H$, so since $H$ is commutative, the cyclic relations are trivial. So $\TSkein_\cat\bigl(S^2\times S^1\bigr)$ is isomorphic to $\End_\cat(H)\simeq H$ as a vector space, with basis the $O_i:=O_{-.\alpha^i}$'s.
To compute their product, we need to decompose \smash{$H\otimes H=\oplus_{k=0}^{p-1} H. (1\otimes \alpha^k)$}.
Then $O_i. O_j$
is multiplication by $\alpha^i\otimes \alpha^j$ on~${H\otimes H}$. It maps $1\otimes\alpha^k$ to $\alpha^i\otimes\alpha^{k+j}$ which is in the $k+j-i$ summand.We get
\[
O_i. O_j = \sum_{k=0}^{p-1} \delta_{i,j}O_i = p\delta_{i,j}O_i = 0.\]

If one freely adjoins a unit to $\TSkein_\cat\bigl(S^2\times S^1\bigr)$ one gets the non-semisimple $(p+1)$-dimensional algebra $\FK[O_0,O_1,\dots,O_{p-1}]/(O_i.O_j=0)$.
\end{proof}

\subsection*{Acknowledgments}

F.C. is supported by CIMI Labex ANR
11-LABX-0040 at IMT Toulouse within the program ANR-11-IDEX-0002-02
and from the french ANR Project CATORE ANR-18-CE40-0024. N.G.\ is
supported by the Labex CEMPI (ANR-11-LABX-0007-01), Institut de
Math\'ematiques de Toulouse and by the NSF grant DMS-2104497. B.P.\ thanks the France 2030 program Centre Henri Lebesgue ANR-11-LABX-0020-01.
We thank Anna Beliakova, Marco De Renzi, Matthieu Faitg and David Jordan for useful conversations and for suggesting key improvements and Kevin Walker for his encouragement.
The authors also wish to warmly thank the anonymous referees for their questions and comments which very much helped improving the quality of the paper.


\pdfbookmark[1]{References}{ref}
\LastPageEnding

\end{document}